\documentclass[11pt]{amsart}
\usepackage{mathrsfs}
\usepackage{geometry}
\usepackage{graphicx}
\usepackage{amssymb}
\usepackage{epstopdf}
\usepackage{amsmath,amscd}
\usepackage{amsthm}
\usepackage{enumerate}
\usepackage{url,verbatim}
\usepackage{subfig}
\usepackage{enumitem}
\usepackage{tikz}
\usepackage{enumerate}

\usepackage[normalem]{ulem}
\usepackage{fixme}

\calclayout

\makeatletter
\newcommand\xleftrightarrow[2][]{%
  \ext@arrow 9999{\longleftrightarrowfill@}{#1}{#2}}
\newcommand\longleftrightarrowfill@{%
  \arrowfill@\leftarrow\relbar\rightarrow}
\makeatother

\RequirePackage[colorlinks,citecolor=blue,urlcolor=blue]{hyperref}
\usepackage{breakurl}
\theoremstyle{plain}
\newtheorem{theorem}{Theorem}
\newtheorem{definition}[theorem]{Definition}
\newtheorem{lemma}[theorem]{Lemma}
\newtheorem{proposition}[theorem]{Proposition}
\newtheorem{corollary}[theorem]{Corollary}
\newtheorem{example}[theorem]{Example}
\newtheorem{assumption}[theorem]{Assumption}
\newtheorem{remark}[theorem]{Remark}

\newtheorem{claim}[theorem]{Claim}

\newcommand\ol{\overline}

\newcommand \sA{\mathcal{A}}

\newcommand\sQ{{\mathcal Q}}

\newcommand\EE{{\mathbb E}}

\newcommand\RR{{\mathbb R}}
\newcommand\ZZ{{\mathbb Z}}
\newcommand\NN{{\mathbb N}}

\newcommand\PP{{\mathbb P}}
\newcommand\HH{{\mathbb H}}

\renewcommand\ell{l}

\newcommand\CC{\mathbb{C}}

\newcounter{mycount}

\numberwithin{equation}{section}
\numberwithin{theorem}{section}
\numberwithin{figure}{section}

\title{Perfect Matchings and Essential Spanning Forests in Hyperbolic Double Circle Packings}
\author{Zhongyang Li}
\address{Department of Mathematics,
University of Connecticut,
Storrs, Connecticut 06269-3009, USA}
\email{zhongyang.li@uconn.edu}
\urladdr{\url{https://mathzhongyangli.wordpress.com}}

\begin{document}

\maketitle

\begin{abstract}We investigate perfect matchings and essential spanning forests in planar hyperbolic graphs via circle packings. 

We prove the existence of nonconstant harmonic Dirichlet functions that vanish in a closed set of the boundary, generalizing a result in \cite{bsinv}. We then prove the existence of extremal infinite volume measures for uniform spanning forests with partially wired boundary conditions and partially free boundary conditions, generalizing a result in \cite{BLPS01}.

Using the double circle packing for a pair of dual graphs, we relate the inverse of the weighted adjacency matrix to the difference of Green's functions plus an explicit harmonic Dirichlet function. This gives explicit formulas for the probabilities of any cylindrical events.

We prove that the infinite-volume Gibbs measure obtained from approximations by finite domains with exactly two convex white corners converging to two distinct points along the boundary is extremal, yet not invariant with respect to a finite-orbit subgroup of the automorphism group. We then show that under this measure, a.s.~there are no infinite contours in the symmetric difference of two i.i.d.~random perfect matchings.

As an application, we prove that the variance of the height difference of two i.i.d.~uniformly weighted perfect matchings under the boundary condition above on a transitive nonamenable planar graph is always finite; in contrast to the 2D uniformly weighted dimer model on a transitive amenable planar graph as proved in \cite{RK01,KOS06}, where the variance of height difference grows in the order of $\log n$, with $n$ being the graph distance to the boundary. This also implies that a.s.~each point is surrounded by finitely many cycles in the symmetric difference of two i.i.d.~perfect matchings, again in contrast to the 2D Euclidean case.
\end{abstract}

\section {Introduction}

A perfect matching, or a dimer cover of a graph is a subset of edges satisfying the constraint that each vertex is incident to exactly one edge. Perfect matchings are natural models in statistical mechanics, for example, dimer configurations on the hexagonal lattice model the molecule structure of graphite. See \cite{RK10,GV21} for more details.

Perfect matchings on 2-dimensional Euclidean lattice have been studied extensively. Among a number of interesting topics including phase transition (\cite{KOS06}), limit shape (\cite{CKP01,ko07,BL17,ZL20,kr20,ZL18,Li21,ZL201,LV21,ZL22,BD23,BB23}), conformal invariance (\cite{RK00,RK01,JD15,Li13,DJ19,BC21}), etc, one fundamental question is to classify the infinite volume probability measures. It was proved in \cite{SS03} that for any given slope in the Newton polygon (which is the same as the set of all the possible slopes), there is a unique translation-invariant, ergodic Gibbs measure for the uniform dimer model on the 2-dimensional square grid. More recently it was proved in \cite{AA23} that for the uniformly weighted random lozenge tilings of arbitrary domains approximating a closed, simply-connected subset of $\RR^2$ with piecewise smooth, simple boundary in the scaling limit, around any point in the liquid region of its limit shape  the local statistics are given by the infinite-volume, translation-invariant, extremal Gibbs measure of the appropriate slope.

Lately the dimer models on general lattices beyond the 2D Euclidean lattice have attracted substantial interest; see for example \cite{CSW23} which proves a large deviation principle for dimers on $\ZZ^3$; see also \cite{LP21} for a generalization of the honeycomb dimer model to higher dimensions. In this paper we focus on the infinite-volume Gibbs measures for the dimer model on hyperbolic planar graphs. Other statistical mechanical models in the hyperbolic plane, such as the percolation model (see \cite{grgP,bsjams,ZL17,GrZL22,GrZL221,ZL231,ZL232,ZLI23}), have been studied extensively; the main aim of this paper is to pursue this direction for the dimer model; and we aspire to open up a new field in the research of the dimer model in the hyperbolic plane.

For statistical mechanical models on hyperbolic graphs, the structures of the space of infinite-volume Gibbs measures are known to be significantly different from the Euclidean case; for instance, although it is known from \cite{MA80} that any Gibbs measure for the 2D ferromagnetic Ising model is a convex combination of the two extremal ones: the one with ``$+$" boundary condition and the one with ``$-$" boundary condition (see also \cite{GM23} for a similar result for the 2D Potts model); however the Gibbs measure for the Ising model on vertex-transitive tilings of the hyperbolic plane could be more complicated; in particular it is known in \cite{NW90,WC97,WC00} when the hyperbolic graph is self-dual or when the vertex-degree is at least 35, under certain temperature, unlike the Euclidean case, the free-boundary Gibbs measure is not $1/2$ of the sum of the ``$+$"-boundary Gibbs measure and ``$-$" boundary Gibbs measure.

In the celebrated papers \cite{RK00,RK01}, the scaling limit of the uniformly weighted dimer model on the square grid with certain boundary conditions (Temperley boundary conditions) is proved to be conformally invariant by relating the inverse adjacency matrix to the differences of Green's functions. The result was later generalized to weighted dimer models on isoradial double graphs (\cite{Li13}). When studying the hyperbolic graph, given that the hyperbolic plane has nonzero curvature, the first question is to find a suitable embedding. It turns out the double circle packing, which is known to exist for any 3-connected, transient, simple, proper planar graph with bounded vertex degree and locally finite dual (see \cite{hn18}) provides a nice framework, given that when vertices of the pair of dual graphs are located at the circle centers, each pair of dual edges are orthogonal. This critical property allows to obtain a relation between the inverse adjacency matrices and the inverse Laplacian.  For the double circle packing graph, in the parabolic case (when the weighted random walk on the graph is recurrent; see \cite{GS22}) it has been proved in \cite{MS13,BW15,GJN20} that the discrete harmonic function converges to the continuous harmonic function in the scaling limit. In this paper we shall focus on the hyperbolic case.

Essentially we use a Temperley bijection between perfect matchings and essential spanning forests, and study the distributions of perfect matchings by studying the corresponding essential spanning forests. For the case of a hypercubic graph $\ZZ^d$, this has been studied in \cite{BP93}, where the celebrated transfer impedance theorem was proved. Cycle rooted spanning forest on curved surface were studied in \cite{KK13}. It was later proved in \cite{KW14} that the transfer current matrix converges to the differential of Green’s function. The analysis of uniform spanning forest on non-amenable graphs was done in \cite{BLPS01}.

The main technique to study the dimer model in the hyperbolic plane is the circle packing theorem \cite{HS1,HS2},  the Kasteleyn-Temperley-Fisher theory  \cite{Kas61,Kas67,TF61} for the dimer model and the analysis of uniform spanning forest on non-amenable graphs (\cite{BLPS01}).   The main results proved in this paper are summarized below:

\begin{theorem}\label{t11}Let $\HH^2$ be the Poincare unit disk model of the hyperbolic plane. Let $G=(V,E)$ be a 3-connected, transient, simple, proper planar graph with bounded vertex degree and locally finite dual.
\begin{enumerate}
\item  For any nonempty proper closed set $\sA\subset\partial\HH^2$, if the subgraph of $G$ obtained by removing an open neighborhood of $\sA$ is transient, then there exists a non-constant harmonic Dirichlet function whose values are constant restricted to $\sA$; see Theorem \ref{le39} for a precise statement.
\item For any closed subset $\sA\subset\partial\HH^2$, there exists an extremal probability measure on weighted essential spanning forests of $G$, with wired boundary condition on $\sA$ and free boundary condition on $\partial\HH^2\setminus \sA$; see Theorem \ref{t314} for a precise statement.
\item  Consider perfect matchings on the bipartite graph $\ol{G}$ obtained from the superpersion of $G$ and $G^+$.
Under the Temperley boundary conditions, we show that the limits of the entries of the inverse weighted adjacency matrix on one class of black vertices are difference of the Dirichlet Green's functions; while limits of the entries of the inverse weighted adjacency matrix on the other class of black vertices are difference of the Dirichlet Green's functions plus an explicit harmonic Dirichlet function; see Theorem \ref{l422} for a precise statement. We also show that this infinite volume Gibbs measure is extremal under natural technical assumptions; see Propositions \ref{lle614} and \ref{p620} for precise statements. 
\item  Using a sequence of finite graphs with two convex white corners and no concave white corners, such that the two boundaries portions divided by the two convex corners converge to two closed subsets $\sA_0$ and $\sA_1$ of $\partial\HH^2$, the limits of the entries of the inverse of the weighted adjacency matrix converge to the difference of the Dirichlet Green's functions plus explicit harmonic Dirichlet functions; see Theorem \ref{le67} for a precise statement. Moreover, this gives an extremal infinite-volume Gibbs measure; see Proposition \ref{lle83}. 
\item For the infinite-volume Gibbs measure given in Theorem \ref{le67}; or in Theorem \ref{l422}, for uniformly weighted perfect matchings on $\ol{G}$, if $G$ satisfies certain isoperimetric inequalities, then the variance of the height difference of two i.i.d.~perfect matchings is always finite, given that the the height difference on the boundary is 0.  See Theorem \ref{le96} for precise statements. We also prove the absence of doubly infinite self-avoiding paths in the symmetric difference of two i.i.d. dimer configurations in the infinite volume Gibbs measure given in Theorem \ref{le67}, or the measure given in Theorem \ref{l422} under additional technical assumptions which is not necessarily nonamenable. See Theorem \ref{le85} for a precise statement.
\end{enumerate}
\end{theorem}

Theorem \ref{t11}(1) generalizes results in \cite{bsinv}; Theorem \ref{t11}(2) generalizes results in \cite{BLPS01}. 
Theorem \ref{t11}(3) and (4), as far as we know, for the first time give explicit formulas for the probabilities of cylinder events with respect to infinite-volume Gibbs measures of perfect matchings in the hyperbolic plane. 
It is worth noting that although the infinite-volume Gibbs measure in Theorem \ref{t11}(4) is extremal, it is in general not invariant with respect to a finite-orbit subgroup of the automorphism group of $G$, even if we assume that both $G$ and $G^+$ are vertex transitive; see Remark \ref{rk7}. 

Theorem \ref{t11}(5) is in contrast with the 2D Euclidean case, where variance of the difference at two points of height differences of two i.i.d. uniformly weighted perfect matchings grows like $\log n$, where $n$ is the distance between two points.
While in the hypercubic lattice $d\geq 3$, whether the associated flow has Gaussian fluctuation is an open problem (\cite{CSW23}). For the infinite volume Gibbs measure obtained from Temperley boundary conditions, the nonexistence of infinite double dimer contours was also obtained in \cite{gr24}
when $G$ is a reversible random rooted graph which is one ended,
bounded degree, nonamenable triangulation almost surely;
and exponential tail of height functions were proved for nonamenable planar triangulations \cite{gr24} by considering the dimer height function as the winding of the branch of the wired spanning forest.

The organization of the paper is as follows. 
\begin{itemize}
\item In Section \ref{sect:DOD}, we first examine definitions and established properties in graph theory; then introduce the Dirac operator on a bipartite graph formed through the superposition of a pair of dual graphs characterized by perpendicular dual edges;  subsequently, several properties associated with this construction are established. 
\item In Section \ref{sect:esf}, we discuss harmonic Dirichlet functions and essential spanning forest on transient planar graphs $G$ circle packed to the hyperbolic plane $\HH^2$; and prove Theorem \ref{t11}(1) (Lemma \ref{le39}) and (2) (Theorem \ref{t314}).
\item In Section \ref{dtbc}, we introduce Temperley boundary conditions for finite subgraphs derived from the superposition of a pair of dual graphs. We then establish connections between Dirac operators, Laplacian operators, and random walks on finite graphs with Temperley boundary conditions and prove Theorem \ref{t11}(3). 
\item
In Section \ref{sect:gbc}, we discuss the infinite volume Gibbs measure of perfect matchings on $\ol{G}$ obtained by using sequences of finite graphs exhausting the infinite graph with boundary conditions other than the Temperley boundary conditions. We prove Theorem \ref{t11}(4).

\item In Section \ref{sect:ex}, we prove that the limit measure for dimer configurations on $\ol{G}$ obtained as in Theorem 1.1(4) is extreme; we also prove that the limit measure for dimer configurations on $\ol{G}$ obtained as in Theorem 1.1(3) is extreme under certain natural technical assumptions.
We first show that the extremity of the infinite-volume measure is equivalent to the tail-triviality; then we show that the infinite-volume measure is tail-trivial. The proof of the tail-triviality of the infinite volume measure is inspired by the proof of Theorem 8.4 in \cite{BLPS01}. One of the major differences is instead of considering uniform spanning trees, we are considering perfect matchings. The corresponding anti-symmetric function for the perfect matchings is a ``discrete integral" version of the anti-symmetric function for the spanning trees. Although it is straightforward to see that the anti-symmetric function for spanning trees has finite energy, the fact that the anti-symmetric function for perfect matchings has finite energy is not so obvious; but we still manage to prove it by using the transience property of the graph.
\item 
In Section \ref{sect:ddc},  We consider the weak limit of probability measures on finite graphs with the boundary conditions in Theorem \ref{t11}(3)(4). We  establish the a.s. non-existence of infinite double dimer contours; and that each point is almost surely enclosed by a finite number of cycles from the symmetric difference of the two i.i.d. dimer configurations with either prescribed distribution when the graph satisfies certain isoperimetric inequalities. The proof depends on the fact that the decay rate of the Green's functions on such graphs  with respect to the distance of two points due to isoperimetric inequalities.
\end{itemize}

\section{Dirac Operators, Orthodiagonal Graphs and Double Circle Packings}\label{sect:DOD}

In this section, we begin by examining definitions and established properties in graph theory. Following that, we introduce the Dirac operator on a bipartite graph formed through the superposition of a pair of dual graphs. These dual graphs feature perpendicular dual edges. Subsequently, we establish several properties associated with this construction.

\subsection{Graph theory}

We begin this section by reviewing definitions and established properties in graph theory.

\begin{definition}\label{df14}
\begin{enumerate}
\item A graph is called vertex-transitive (resp. quasi-transitive) when there is a unique orbit
(at most finitely many orbits) of vertices under the action of its automorphism group.

\item A graph is called planar if it can be drawn in the plane such that vertices are represented by points and edges are represented by continuous curves, and no two edges can intersect at a point which is not a vertex.

We call a drawing of the graph into the plane $\RR^2$ a proper embedding in $\RR^2$ if any compact subset $K$ in $\RR^2$ intersects at most finitely many edges and vertices.

\item The number of ends of a connected graph is the supreme over its finite
subgraphs of the number of infinite components that remain after removing the subgraph.


\item Let $G=(V,E)$ be a graph. Let $\Gamma$ be a subgroup of the automorphism group of $G$ acting quasi-transitively on $G$. Let $\pi$ be a probability measure on subgraphs of $G$.
\begin{enumerate}
\item We say $\pi$ is $\Gamma$-invariant if for any event $\mathcal{A}$ and $\gamma\in \Gamma$,
\begin{align*}
\pi(\mathcal{A})=\pi(\gamma\mathcal{A}).
\end{align*}
\item We say $\mu$ is $\Gamma$-ergodic if for any event $\mathcal{A}$ satisfying $\mathcal{A}=\gamma \mathcal{A}$ for all $\gamma\in \Gamma$
\begin{align*}
\pi(\mathcal{A})\in\{0,1\}.
\end{align*}
\end{enumerate}
\item A graph $G=(V,E)$ is called amenable if
\begin{align*}
i_{E}(G):=\inf_{K\subset V;0<|K|<\infty}\frac{|\partial_{E}K|}{|K|}=0
\end{align*}
where $\partial_{E}K$ consists of all the edges with exactly one vertex in $K$ and one vertex not in $K$. If $i_{E}(G)>0$, the graph is called non-amenable.
\item Let $G=(V,E)$ be an infinite, connected, locally finite graph. Let $\{X_n\}_{n\geq 0}$ be a simple random walk on $G$  and let $p_n(y,x)=\mathbb{P}(X_n=y|X_0=x)$. Define the spectral radius $\rho$ of $G$ by
\begin{align*}
\rho(G):=\limsup_{n\rightarrow\infty}p_n(x,y)^{\frac{1}{n}}.
\end{align*}
Given that $G$ is connected $\rho(G)$ does not depend on the choices of $x,y$.
\item (Theorem 6.7 in \cite{ly16})Let $G=(V,E)$ be an infinite, connected, locally finite graph, then $G$ is non-amenable if and only if $\rho(G)<1$.
\end{enumerate}
\end{definition}


\subsection{Dirac operators}

In this section, we introduce the Dirac operator on a bipartite graph constructed from the superposition of a pair of dual graphs with perpendicular dual edges and subsequently establish several of its properties.

\begin{definition}\label{df32}
\begin{enumerate}
\item Let $G=(V,E)$ be a graph, and let $\Lambda\subset V$ be a subset of vertices. The subgraph $G_{\Lambda}$ of $G$ induced by $\Lambda$ has vertex set $\Lambda$; and two vertices in $\Lambda$ are joined by an edge in $G_{\Lambda}$ if and only in if they are joined by an  edge in $G$.
\item 
Let $G=(V,E)$ be a graph. We say $G_j$ is a subgraph of $G$ consisting of faces of $G$ if $G_j$ is the union of faces in $G$. 
\item Let $G=(V,E)$ be a one-ended, planar graph, properly embedded in to $\HH^2$ with no infinite faces. Let $G^+$ be the dual graph of $G$. Let $G_j$ be a finite subgraph of $G$. Let $\partial^+{G}_j$ consist of all the edges in $G^+$, whose dual edge joins exactly one vertex in $\Lambda$ and one vertex in $V\setminus \Lambda$. We say $G_j$ is simply-connected if it is connected and $\partial^+ G_{\Lambda}$ is a simple closed curve.
\item 
Let $G=(V,E)$ be a graph. Let $G_j$ is a subgraph of $G$ consisting of faces of $G$. Define $\partial G_j$ to be the set of all the vertices in $G_j$ adjacent to at least one vertices no in $G_j$; as well as edges joining two vertices in $G_j$, each of which is adjacent to at least one vertices no in $G_j$.
\end{enumerate}
\end{definition}

\begin{definition}\label{dfpds}Let $G=(V,E)$ be an infinite, connected graph in the hyperbolic plane in which each face is bounded and finite. Let $G^{+}=(V^+,E^+)$ be the planar dual graph of $G$; i.e. each vertex of $G^+$ corresponds to a face of $G$; two vertices in $G^+$ are joined by an edge in $G^+$ if their corresponding faces in $G$ share an edge.

Let $\overline{G}$ be a bipartite graph obtained from the superposition of $G$ and $G^+$. A black vertex of $\overline{G}$ is either a vertex of $G$ or a vertex of $G^{+}$; a white vertex corresponds to an edge in $G$ (or equivalently, a dual edge in $G^+$). A black vertex $b$ is joined to a white vertex $w$ by an edge in $\overline{G}$ if and only if one of the following two cases occur
\begin{itemize}
\item $b$ is a vertex of $G$, and $w$ corresponds to edge of $G$ incident to $b$; or 
\item $b$ is a vertex of $G^+$, and $w$ corresponds to an edge of $G^+$ incident to $b$.
\end{itemize}
\end{definition}

Throughout this paper, we make the following assumption
\begin{assumption}\label{ap24}Let $G=(V,E)$ be an infinite, connected graph properly embedded into the hyperbolic plane such that each face is bounded and finite,  with edge weights $\{\nu(e)>0\}_{e\in E}$. Let $G^+=(V^+,
E^+)$ be the dual graph with edge weights $\{\nu(f)>0\}_{f\in E^+}$. Assume there exists a proper embedding of $(G,G^+)$ into the hyperbolic plane, such that 
\begin{itemize}
\item each vertex of $G^+$ is placed in the interior of a face of $G$; and
\item each edge of $G$ and $G^+$ is represented by line segments in the Euclidean plane joining adjacent vertices; and
\item each pair of dual edges $(e,e^+)$ are perpendicular. 
\end{itemize}
Let $\ol{G}$ be the bipartite graph obtained from the superposition of $G$ and $G^+$ as in Definition \ref{dfpds}.
\end{assumption}
 Note that each face of $\ol{G}$ is a quadrangle and can be inscribed in a circle. Therefore $\ol{G}$ is a circle pattern graph; see \cite{KLRR22} for more about dimer models on circle pattern graphs.

 Define the $\overline{\partial}$ operator as a matrix with rows and columns indexed by vertices of $\ol{G}$. Then entries of $\ol{\partial}$ are defined as follows:
\begin{itemize}
\item If two vertices $u,v$ are not adjacent in $\ol{G}$, let $\overline{\partial}(u,v)=0$. 
\item If $w$ is a white vertex and $b$ is a vertex of $\ol{G}$ and $w$ and $b$ are adjacent vertices in $\ol{G}$, let $f$ be the edge of $G$ or $G^+$ containing $b$ as an endpoint and $w$ as an interior point; define $\overline{\partial}(w,b)=\overline{\partial}(b,w)$ be the complex number with length $\nu(f)$ and direction the same as the ray in the Euclidean plane pointing from $w$ to $b$.
\end{itemize}

\begin{lemma}\label{le14}Suppose Assumption \ref{ap24} holds. Let $w_1,b_1,w_2,b_2$ be all the vertices along a face of $\ol{G}$ in cyclic order. Then 
\begin{align*}
(-1)\prod_{i=1}^2\frac{\ol{\partial}(w_{i},b_{i})}{\ol{\partial}(b_{i},w_{i+1})}>0
\end{align*}
\end{lemma}

\begin{proof}Note that the polygon formed by $w_1,b_1,w_2,b_2$ in cyclic order is inscribed to a circle $\mathcal{C}$ in the Euclidean plane. Assume $\mathcal{C}$, as a Euclidean circle has center $v_0$. Let $l_1$, $l_2$, $l_{3}$, $l_{4}$ be the  (oriented) perpendicular bisectors to the line segments $w_1b_1$, $b_1w_2$, $w_2b_2$, $b_2w_1$, respectively such that each one of $l_1$, $l_2$, $l_3$, $l_4$ has starting point $v_0$. Then for each $i\in\{1,2\}$;
\begin{align*}
\arg(\ol{\partial}(w_i,b_i))=\arg(l_{2i-1})-\frac{\pi}{2}\\
\arg(\ol{\partial}(b_i,w_{i+1}))=\arg(l_{2i})+\frac{\pi}{2}.
\end{align*}
Note also that 
\begin{align*}
\sum_{i=1}^2\left[\arg(l_{2i})-\arg(l_{2i-1})\right]=\pi.
\end{align*}
Then the lemma follows.
\end{proof}

\subsection{Double circle packing graphs}

In this section, we examine the established facts concerning how a double circle packing graph inherently leads to an orthodiagonal graph.

\begin{definition}\label{df26}
\begin{enumerate}
\item A circle packing is a collection of discs $P=\{C_v\}$ in the plane $\CC$ such that any two distinct discs in $P$ have disjoint interiors.

\item Given a circle packing $P$, the nerve $G(P)$ of $P$ is the graph with vertex set $P$ and two vertices are joined by an edge if and only if their corresponding circles are tangent. The tangency graph $G(P)$ have a planar embedding by placing a vertex at the center of each circle, and each edge is represented by a line segment between centers of tangent circles.

\item A double circle packing of $G$ is a pair of circle
packings $(P,P^+)$, such that
\begin{enumerate}
\item $G$ is the nerve of $P$, and the dual graph $G^+$ is the nerve of $P^+$.

\item Primal and dual circles are perpendicular. More precisely, for each vertex $v$ and face $f$
of G, the discs $P^{+}(f)$ and $P(v)$ have nonempty intersection if and only if
$f$ is incident to $v$, and in this case the boundary circles of $P^{+}(f)$  and $P(v)$ intersect at right angles.
\end{enumerate}
\end{enumerate}
\end{definition}
See Figure \ref{fig:dcp} for an example of double circle packing.

\begin{figure}
\centering
\includegraphics[width=.6\textwidth]{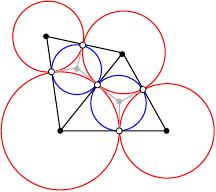}
\caption{Double circile packing: the primal graph is represented by black lines; the dual graph is represented by gray lines; the cirle packing for the primal graph is represented by red circles; the circle packing for the dual graph is represented by blue circles.}
\label{fig:dcp}
\end{figure}

From Definition \ref{df26}, it is straightforward to check the following
\begin{itemize}
\item Let $(P,P^+)$ be a double circle packing of $G$; let $e=(b_3,b_4)$ be an edge of $G$ and $e^+=(b_1,b_2)$ be the dual edge of $G$ in $G^+$. For $i\in\{3,4\}$ (resp.\ $j\in\{1,2\}$), let $C_i$ (resp.\ $C_j$) be the circle in $P$ corresponding to $b_i$ (resp.\ $b_j$). Then the tangent point of $C_3$ and $C_4$ in $P$ is the same as the tangent point of $C_1$ and $C_2$ in $P^+$.
\end{itemize}

\begin{proposition}(\cite{hz99}; see also \cite{hn18})\label{p27} Every infinite, polyhedral (meaning that is it both simple (i.e., not containing any loops or multiple edges) and 3-connected
(i.e. the subgraph induced by $V \ {u, v}$ is connected for each $u, v\in V$))
proper plane graph G with locally finite dual admits a double circle packing in either the Euclidean plane (CP parabolic) or the hyperbolic plane (CP hyperbolic), but not both, and that this packing is unique up to M\"obius
transformations and reflections. In particular, in the hyperbolic case, the packing is unique up to
isometries of the hyperbolic plane; and $G$ is CP hyperbolic if and only if it is transient
for simple random walk.
\end{proposition}

\section{Essential Spanning Forest}\label{sect:esf}

In this section, we discuss harmonic Dirichlet functions and essential spanning forest on transient planar graphs $G$ circle packed to the hyperbolic plane $\HH^2$. Generalizing results in \cite{bsinv}, in Lemma \ref{le39} we show that for any nonempty proper closed set $\sA\subset\partial\HH^2$, if the subgraph of $G$ obtained by removing an open neighborhood of $G$ is transient, then there exists a non-constant harmonic Dirichet function whose values are constant restricted to $\sA$. Generalizing results in \cite{BLPS01}, we show that for any closed subset $\sA\subset\partial\HH^2$, there exists an ergodic probability measure on weighted essential spanning forests of $G$, with wired boundary condition on $\sA$ and free boundary condition on $\partial\HH^2\setminus \sA$; see Theorem \ref{t314}.

\begin{definition}\label{df31}
\begin{enumerate}
\item Let $G=(V,E)$ be a graph with a countable vertex set $V$ and edge set $E$. Let $\vec{E}$ be the directed edge set; i.e.,~each edge in $E$ corresponds to two edges in $\vec{E}$ with opposite orientations.  For each directed edge $e\in \vec{E}$, write $\underline{e}$ (resp. $\ol{e}$) for the tail (resp.\ head) of $\vec{E}$; the edge is oriented from its tail to its head. Write $-e$ for the edge with reversed orientation. Each edge occurs with both orientations. We use $E$ to denote the set of unoriented edges.

\item A network is a pair $(G,C)$, where $G$ is a connected graph with at least two vertices, and $C:E\rightarrow (0,\infty)$ is a function from $E$ to positive reals.

\item Let $(G,C)$ be an infinite network, meaning that the graph $G=(V,E)$ has infinitely many vertices; i.e., $|V|=\infty$.
A loopless subgraph of $G$ that contains every vertex of $G$ and is not necessarily connected
is called a spanning forest, and a spanning forest in which every vertex is connected to infinitely many others is called an essential spanning forest.

\item Let $(G,C)$ be an infinite network. Let $\Omega$ be the set consisting of all the essential spanning forest. A Gibbs measure $\mu$ on $\Omega$ is a probability measure on $\Omega$ satisfying the following condition: let $\mathcal{F}$ be a randomly chosen essential spanning forest of $G$ according to the distribution $\mu$; then for any finite subgraph  $\Lambda=(V_{\Lambda},E_{\Lambda})\subset G$ with undirected edge set $E_{\Lambda}$, a subgraph $S\subset\Lambda$ and a subset of unoriented edges $B\subset E\setminus E_{\Lambda}$
\begin{align*}
\mathbb{P}(\mathcal{F}=S\cup B|B=\mathcal{F}\cap [E_0\setminus E_{\Lambda,0}])\begin{cases}=0&\mathrm{if}\ S\cup B\notin \Omega\\ \propto \prod_{e\in S}C(e)&\mathrm{If}\ S\cup B\in \Omega\end{cases}
\end{align*}

\item An anti-symmetric function $\theta:\vec{E}\rightarrow\RR$ is one satisfying
\begin{align*}
\theta(e)=-\theta(-e).
\end{align*}
Let 
\begin{align*}
R(e):=\frac{1}{C(e)}.
\end{align*}
For anti-symmetric functions $\theta,\theta':\vec{E}\rightarrow \RR$, define an inner product
\begin{align*}
\langle \theta,\theta' \rangle_R:=\frac{1}{2}\sum_{e\in E}R(e)\theta(e)\theta'(e)=\sum_{e\in E_0}R(e)\theta(e)\theta'(e)
\end{align*}

Let $l_{-}^2(E)$ be the Hilbert space of all anti-symmetric functions $\theta:\vec{E}\rightarrow \RR$ with $\mathcal{E}(\theta):=\langle \theta,\theta\rangle_R$ finite.
\item Let $l^2(V)$ be the Hilbert space of functions on $V$ with inner product
\begin{align*}
\langle f,g \rangle_C:=\sum_{v\in V}C_v f(v)g(v).
\end{align*}
where
\begin{align*}
C_v:=\sum_{e=(v,w)\in E}C(e)
\end{align*}

\item Define the gradient operator $\triangledown: l^2(V)\rightarrow l_-^2(E)$ by 
\begin{align*}
\triangledown F(e):=C(e)[F(\ol{e})-F(\underline{e})]
\end{align*}
\item Define the divergence operator $div:l_-^2(E)\rightarrow l^2(V)$ by
\begin{align*}
\mathrm{div} \theta(v):=\frac{1}{C_v}\sum_{\underline{e}=v}\theta(e).
\end{align*}
\item A function $f:V\rightarrow \RR$ is harmonic at a vertex $v$ if 
\begin{align*}
f(v)=\sum_{e=(v,w)\in E}\frac{C(e)f(w)}{C_v}
\end{align*}
or equivalently, if $\mathrm{div} (\triangledown f)(v)=0$
\item A function $f:V\rightarrow \RR$ is Dirichlet if it has finite Dirichlet energy, i.e.
\begin{align*}
\mathcal{E}(\triangledown f):=\sum_{e=(v,w)\in E}C(e)(f(v)-f(w))^2<\infty
\end{align*}
The space of Dirichlet functions on $G$ is denoted by $\mathbf{D}(G)$.
\item Let $\mathbf{L}(G)$ be the set of all real functions $V\rightarrow\RR$ and $\mathbf{L}_0
(G)$ be the set of all
$u\in \mathbf{L}(G)$ with finite support.
Let $\mathbf{D}_0
(G)$ be the closure of $\mathbf{L}_0(G)$ in $\mathbf{D}(G)$ with respect to the norm $\|u\|:=\sqrt{\mathcal{E}(\triangledown u)+|u(b_0)|^2}$, where $b_0$ is a fixed vertex in $V$.
\item We say $G$ is hyperbolic of type 2 if $\mathbf{D}(G)\neq \mathbf{D}_0(G)$.
\item A harmonic Dirichlet function on $G$ is a Dirichlet function which is harmonic at every vertex of $G$. The space of harmonic Dirichlet functions on $G$ is denoted by $\mathbf{HD}(G)$.
\item Given a directed edge $e\in E$, let $\chi^e:=\mathbf{1}_e-\mathbf{1}_{-e}$ be the unit flow along $e$. Let
\begin{align*}
\bigstar:=\triangledown l^2(V);
\end{align*}
i.e. $\bigstar$ is the subspace in $l^2_-(E)$ spanned by the stars $\sum_{\underline{e}=v}C(e)\chi^e=-\triangledown \mathbf{1}_v$.
\item If $e_1,e_2,\ldots,e_n$ is an oriented cycle in $G$, then $\sum_{i=1}^n\chi^{e_i}$ is called a cycle. Let $\diamondsuit\subset l_-^2(E)$ be the subspace spanned by these cycles.
\end{enumerate}
\end{definition}

\begin{definition}\label{df2}
\begin{enumerate}
Let $G=(V,E)$ be a graph.
\item Let $\gamma$ be a path in $G$ and $q:E\rightarrow(0,\infty)$ be a metric on $G$. Define the $q$-length of a path $\gamma$ in $G$ is the sum $q(e)$ over all edges in $\gamma$, i.e.
\begin{align*}
\mathrm{length}_q(\gamma)=\sum_{e\in \gamma\cap E}q(e)
\end{align*}
\item A metric $q:E\rightarrow(0,\infty)$ is an $L^2(E)$ metric if 
\begin{align*}
\|q\|^2:=\sum_{e\in E}C(e)[q(e)]^2<\infty
\end{align*}
\item Let $q$ be a metric on $G$ and let $d_q$ be the associated distance function. Let $C_q(G)$ be the completion of $(V,d_{q})$, and let the $q$-boundary of $G$ be $\partial_q G=C_{q}(G)\setminus V$.
\item Let $\Gamma$ be a collection of infinite paths in $G$. Define the extremal length of $\Gamma$ as follows:
\begin{align*}
\mathrm{EL}(\Gamma):=\sup_{q}\inf_{\gamma\in \Gamma}\frac{\mathrm{length}_q(\gamma)}{\|q\|^2}
\end{align*}
If $\mathrm{EL}(\Gamma)=\infty$, we say $\Gamma$ is null.
\item We say that a property holds for almost every path of a family $\Gamma$ of paths if it does for the member of $\Gamma$ except for those belonging to a subfamily with infinite extremal length.
\end{enumerate}
\end{definition}

\begin{definition}\label{df33}Let $G=(V,E)$ be a graph. For a
set of edges $K\subset E$, let $\mathcal{F}(K)$ denote the $\sigma$-field of events depending only on $K$. Define
the tail $\sigma$-field to be the intersection of $\mathcal{F}(E\setminus K)$ over all finite $K$. We say that a measure on $2^E$ has trivial tail if every event in the tail $\sigma$-field has measure either 0 or 1.

It is known that tail triviality of a measure $\mu$ is equivalent to 
\begin{small}
\begin{align}
\forall A_1\in \mathcal{F}(E);\ \forall \epsilon>0;\ \exists K\ \mathrm{finite},\ \forall A_2\in \mathcal{F}(E\setminus K),\ |\mu(A_1\cap A_2)-\mu(A_1)\mu(A_2)|<\epsilon.\label{dctt}
\end{align}
\end{small}
See, e.g.~page 120 of \cite{geo88}.
\end{definition}

\begin{lemma}\label{lem14}Assume that G is of hyperbolic type of order 2. Then every
$f\in \mathbf{D}(G)$ can be decomposed uniquely in the form: $f = g + h$, where $g\in \mathbf{D}_0(G)$
and $h\in \mathbf{HD}(G)$.
\end{lemma}

\begin{proof}See Theorem 2.1 in \cite{my85}.
\end{proof}

\begin{lemma}\label{le35}Let $G=(V,E)$ be a 3-connected, transient, simple planar graph with bounded vertex degree and locally finite dual that can be circle packed in the hyperbolic plane $\HH^2$. Denote by $P_{a,\infty}(G)$ the set of all paths from $a\in V$ to $\partial\HH^2$ and by $P_{\infty}(G)$ the union of $P_{a,\infty}(G)$ for all $a\in V$. Then
\begin{enumerate}
\item For every $u\in \mathbf{D}(G)$, $u(x)$ has a limit as $x$ tends to $\partial\HH^2$ for almost every $P\in P_{\infty}(G)$. We denote this limit by $u(P)$.
\item Let $u\in \mathbf{D}(G)$. Then $u\in \mathbf{D}_0(G)$ if and only if $u(P)=0$ for almost every infinite path $P$.
\end{enumerate}
\end{lemma}

\begin{proof}See Theorem 3.1 of \cite{KY84} for part (1) of the lemma; see Theorem 3.2 of \cite{my85} for part (2) of the lemma.
\end{proof}


\begin{definition}Let $G=(V,E)$ be a 3-connected, transient, simple planar graph with bounded vertex degree and locally finite dual that can be circle packed in the hyperbolic plane $\HH^2$.  Let $m:E\rightarrow(0,\infty)$ be a metric defined by
\begin{align}
m(e=(u,v))=r_u+r_v;\label{dmm}
\end{align}
where $r_u$ (resp.\ $r_v$) is the radius of the circle centered at $u$ (resp.\ $v$) in the circle packing of which $G$ is the nerve.

Let $\mathcal{A}\subset \partial\HH^2$ be an arbitrary nonempty proper closed subset of $\partial\HH^2$.
Let $\mathcal{S}_{\mathcal{A}}$ be the subspace of $\triangledown\mathbf{HD}(G)$ defined as follows:
\begin{align}
\mathcal{S}_{\mathcal{A}}:&=\triangledown \mathcal{R}_{\sA}\label{dsa}\\
\mathcal{R}_{\sA}:&=\mathbf{HD}(G)\cap\left[\mathrm{closure\ of}\ \{f\in \mathbf{D}(G): \forall z\in \sA,\ \exists\mathrm{open\ set}\ U_z\subset V,\ \label{dra}\right.\\
&\left.\mathrm{s.t.}\ z\in U_z,\ \mathrm{and}\ \mathrm{supp}(f)\cap U_z=\emptyset\}\right].\notag
\end{align}
where the closure is with respect to the norm $\|u\|:=\sqrt{\mathcal{E}(\triangledown u)+|u(b_0)|^2}$ with $b_0$ a fixed vertex in $V$, and open set $U_z$ is with respect to the topology induced by the metric $m$.
\end{definition}

\begin{lemma}\label{le38}Let $G$ be a locally finite connected graph, and let $\Gamma$
be the collection of all infinite paths in $G$. Then G is recurrent if and only if $\Gamma$ is
null.
\end{lemma}

\begin{proof}See \cite{Ya77,Ya79}; see also Theorem 3.2 in \cite{bsinv}.
\end{proof}

\begin{lemma}The metric $m$ defined by (\ref{dmm}) is an $L^2(E)$ metric. 
\end{lemma}
\begin{proof}
Note that
\begin{align}
\|m\|^2=\sum_{e\in E}[m(e)]^2=\sum_{e=(u,v)\in E}[r_u+r_v]^2\leq 2\sum_{e=(u,v)\in E}[r_u^2+r_v^2]<\infty,\label{ml2}
\end{align}
where the last inequality follows from the fact that $G$ has bounded vertex degree, the circles have disjoint interior, and therefore the sum of area of circles in the circle packing is bounded above by the area of the unit disk.
\end{proof}

\begin{lemma}\label{ll39}Let $G=(V,E)$ be a 3-connected, transient, simple planar graph with bounded vertex degree and locally finite dual that can be circle packed in the hyperbolic plane $\HH^2$.
Let $\theta_k:\vec{E}\rightarrow \RR$ be a sequence of $l^2_{-}(E)$ anti-symmetric function such that
\begin{itemize}
\item For each $k$, there exists a neighborhood (with respect to the topology induced by the metric $m$) $U_{k}$ of $\sA$ satisfying
\begin{align*}
\mathrm{supp}(\theta_k)\cap U_{k}=\emptyset.
\end{align*}
\item 
there exists $\theta\in l^2_{-}(E)$,
\begin{align}
\lim_{k\rightarrow\infty}\|\theta_k-\theta\|_R=0,\label{limk}
\end{align}
where the norm $\|\cdot\|_{R}$ is induced by the inner product defined in Definition \ref{df31}(5).
\end{itemize}
Then for almost every singly infinite path $\gamma=(\gamma(0),\gamma(1),\ldots)$ satisfying $\lim_{n\rightarrow\infty}d_m(\gamma(n),z)=0$ for some $z\in\sA$, we have
\begin{align}
\lim_{k\rightarrow\infty}\sum_{i=1}^\infty|\theta_k(e_i)-\theta(e_i)|=0.\label{cl310}
\end{align}
where $e_i\in\vec{E}$ is the directed edge from $\gamma(i-1)$ to $\gamma(i)$.
\end{lemma}

\begin{proof}Let 
\begin{align*}
\Gamma:=\{\gamma=(\gamma(0),\gamma(1),\ldots)\in P_{\infty}:\sum_{i=0}^{\infty}|\theta_k(e_i)-\theta(e_i)|=\infty\}
\end{align*}
Since $\theta_k-\theta\in l^2_{-}(E)$, there exists an $L^2(E)$ metric under which every path in $\Gamma$ has infinite length; and therefore $\Gamma$ is null. Therefore, for almost every singly infinite path $\gamma=(\gamma(0),\gamma(1),\ldots)$ , $\sum_{i=0}^{\infty}|\theta_k(\gamma(i))-\theta(\gamma(i))|<\infty$.

From (\ref{limk}), we infer that 
\begin{align}
\lim_{k\rightarrow\infty}\theta_k(e)= \theta(e);\ \forall e\in \vec{E}.\label{gcv}
\end{align}

For each $a\in V$, define
\begin{align*}
P_{a,\infty,\sA}:=\{\gamma\in P_{b_0,\infty}: \lim_{n\rightarrow\infty}d_m(\gamma(n),z)=0, \mathrm{for\ some}\ z\in\sA\}.
\end{align*}
Since for any singly infinite path in $G$ starting from some vertex $a\in V$, we may add finitely many different steps at the beginning to obtain an singly infinite path in $P_{b_0,\infty}$ without changing the limiting behavior, it suffices to show that for almost every singly infinite path $\gamma\in P_{b_0,\infty,\sA}$, $\lim_{n\rightarrow\infty}\sum_{i=1}^{\infty}|\theta_k(e_i)-\theta(e_i)|=0$.

Let
\begin{align}
\Gamma_0:=\{\gamma\in P_{b_0,\infty,\sA}(G): \sum_{i=0}^{\infty}|\theta_k(e_i)-
\theta(e_i)|<\infty\}.\label{dg0}
\end{align}
then it is straightforward to check that $P_{b,\infty}\setminus \Gamma_0$ is null by the same arguments as above.

Let $\eta\in \Gamma_0$.
Assume there exist $\epsilon_1>0$, such that for any $N,K$, there exist $n>N$ and $k>K$, satisfying 
\begin{align}
\sum_{i=n}^{\infty}|\theta_k(e)|>\epsilon_1\label{nkg}
\end{align}
Then let $n_1,k_1$ be such that (\ref{nkg}) holds with $(n,k)$ replaced by $(n_1,k_1)$ and 
\begin{align}
\sqrt{\sum_{i=n_1}^{\infty}C(e_i)|\theta(e)|^2}<\frac{1}{2};\label{cc1}
\end{align}
and
\begin{align}
\sqrt{\sum_{i=1}^{\infty}C(e_i)|\theta_{k_1}(e_i)-\theta(e_i)|^2}<\frac{1}{2};\label{cc2}
\end{align}
This is possible since $\theta\in l_{-}^2(E)$, and $\theta_k\rightarrow \theta$ with respect to the norm $\|\cdot\|_R$. Note that (\ref{cc1}) and (\ref{cc2}) implies 
\begin{align*}
\sqrt{\sum_{i=n_1}^{\infty}C(e_i)|\theta_{k_1}(e_i)|^2}<1;
\end{align*}

Let $N_2>n_1$ be such that $\mathrm{supp}(\theta_{k_1})\cap \{\eta(n):n\geq N_2\}=\emptyset$, this is possible since $\lim_{n\rightarrow\infty}\eta(n)\in \sA$. Let $n_2>N_2,k_2>k_1$ be  such that (\ref{nkg}) holds with $(n,k)$ replaced by $(n_2,k_2)$, and
\begin{align}
\sqrt{\sum_{i=n_2}^{\infty}C(e_i)|\theta(e_i)|^2}<\frac{1}{2^2};\label{cc3}
\end{align}
and
\begin{align}
\sqrt{\sum_{i=1}^{\infty}C(e_i)|\theta_{k_2}(e_i)-\theta(e_i)|^2}<\frac{1}{2^2};\label{cc4}
\end{align}
Then (\ref{cc1}) and (\ref{cc2}) implies
\begin{align*}
\sqrt{\sum_{i=n_2}^{\infty}C(e_i)|\theta_{k_2}(e_i)|^2}<\frac{1}{2}.
\end{align*}

Continuing this process one can find a sequence $(n_i,k_i)_{i=1}^{\infty}$ satisfying all the following conditions
\begin{itemize}
\item $n_i<n_{i+1}$, $k_i<k_{i+1}$ for each $i\in \NN$; and
\item $\mathrm{supp}(\theta_{k_i})\cap \{\eta(n):n\geq n_{i+1}\}=\emptyset$; and
\item (\ref{nkg}) holds with $(n,k)$ replaced by $(n_i,k_i)$.
\item $\sqrt{\sum_{j=n_2}^{\infty}C(e_j)|\theta_{k_i}(e_j)|^2}<\frac{1}{2^{i-1}}$, for each $i\in \NN$.
\end{itemize}
Then one can find an $L^2(E)$ metric $q$, such that $\mathrm{length}_q(\eta)=\infty$, for any $\eta$ satisfying (\ref{nkg}). Hence the collection of all $\eta$ satisfying (\ref{nkg}) is null. We infer that for almost every $\gamma\in P_{b,\infty,\sA}$, for any $\epsilon>0$, there exists $N$, $K$ (depending only on $\epsilon$ but is independent of each other), such that for all $n>N$ an $k>K$, 
\begin{align}
\sum_{i=n}^{\infty}|\theta_k(e_i)|<\epsilon\label{rle}
\end{align}

We shall prove (\ref{cl310}) for almost every $\gamma\in \Gamma_0$.
Assume there exists $\phi\in \Gamma_0$ such that (\ref{rle}) holds and (\ref{cl310}) does not hold; then there exists $\epsilon_0>0$, such that for any $K$, there exists $k>K$, with 
\begin{align}
\sum_{i=1}^{\infty}|\theta_k(r_i)-\theta(r_i)|\geq \epsilon_0\label{ff1}
\end{align}
where $r_i$ is the directed edge from $\phi(i-1)$ to $\phi(i)$.

Since (\ref{rle}) holds for $\phi$, let $N_1,K_1$ positive integers such that for all $n
\geq N_1$ and $k
\geq K_1$,
\begin{align}
\sum_{i=n}^{\infty}|\theta_k(r_i)|<\frac{\epsilon_0}{4}\label{ff2}
\end{align}

From the definition of $\Gamma_0$ in (\ref{dg0}), we see that there exists $N_0>0$, such that 
\begin{align}
\sum_{i=N_0}^{\infty}|\theta(r_i)|<\frac{\epsilon_0}{4}\label{ff3}
\end{align}
Since $N_2=\max\{N,N_0,N_1\}$ is finite, by (\ref{gcv}) there exists $K_0>0$, such that for any $k\geq K_0$,
\begin{align}
\sum_{i=1}^{N_2}|\theta_k(r_i)-\theta(r_i)|<\frac{\epsilon_0}{4}\label{ff4}
\end{align}

Let $K_2=\max\{K,K_0,K_1\}+1$. We have for any $n>N_2$
\begin{align*}
\sum_{i=n}^{\infty}|\theta_{K_2}(r_i)|&
\geq \sum_{i=1}^{\infty}|\theta_k(r_i)-\theta(r_i)|-\sum_{i=1}^{N_2}|\theta_{K_2}(r_i)-\theta(r_i)|\\
&-\sum_{i=N_2}^{\infty}|\theta(r_i)|-
\sum_{i=N_2}^{\infty}|\theta_{K_2}(r_i)|\\
&\geq \epsilon_0-\frac{\epsilon_0}{4}-\frac{\epsilon_0}{4}-\frac{\epsilon_0}{4}=\frac{\epsilon_0}{4}.
\end{align*}
where the last inequality follows from (\ref{ff1}), (\ref{ff2}), (\ref{ff3}), (\ref{ff4}). But this is a contradiction to the fact that $\lim_{n\rightarrow\infty}\phi(n)\in \sA$, and $\mathrm{supp}(\theta_{K_2})\cap U_{K_2}=\emptyset$, for some open set $U_{K_2}\supset \sA$, which implies $\sum_{i=n}^{\infty}|\theta_{K_2}(r_i)|=0$. This completes the proof.
\end{proof}

\begin{corollary}Let $G=(V,E)$ be a 3-connected, transient, simple planar graph with bounded vertex degree and locally finite dual that can be circle packed in the hyperbolic plane $\HH^2$.
Let $\mathcal{A}\subset \partial\HH^2$ be an arbitrary nonempty proper closed subset of $\partial\HH^2$.
Let $\mathcal{R}_{\mathcal{A}}$ be defined as in (\ref{dra}). Let $f\in  \mathbf{HD}$. If $f\in \mathcal{R}_{\sA}$ then for almost every singly infinite path $\gamma=(\gamma(0),\gamma(1),\ldots)$ satisfying $\lim_{n\rightarrow\infty}d_m(\gamma(n),z)=0$ for some $z\in \sA$, $\lim_{n\rightarrow\infty}f(\gamma(n))=0$.
\end{corollary}

\begin{proof}If $f\in \mathcal{R}_{\sA}$, then there exist a sequence $f_k\in \mathbf{D}(G)$ satisfying for any $k\in \NN$ there exists a neighborhood $U_k$ of $\sA$, with $\mathrm{supp}(f_k)\cap U_k=\emptyset$; and $f_k$ converge to $f$ with respect to the norm $\|u\|:=\sqrt{\mathcal{E}(\triangledown u)+|u(b_0)|^2}$. This implies
\begin{align*}
\lim_{k\rightarrow\infty}f_k(b_0)=f(b_0).
\end{align*}
Then the corollary follows from Lemma \ref{ll39} by letting
\begin{align*}
\theta(e):=R(e)\triangledown f(e);\  \mathrm{and}\ \theta_k(e):=R(e)\triangledown f_k(e)\ \forall e\in \vec{E}\ \mathrm{and}\ k\geq 1.
\end{align*}
\end{proof}

\begin{definition}\label{df311}Given two subsets $A$ and $Z$ of vertices of a network, we call $\theta\in l_{-}(E)$ a flow between $A$ and $Z$ if $\mathrm{div}\theta$ is 0 off of $A$ and $Z$; if it is nonnegative on $A$ and nonpositive on $Z$, then we say that $\theta$ is a flow from $A$ to $Z$.

Let $\theta$ be a flow from $A$ to $Z$. We call
\begin{align*}
\mathrm{Strength}(\theta):=\sum_{a\in A}\mathrm{div}\theta (a)
\end{align*}
the strength of the flow $\theta$. A flow of strength 1 is called a unit flow.

A voltage function is a function on the vertices of the network that is harmonic at all $x\notin A\cup Z$. Given a voltage function $v$, we define the associated current flow function $i$ on the edges by
\begin{align*}
i(x,y)=c(x,y)[i(x)-i(y)].
\end{align*}
A unit current flow is current flow associated to a voltage function which also has strength 1.
\end{definition}

\begin{lemma}(Theorem 2.11 in \cite{ly16})\label{ls83}Let $G$ be a denumerable connected network. Random walk on G is transient if and only if there is a unit flow on $G$ of finite energy from some (every) vertex to $\infty$.
\end{lemma}

\begin{corollary}\label{c313}Let $G=(V,E)$ be a 3-connected, transient, simple proper planar graph with bounded vertex degree and locally finite dual that can be circle packed in the hyperbolic plane $\HH^2$.

Let $\mathcal{A}\subset \partial\HH^2$ be an arbitrary nonempty proper closed subset of $\partial\HH^2$. Then the following two statements are equivalent.
\begin{enumerate}
\item There exists an open set $\mathcal{O}_{\mathcal{A}}$ satisfying $\sA\subset \mathcal{O}_{\mathcal{A}}$, such that for any open set $U_{\mathcal{A}}$ satisfying $\mathcal{A}\subset U_{\mathcal{A}}\subset \mathcal{O}_{\mathcal{A}}$, $G\setminus U_{\mathcal{A}}$ is transient.
\item There exists an open set $U_{\mathcal{A}}$ satisfying $\sA\subset \mathcal{O}_{\mathcal{A}}$, such that, $G\setminus U_{\mathcal{A}}$ is transient.
\end{enumerate}
\end{corollary}

\begin{proof}It is straightforward to see that $(1)\Rightarrow (2)$. The fact that $(2)\Rightarrow(1)$ follows from Lemma \ref{ls83}.
\end{proof}

\begin{theorem}\label{le39}Let $G=(V,E)$ be a 3-connected, transient, simple planar graph with bounded vertex degree and locally finite dual that can be circle packed in the hyperbolic plane $\HH^2$. Assume
\begin{align}
\sup_{e\in E}C(e)=C_0<\infty.\label{bdc}
\end{align}
Let $\mathcal{A}\subset \partial\HH^2$ be an arbitrary nonempty proper closed subset of $\partial\HH^2$ satisfying the following conditions
\begin{enumerate}[label=(\alph*)]
\item There exists an open set $U_{\mathcal{A}}\supseteq \sA$, such that   $G\setminus U_{\mathcal{A}}$ is transient.
\end{enumerate}

Then there exists a nonconstant harmonic Dirichlet function $f:V\rightarrow \RR$ such that $\triangledown f\in \mathcal{S}_{\mathcal{A}}$.
\end{theorem}

\begin{proof}We identify each vertex of $G$ as a point in the unit disk which is the center of the corresponding circle in the circle packing of $G$. By Corollary \ref{c313}, condition (a) implies that there exists an open set $\mathcal{O}_{\mathcal{A}}$ satisfying $\sA\subset \mathcal{O}_{\mathcal{A}}$, such that for any open set $U_{\mathcal{A}}$ satisfying $\mathcal{A}\subset U_{\mathcal{A}}\subset \mathcal{O}_{\mathcal{A}}$, $G\setminus U_{\mathcal{A}}$ is transient.

Define $f:V\rightarrow\RR$ by
\begin{align*}
f_0(v)=d_m(v,\mathcal{A});\ \forall v\in V
\end{align*}
Then we claim that $f_0$ is a nonconstant Dirichlet function in the closure (with respect to the norm $\|u\|:=\sqrt{\mathcal{E}(\triangledown u)}$) of the set of functions 
\begin{align}
\{g\in \mathbf{D}(G):\forall z\in \sA,\ \exists\mathrm{open\ set}\ U_z\subset V,\ 
\mathrm{s.t.}\ u\in U_z,\ \mathrm{and}\ \mathrm{supp}(g)\cap U_z=\emptyset\}.\label{dst}
\end{align}

To see why that is true, note that the function $f_0$ has finite Dirichlet energy because for each $e=(u,v)\in E$, we have
\begin{align*}
|f_0(u)-f_0(v)|=|d_m(u,\mathcal{A})-d_m(v,\mathcal{A})|\leq m(e);
\end{align*}
then
\begin{align*}
&\mathcal{E}(\triangledown f_0)=\sum_{e=(u,v)\in E_0}C(e)(f_0(u)-f_0(v))^2\leq C_0 \sum_{e=(u,v)\in E_0}[m(e)]^2
<\infty;
\end{align*}
where the last inequality follows from (\ref{ml2}). The fact that $f$ is nonconstant follows directly from its definition.

Now we show that $f_0$ is in the closure of the set of functions defined by (\ref{dst}). For each $\epsilon,\delta>0$, define
\begin{align*}
U_{\sA,\epsilon}:&=\{v\in V:d_{m}(v,\sA)<\epsilon\};\qquad U_{\partial \HH^2,\delta}:=\{v\in V:d_{m}(v,\partial\HH^2)<\delta\}.
\end{align*}
and
\begin{align*}
Q_{\sA,\epsilon}:=\{z\in \HH^2,d_{\RR^2}(z,\sA)<\epsilon\}
\end{align*}
where $d_{\RR^2}$ is the distance in the Euclidean plane.
Then one can see that 
\begin{align*}
U_{\sA,\epsilon}\subset Q_{\sA,\epsilon}
\end{align*}
For each $k\geq 1$, define $f_k:V\rightarrow\RR$ by
\begin{align*}
f_k(v):=d_m(v,U_{\sA,\frac{1}{k}})
\end{align*}
Then for each edge $e=(u,v)\in E\setminus U_{\partial \HH^2,\delta}$,
\begin{align*}
C(e)([f_k(u)-f_k(v)]-[f_0(u)-f_0(v)])^2\leq  C_0[|f_k(u)-f_0(u)|+|f_k(v)-f_0(v)|]^2\leq \frac{4C_0}{k^2}
\end{align*}
Choose $\delta_k>0$ such that
\begin{align*}
\lim_{k\rightarrow\infty}\frac{4C_0}{k^2}|E\setminus U^{\circ}_{\partial\HH^2,\delta_k}|=0;\ \mathrm{and}\ \lim_{k\rightarrow\infty}\delta_k\rightarrow 0.
\end{align*}
where
\begin{align*}
 U^{\circ}_{\partial\HH^2,\delta_k}=\{v\in  U_{\partial\HH^2,\delta_k}: v'\in U_{\partial\HH^2,\delta_k};\forall v'\sim v\ \mathrm{in}\ G\}
\end{align*}
Then we have 
\begin{align*}
\|f_k-f\|^2&=\sum_{e=(u,v)\in E}C(e)([f_k(u)-f_k(v)]-[f_0(u)-f_0(v)])^2+(f_k(b_0)-f(b_0))^2\\
&\leq \frac{C_0}{k^2}+\sum_{e=(u,v)\in E\setminus U^{\circ}_{\partial\HH^2,\delta_k}}C(e)([f_k(u)-f_k(v)]-[f_0(u)-f_0(v)])^2\\
&+\sum_{e=(u,v)\in E\cap U^{\circ}_{\partial\HH^2,\delta_k}}C(e)([f_k(u)-f_k(v)]-[f_0(u)-f_0(v)])^2\\
&\leq \frac{C_0}{k^2}+\frac{4C_0}{k^2}|E\setminus U^{\circ}_{\partial\HH^2,\delta_k}|+4C_0\mathrm{area}(Q_{\sA,\delta_k})\\
&\leq \frac{C_0}{k^2}+\frac{4C_0}{k^2}|E\setminus U^{\circ}_{\partial\HH^2,\delta_k}|+8\pi C_0\delta_k
\rightarrow 0
\end{align*}
as $k\rightarrow\infty$. It follows that $f_0$ is in the closure of the set of functions defined by (\ref{dst}).

Since $\mathcal{A}$ is an non-empty, proper, closed subset of $\partial\mathbb{H}^2$, from the definition of $f_0$, we see that $f_0$ is nonconstant on $\partial \HH^2$.

Moreover, since $\sA$ satisfies conditions (a), we claim that and $f_0\notin  \mathbf{D}_0(G)$. 
To see why that is true, let $\sA^c=\partial\HH^2\setminus \sA$. Note that by Lemma \ref{le35}(2), $f_0\in \mathbf{D}_0(G)$ if and only if for almost every singly infinite path $\gamma=(\gamma(0),\gamma(1),\ldots,)$, $\lim_{n\rightarrow\infty}f_0(\gamma(n))=0$. However, from the definition of $f_0$ we see that for any $z\in \sA^c$, and any path $l:=(l(0),l(1),\ldots)$ satisfying $\lim_{n\rightarrow\infty}d_m(l(n),z)=0$, we have $\lim_{n\rightarrow\infty}f_0(l(n))\neq 0$.

Let $\mathcal{O}_{\sA}$ be a neighborhood of $\sA$ given as in condition (a). Then for any neighborhood $U_{\sA}$ satisfying $\sA\subset U_{\sA}\subset\mathcal{O}_{\sA}$, and almost every singly infinite path $\zeta=(\zeta(0),\zeta(1),\ldots,)$ in $G\setminus U_{\sA}$, $\zeta(n)$ must converge to a point $x\in \sA^c$ with respect to the metric $m$; since those do not converge must have infinite lengths in $m$.
By condition (a) and Lemma \ref{le38}, the collection of all the singly infinite paths in $G\setminus U_{\sA}$ is not null. Therefore the statement that for almost every singly infinite path $\gamma=(\gamma(0),\gamma(1),\ldots,)$, $\lim_{n\rightarrow\infty}f_0(\gamma(n))=0$ does not hold. We infer that $f\notin \mathbf{D}_0(G)$.

Therefore $\mathbf{D}(G)\neq \mathbf{D}_0(G)$, and $G$ is hyperbolic type of order 2. By Lemma \ref{lem14}, $f_0$ has the following decomposition
\begin{align*}
f_0=g_0+f
\end{align*}
where $g_0\in \mathbf{D}_0(G)$, and $f\in \mathbf{HD}(G)$. Since $f_0\notin \mathbf{D}_0(G)$, $f\neq 0$.
From the identity $f=f_0-g_0$, we see that for almost every singly infinite path $P$ converging to a point in $\sA$, $f(P)=0$, while it is not true that for almost every singly finite path $P$ converging to a point in $\sA^c$, $f(P)=0$. Therefore $f$ is nonconstant. $f$ is in the closure of the set defined by (\ref{dst}) because both $f_0$ and $g_0$ are in the closure of the set defined by (\ref{dst}).
Then the conclusion of the lemma follows.
\end{proof}

\begin{assumption}\label{ap311}Let $G=(V,E)$ be a 3-connected, transient, simple planar graph with bounded vertex degree and locally finite dual that can be circle packed in the hyperbolic plane $\HH^2$. Suppose (\ref{bdc}) holds. Let $\mathcal{A}\subset \partial\HH^2 $ be an arbitrary closed subset of $\partial\HH^2$. Let
\begin{align*}
 V_1\subseteq V_2\subseteq\ldots \subseteq V_j\subseteq\ldots
\end{align*}
be an increasing sequence of sets of vertices such that all the following conditions hold
\begin{enumerate}
\item 
$\cup_{j=1}^{\infty} V_j=V$;
and
\item let $G_j$ be the subgraph of $G$ induced by $V_j$; i.e.~the vertex set of $G_j$ is $V_j$, two vertices in $V_j$ are joined by an edge in $G_j$ if and only if they are joined by an edge in $G$. Then for each point $z\in \mathcal{A}$, there exists a singly infinite path $\gamma=(\gamma(0),\gamma(1),\ldots,\gamma(n),\ldots)$ in $G_j$ such that
\begin{align*}
\lim_{n\rightarrow\infty}d_m(\gamma(n),z)=0.
\end{align*}
\end{enumerate}
where $m$ is the metric defined as in (\ref{dmm}).
Let
\begin{align*}
V_j\supseteq U_{j,1}\supseteq U_{j,2}\supseteq\ldots\supseteq U_{j,s}\supseteq \ldots
\end{align*}
be a decreasing subset of vertices such that
\begin{itemize}
\item $\cap_{s=1}^{\infty}\ol{U}_{j,s}=\mathcal{A}$; where $\ol{U}_{j,s}$ is the closure of $U_{j,s}$ under the metric $m$; and
\item for each finite $s$, $V_j\setminus U_{j,s}$ is a finite set of vertices.
\end{itemize}

Let $G_{j,s}^{\bullet}$ be the graph obtained from $G_j$ by contracting all the vertices in $U_{j,s}$ and edges joining two vertices in $U_{j,s}$ into one vertex $z_{j,s}$.
\end{assumption}

\begin{assumption}\label{ap312}Let $G=(V,E)$ be a 3-connected, transient, simple proper planar graph with bounded vertex degree and locally finite dual that can be circle packed in the hyperbolic plane $\HH^2$. Suppose (\ref{bdc}) holds. Let $\mathcal{A}\subset \partial\HH^2 $ be an arbitrary closed subset of $\partial\HH^2$. Let
\begin{align*}
V\supseteq V_1\supseteq V_2\supseteq\ldots V_j\supseteq \ldots
\end{align*}
be a decreasing subset of vertices such that
\begin{itemize}
\item $\cap_{j=1}^{\infty}\ol{V}_j=\sA$; where $\ol{V}_j$ is the closure of $V_j$ under the metric $m$. 
\end{itemize}
Let $G_j^{\bullet}$ be graph obtained from $G$  by contracting all the vertices in $V_j$ and edges joining two vertices in $V_j$ into one vertex denoted by $\mathbf{a}_j$.
For each $j\in \NN$, let
\begin{align*}
 U_{j,1}\subseteq U_{j,2}\subseteq\ldots \subseteq U_{j,s}\subseteq\ldots
\end{align*}
be an increasing sequence of finite sets of vertices such that all the following conditions hold
\begin{enumerate}
\item $\mathbf{a}_j\in U_{j,s}$; and
\item 
$\cup_{j=1}^{\infty} U_{j,s}=[V\setminus V_j]\cup\{\mathbf{a}_j\}$;
\end{enumerate}

Let $G_{j,s}^{\bullet}$ be the subgraph of $G_j^{\bullet}$ induced by $U_{j,s}$; i.e.~the vertex set of $G_{j,s}$ is $U_{j,s}$, two vertices in $U_{j,s}$ are joined by an edge in $G_{j,s}$ if and only if they are joined by an edge in $G$. 
\end{assumption}

\begin{theorem}\label{t314}Suppose either Assumption \ref{ap311} or Assumption \ref{ap312} holds. 
Let $\mu_{j,s}$ be the Gibbs measure for spanning trees on $G_{j,s}$, then $\lim_{j\rightarrow\infty}\lim_{s\rightarrow\infty}\mu_{j,s}$ converges weakly to a  measure $\mu_{\mathcal{A}}$ on essential spanning forest on $G$. Moreover,
\begin{enumerate}
\item $\mu_{\mathcal{A}}$ has trivial tail.
\item If $\mathcal{A}=\emptyset$, $\mu_{\mathcal{A}}=FSF$, where FSF represents the limit measure of free spanning trees on finite graphs;
\item If $\mathcal{A}=\partial\HH^2$, $\mu_{\mathcal{A}}=WSF$, where WSF represents the limit measure of wired spanning trees on finite graphs, i.e., identifying all the vertices on the boundary as one vertex;
\item Let $\mathfrak{F}$ be the random essential spanning forest on $G$, then for any $e\in E$,
\begin{align*}
WSF(e\in \mathfrak{F})\leq \mu_{\mathcal{A}}(e\in \mathfrak{F})\leq FSF(e\in \mathfrak{F}).
\end{align*}
\item If $G$ is quasi-transitive, $\mu_{\mathcal{A}}$ is ergodic but not necessarily automorhpism-invariant.
\item If $\sA$ satisfies condition (a) in Lemma \ref{le39}, then $\mu_{\sA}\neq WSF$.
\end{enumerate}
\end{theorem}

\begin{proof}We only prove the conclusion under Assumption \ref{ap311}; the conclusion under Assumption \ref{ap312} can be proved similarly.

Note that $\triangledown\mathbf{HD}(G)$ is a subspace of $l_-^2(E)$. It is straightforward to check that
\begin{align*}
\triangledown \mathbf{HD}(G)\perp \bigstar;\qquad \triangledown \mathbf{HD}(G)\perp\diamondsuit; \qquad \diamondsuit\perp\bigstar;
\end{align*}
where $\perp$ is with respect to the inner product $\langle \cdot,\cdot \rangle_R $ in $l_-^2(E)$. Indeed, $l^2_{-}(E)$ has a decomposition as follows:
\begin{align*}
l_-^2(E)=\bigstar\oplus\diamondsuit\oplus \triangledown\mathbf{HD}(G)
\end{align*}

For each $e\in E$, define
\begin{align}
I_{e,\mathcal{A}}=P_{\bigstar\oplus \mathcal{S}_{\mathcal{A}}}\chi_{e};\label{diea}
\end{align}
i.e. $I_e$ is the orthogonal projection of $\chi_e$ to the subspace $\bigstar\oplus \mathcal{S}_{\mathcal{A}}$ of $l_{-}^2(E)$. Let $E_{\frac{1}{2}}\subset \vec{E}$ be a subset of $\vec{E}$ such that each edge in $E$ appears exactly once in $E_{\frac{1}{2}}$ with a fixed orientation. Then define a matrix $Y$ with rows and columns indexed by $E_{\frac{1}{2}}$ by
\begin{align*}
Y_{\mathcal{A}}(e,f)=I_{e,\mathcal{A}}(f).
\end{align*}
Then the weak convergence of $\lim_{j\rightarrow\infty}\lim_{s\rightarrow\infty}\mu_{j,s}$ follows from Lemmas \ref{lem16} and \ref{lem17}.

Let $\langle\chi_{F} \rangle $ be the linear space spanned by $\{\chi_f:f\in F\}$.
The fact that $\mu_{\mathcal{A}}$ has trivial tail follows from Lemma \ref{l112} by letting $s\rightarrow\infty$ and then letting $j\rightarrow\infty$, and the fact that
\begin{align*}
\|P_{\langle\chi_{F} \rangle}I_{e,j,s}\|^2\rightarrow 0
\end{align*}
given that $d_{G}(F,K)\rightarrow 0$ and $e\in K$. Then (1) follows.

Since tail triviality implies mixing, when $G$ is quasi-transitive and infinite, $\mathrm{Aut}(G)$ has an infinite orbit, we infer that in this case $\mu_{\mathcal{A}}$ is ergodic. If $\mathcal{A}$ is not automorphism invariant, then $\mu_{\mathcal{A}}$ is not automorphism invariant. Then (5) follows.

(2)-(3) are straightforward to verify from the deterministic expression of $\mu_{\mathcal{A}}$. To see why (4) is true, note that by Lemma \ref{lem16} we have
\begin{align*}
\mu_{\mathcal{A}}(e\in \mathfrak{F})=I_{e,\mathcal{A}}(e)=C(e)\langle P_{\bigstar\oplus \mathcal{S}_{\mathcal{A}}}\chi_{e},\chi_e\rangle_{R}=C(e)\|P_{\bigstar\oplus\mathcal{S}_{\mathcal{A}}}\chi_{e}\|_R^2.
\end{align*}
Moreover, it is known from \cite{BLPS01} that
\begin{align*}
&FSF(e\in \mathfrak{F})=C(e)\|P_{\bigstar}\chi_{e}\|_R^2;\qquad WSF(e\in \mathfrak{F})=C(e)\|P^{\perp}_{\diamondsuit}\chi_{e}\|_R^2;
\end{align*}
Then (4) follows from the fact that 
\begin{align*}
\bigstar\subset [\bigstar\oplus \mathcal{S}_{\mathcal{A}}]\subset [\bigstar\oplus \triangledown \mathbf{HD}(G)]= [\diamondsuit^{\perp}].
\end{align*}

When $\sA$ satisfies condition (a) in Lemma \ref{le39}, then by Lemma \ref{le39} we see that $\mathcal{S}_{\sA}$ is non-trivial linear space. Therefore $\bigstar$ is a proper subset of $\bigstar\oplus \mathcal{S}_{\sA}$. Then (6) follows.
\end{proof}

Let $G$ be an infinite graph. We define two
possibly different currents,
\begin{align*}
I_{e,F}
:= P^{\perp}_{\diamondsuit}\chi_e,
\end{align*} 
the free current between the endpoints of $e$ (also called the “limit current”), and
\begin{align*}
I_{e,W}
:= P_{\bigstar}\chi_e ,
\end{align*}
the wired current between the endpoints of $e$ (also called the “minimal current”). Recall the following two propositions

\begin{proposition}\label{pp16}(Free Currents) Let G be an infinite graph exhausted by 
subgraphs $\{G_n=(V_n,E_n)\}_{n=1}^{\infty}$. Let e be an edge in $G_1$ and let
\begin{align*}
\bigstar_{G_n}:=\triangledown l^2(V_n)
\end{align*}
$I_n := P_{\bigstar{G_n}}\chi_e$
. Then $\|I_n-I_{e,F}\|_{R}
\rightarrow 0$ as $n\rightarrow\infty$
and $\mathcal{E}(I_{e,F}
)=I_{e,F}
(e)$.
\end{proposition}

\begin{proof}See Corollary 3.17 of \cite{Sor94}; see also Proposition 7.1 of \cite{BLPS01}; in the latter further assumption that each subgraph $G_n$ is finite was made; while in Corollary 3.17 of \cite{Sor94} $G_n$ is not required to be finite.
\end{proof}

\begin{proposition}\label{pp17}
Let $G$ be an infinite graph exhausted by
subnetworks $\{G_n\}_{n=1}^{\infty}$. Let $G_{n,W}=(V_{n,W},E_{n,W})$
be formed by identifying the complement of $G_n$ in $G$ to a single
vertex. Let $e$ be an edge in $G_1$ and
\begin{align*}
\bigstar_{G_{n,W}}:=\triangledown l^2(V_{n,W})
\end{align*}
Define
$I_n := P_{\bigstar_{G_{n,W}}}\chi_e$. Then $\|I_n-I_{e,W}
 \|_{R}\rightarrow 0$ as $n\rightarrow\infty$
and $\mathcal{E}(I_{e,W}
) = I_{e,W}
(e)$, which is the minimal energy among all $\theta\in l^2_{-}(E)$ satisfying
$\mathrm{div}\theta = \mathrm{div} \chi_e$.
\end{proposition}

\begin{proof}See Theorem 3.25 of \cite{Sor94}; see also Proposition 7.2 of \cite{BLPS01} with the further assumption that each $G_n$ is finite.
\end{proof}

\begin{lemma}\label{lem16}Suppose (\ref{bdc}) holds. Let $\mathcal{T}$ be the weighted spanning tree on $G_{j,s}$, and let $e\in E_{\frac{1}{2}}$. Then
\begin{align*}
\lim_{j\rightarrow\infty}\lim_{s\rightarrow\infty}\mu_{j,s}(e\in \mathcal{T}):=I_{e,\mathcal{A}}(e);
\end{align*}
where $I_{e,\mathcal{A}}$ is defined by (\ref{diea}).
\end{lemma}

\begin{proof}Consider the weighted spanning tree model on the finite graph $G_{j,s}=(V_{j,s},E_{j,s})$. By the transfer current theorem on finite graphs (see \cite{BP93}; see also Theorem 4.1 of \cite{BLPS01}), we obtain
\begin{align*}
\mu_{j,s}(e\in \mathcal{T})
=I_{e,j,s}(e)
\end{align*}
where we let $\diamondsuit_{j,s}$ be the subspace of $l^2_{-}(E_{j,s})$ generated by $\sum_{i=1}^n\chi^{e_i}$  for all the oriented cycles $e_1,e_2,\ldots,e_n$ in $G_{j,s}$, and assume
\begin{align*}
I_{e,j,s}:=P_{\bigstar_{j,s}}\chi_e.
\end{align*}
and $P_{\bigstar_{j,s}}\chi_e$ is the orthogonal projection of $\chi_e\in l^2_{-}(E_{j,s})$ to the subspace $\bigstar_{j,s}:=\triangledown l^2(V_{j,s})$.
Note that for each fixed $e,j$ by Proposition \ref{pp17} we obtain
\begin{align*}
\lim_{s\rightarrow\infty}\|I_{e,j,s}-I_{e,j,W_j}\|_R=0;
\end{align*}
where
\begin{align*}
I_{e,j,W_j}=P_{\bigstar_j}\chi_e;
\end{align*}
and $\bigstar_j:=\triangledown l^2(V_j)$.

Moreover, by Proposition \ref{pp16} we obtain
\begin{align*}
\lim_{j\rightarrow\infty}\|I_{e,j,W_j}-I_{e,\mathcal{A}}\|_R=0;
\end{align*}
then the lemma follows.

\end{proof}

Let $F$ be a set of edges. The contracted graph $G/F$ is defined by by identifying
every pair of vertices that are joined by edges in $F$. We identify the set of edges in $G$ and in $G/F$. Let $T_G$ (resp.\ $T_{G/F}$) denote a subtree of $G$ (resp.\ $G/F$). Recall that $\langle\chi_{F} \rangle $ be the linear space spanned by $\{\chi_f:f\in F\}$. Let 
\begin{align*}
Z:=P_{\bigstar\oplus\mathcal{S}_{\mathcal{A}}}\langle\chi_{F} \rangle. 
\end{align*}
Let
\begin{align*}
\bigstar\oplus \mathcal{S}_{\mathcal{A}}=\widehat{\bigstar}\oplus Z;
\end{align*}
Let $e$ be an edge that does not form a cycle together with edges in $F$. Let
$\widehat{I}_{e}:=P_{\widehat{\bigstar}}\chi_e$. Then
\begin{align}
\widehat{I}_{e}:=P_{\widehat{\bigstar}}\chi_e=P_{Z}^{\perp}P_{\bigstar\oplus \mathcal{S}_{\mathcal{A}}}\chi_e=
P_{\langle \chi_{F}\rangle}^{\perp}P_{\bigstar\oplus \mathcal{S}_{\mathcal{A}}}\chi_e
=P_{\langle \chi_F\rangle}^{\perp} I_{e,\mathcal{A}}\label{php}
\end{align}
where the identity $P_{Z}^{\perp}P_{\bigstar\oplus \mathcal{S}_{\mathcal{A}}}\chi_e=
P_{\langle \chi_{F}\rangle}^{\perp}P_{\bigstar\oplus \mathcal{S}_{\mathcal{A}}}\chi_e$ follows from the fact that $P_{\bigstar\oplus \mathcal{S}_{\mathcal{A}}}\chi_e\in \bigstar\oplus \mathcal{S}_{\mathcal{A}}$.

Recall the following lemma:
\begin{lemma}\label{le19}(Proposition 4.2 in \cite{BLPS01})Let $G$ be a finite connected graph. Assuming that there is no cycle of $G$ in $F$, the distribution of $T_{G}$ conditioned on $F\subset T_G$ is equal to the distribution of $T_{G/F}\cup F$ when we think of $T_G$ and $T_{G/F}$ as sets of edges.
\end{lemma}

\begin{lemma}\label{lem17}Suppose (\ref{bdc}) holds. Let $\mathcal{T}$ be the weighted random spanning tree on $G_{j,s}$, and let $e_1,\ldots,e_k\in E_{\frac{1}{2}}$. Then
\begin{align*}
\lim_{j\rightarrow\infty}\lim_{s\rightarrow\infty}\mu_{j,s}(e_1,\ldots,e_k\in \mathcal{T}):=\det[Y_{\mathcal{A}}(e_i,e_j)]_{1\leq i,j\leq k}
\end{align*}
\end{lemma}

\begin{proof}The proof is an adaptation of the proof of (4.4) in \cite{BLPS01}.

If some cycle can be formed from the edges $e_1,\ldots,e_k$, then the linear combination of corresponding columns of $[Y_{\mathcal{A}}(e_i,e_j)]$ is zero. Therefore, we assume that there are no such cycles for the remainder of the proof.

Note that
\begin{align*}
Y_{\mathcal{A}}(e,f)=I_{e,\mathcal{A}}(f)&=C(e)\langle I_{e,\mathcal{A}},\chi_f \rangle
=C(e)\langle P_{\bigstar\oplus \mathcal{S}_{\mathcal{A}}}\chi_e,\chi_f\rangle_R\\
&=C(e)\langle P_{\bigstar\oplus \mathcal{S}_{\mathcal{A}}}\chi_e,P_{\bigstar\oplus \mathcal{S}_{\mathcal{A}}}\chi_f\rangle_R=C(e)\langle I_{e,\mathcal{A}},I_{f,\mathcal{A}} \rangle_{R}
\end{align*}
Therefore
\begin{align*}
\det [Y_{\mathcal{A}}(e_i,e_j)]_{1\leq i,j\leq k}=\left[\prod_{1\leq i\leq k} C(e_i)\right]\det Y_k;
\end{align*}
where $Y_k$ is the Gram matrix with entries $\langle I_{e_i,\mathcal{A}}, I_{e_j,\mathcal{A}}\rangle_R$. The determinant of a Gram matrix is the squared volume of the parallelepiped spanned by its determining vectors; therefore
\begin{align*}
\det [Y_{\mathcal{A}}(e_i,e_j)]_{1\leq i,j\leq k}=\prod_{i=1}^k \left[C(e_i)
\|P_{Z_i}^{\perp}I_{e_i,\mathcal{A}}\|^2_R\right]
\end{align*}
where $Z_i$ is the linear space spanned by $\{I_{e_1,\mathcal{A}},\ldots, I_{e_{i-1},\mathcal{A}}\}$.

By Lemmas \ref{lem16}, \ref{le19}, we have
\begin{align*}
\lim_{j\rightarrow\infty}\lim_{s\rightarrow\infty}\mu_{j,s}(e_i\in \mathcal{T}|e_1,\ldots,e_{i-1}\in \mathcal{T})=\widehat{I}_{e_i}(e_i)
\end{align*}
with $F=\{e_1,\ldots,e_{i-1}\}$. Moreover
\begin{align*}
\widehat{I}_{e_i}(e_i)=C(e_i)\langle \widehat{I}_{e_i},\chi_{e_i} \rangle_{R}
=C(e_i)\langle P_{\widehat{\bigstar}}\chi_{e},P_{\widehat{\bigstar}}\chi_{e}
\rangle_R=C(e_i)\|P_{Z_i}^{\perp}I_{e,\mathcal{A}}\|^2.
\end{align*}
where the last identity follows from (\ref{php}).
Therefore we have
\begin{align*}
\lim_{j\rightarrow\infty}\lim_{s\rightarrow\infty}\mu_{j,s}(e_1,\ldots,e_k\in \mathcal{T})&=\prod_{i=1}^k\lim_{j\rightarrow\infty}\lim_{s\rightarrow\infty} \mu_{j,s}(e_i\in \mathcal{T}|e_1,\ldots,e_{i-1}\in \mathcal{T})\\
&=\prod_{i=1}^k\|P_{Z_i}^{\perp}\left[C(e_i)I_{e_i,\mathcal{A}}\|^2\right].
\end{align*}
Then  the lemma follows.
\end{proof}

\begin{lemma}\label{l112}Let $T$ be the weighted random spanning tree on the finite graph $G_{j,s}$. Let $F$ and $K$ be disjoint nonempty sets of edges. If $A_1\in \mathcal{F}(K)$ and $A_2\in \mathcal{F}(F)$, then
\begin{align*}
|\mu_{j,s}[A_1\cap A_2]-\mu_{j,s}[A_1]\mu_{j,s}[A_2]|\leq \left(2^{2|K|}|K|\sum_{e\in K}C(e)\|P_{\langle \chi_{F}\rangle}I_{e,j,s}\|_R^2\right)^{\frac{1}{2}}
\end{align*}
\end{lemma}

\begin{proof}See Theorem 8.4 of \cite{BLPS01}.
\end{proof}

\section{Dimers and Temperley Boundary Conditions}\label{dtbc}

In this section, we introduce Temperley boundary conditions for finite subgraphs derived from the superposition of a pair of dual graphs. We then establish connections between Dirac operators, Laplacian operators, and random walks on finite graphs with Temperley boundary conditions. We prove explicit expressions of the limits of entries of the inverse Dirac operator under Temperley boundary conditions, which give a unique infinite-volume Gibbs measure for perfect matchings on the bipartite graph obtained from the superposition of a transient planar graph and its dual graph; see Lemma \ref{l422}. 

\subsection{Perfect matchings on finite subgraphs: probability measure and local statistics}

\begin{definition}Let $G=(V,E)$ be a graph. A perfect matching, or dimer configuration $M$ on $G$ is a subset of edges such that each vertex in $V$ is incident to exactly one edge in $M$.
\end{definition}

We define a probability measure for perfect matchings on  finite subgraphs of $\ol{G}$ as follows:

\begin{definition}Suppose Assumption \ref{ap24} holds. Let $\ol{G}_0$ be a finite subgraph of $\ol{G}$.
Let $\mathcal{M}$
be the collection of all the perfect matchings on $\ol{G}_0$. Define the partition function for perfect matchings on $\ol{G}_0$ by
\begin{align*}
Z(\ol{G}_0):=\sum_{M\in \mathcal{M}}\prod_{e=(w,b)\in M}|\ol{\partial}(w,b)|
\end{align*}
For each $M\in \mathcal{M}$, let the probability of $M$ be
\begin{align*}
\mathbb{P}(M):=\frac{1}{Z(\ol{G}_0)}\prod_{e=(w,b)\in M}|\ol{\partial}(w,b)|
\end{align*}
\end{definition}

From Lemma \ref{le14}, we obtain the following corollary:

\begin{corollary}\label{c18}Suppose Assumption \ref{ap24} holds. Let $\ol{G}_0$ be the finite subgraph of $\ol{G}$ admitting a perfect matching. Then
\begin{enumerate}
\item
\begin{align*}
Z(\ol{G}_0)=\left|\det\left[\ol{\partial}|_{W_{\ol{G}_0},B_{\ol{G}_0}}\right]\right|
\end{align*}
where $\ol{\partial}|_{W_{\ol{G}_0},B_{\ol{G}_0}}$ is the submatrix of $\ol{\partial}$ with rows restricted to white vertices of $\ol{G}_0$ and columns restricted to black vertices of $\ol{G}_0$.
\item Let $t\geq 1$ be a positive integer, and $e_1,\ldots,e_{t}$ be $t$ distinct edges of $\mathcal{G}_{j}$ with no common vertices. Let $e_1,\ldots,e_{t}\in M$ be the event that all the edges $e_1,\ldots,e_t$ appear in a random perfect matching. Then 
\begin{align*}
\mathbb{P}(e_1,\ldots,e_t\in M)=\prod_{i=1}^{t}\nu(e_i)\left|\det [\ol{\partial}|_{W_{\ol{G}_0},B_{\ol{G}_0}}]^{-1}_{\{w_1,\ldots,w_t\},\{b_1,\ldots,b_t\}}\right|
\end{align*}
where $e_i=(w_i,b_i)$ has white vertex $w_i$ and black vertex $b_i$ for all $1\leq i\leq t$.
\end{enumerate}
\end{corollary}

\begin{proof}Part (1) of the corollary is inspired by the Kasteley-Temperley-Fisher algorithm to count the number of perfect matchings in a finite planar graph; see \cite{Kas61,TF61,Kas67}. Assume
\begin{align*}
|W_{\ol{G}_0}|=|B_{\ol{G}_0}|=k
\end{align*}
and write
\begin{align*}
A:=\ol{\partial}|_{W_{\ol{G}_0},B_{\ol{G}_0}}.
\end{align*}
Let $S_k$ be the symmetric group of $k$ elements. Then we have 
\begin{align}
\det\left[A\right]:=\sum_{\sigma\in S_k}(-1)^{I(\sigma)}\prod_{i=1}^k A_{i,\sigma(i)};\label{dt}
\end{align}
where $I(\sigma)$ is the inversion number of the permutation $\sigma$. Each nonzero term on the right hand side of (\ref{dt}) corresponds to a dimer configuration; Lemma \ref{le14} implies that each nonzero term in the right hand side of (\ref{dt}) has the same argument and the modulus equal to the product of weights of edges present in the dimer configuration.

Part (2) of the corollary is inspired by the work of Kenyon to compute local statistics of lattice dimers; see \cite{RK97}.
\end{proof}

Let $\overline{D}$ be obtained from $\overline{\partial}$ by multiplying the weight of each edge $e$ by $\frac{1}{\sqrt{\nu(e)\nu(e^{+})}}$, where $e^{+}$ is the dual edge of $e$. More precisely we have
\begin{align*}
\overline{D}=S\overline{\partial}S^*
\end{align*}
where for each black vertex $b$ we have
\begin{align*}
S(b,b)=1;
\end{align*}
and for each white vertex $w$, let $e$, $e^+$ be the pair of edges in $\overline{G}$ and $\overline{G}^+$ in which $w$ lies, we have
\begin{align*}
S(w,w)=\frac{1}{\sqrt{\nu(e)\nu(e^{+})}};
\end{align*}
and $S^*$ is the transpose conjugate of $S$.

\subsection{Temperley Boundary Conditions}\label{sect42}

Suppose Assumption \ref{ap24} holds.
Assume that $G_{j}$ is a finite simply-connected subgraph of $G$ drawn on $\HH^2$ consisting of faces of $G$ and is induced by a finite subset of vertices.

Let $G_{j}^+$ be the interior dual graph of $G_{j}$, i.e., each vertex of $G_{j}^{+}$ corresponds to a finite face of $G_{j}$; and two vertices of $G_{j}^+$ are joined by an edge $e^+$ of $G_{j}^+$ if and only if the corresponding faces in $G$ share an edge $e$. Let $\ol{G}_j$ be the bipartite graph obtained from the superposition of $G_{j}$ and $G_{j}^+$ as in Definition \ref{dfpds}.

\begin{figure}
\centering
\includegraphics[width=.35\textwidth]{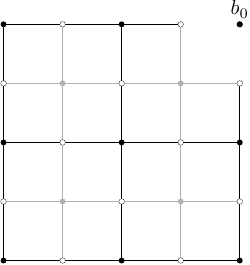}
\caption{Temperley Boundary Condition: the primal graph is represented by black lines; the dual graph is represented by gray lines.}
\label{fig:tpl}
\end{figure}

\begin{lemma}\label{le33}Suppose Assumption \ref{ap24} holds. Let $\mathcal{G}_{j}$ be the graph obtained from $\ol{G}_{j}$ by removing a a fixed vertex $b_0$ on the boundary. Then $\mathcal{G}_{j}$ admits a perfect matching.
\end{lemma}

\begin{proof}A directed spanning tree of a graph is a connected acyclic collection of edges of the graph, where each edge has a chosen direction such that each vertex but one has exactly one outgoing edge. The single vertex with no outgoing edges is called the root of the tree.

Recall that $b_0$ be the unique black vertex of $\ol{G}_{j}$ that is not in $\mathcal{G}_{j}$. Consider a directed spanning tree $T_{b_0}$ of $G_{j}$ rooted at $b_0$. Let $R_{b_0}^{+}$ be the  collection of dual edges of $\mathcal{G}_{j}$, such that a dual edge is present in $R_{b_0}^{+}$ if and only if its corresponding primal edge is absent in $T_{b_0}$.

We claim that each connected component of $R_{b_0}^{+}$ contains exactly one dual vertex not in the interior dual graph  $\ol{G}_{j}^{+}$. Indeed, if a component of $R_{b_0}^{+}$ contains at least two dual vertices not in $\ol{G}_{j}^{+}$, then $T_{b_0}$ will not be connected. If a component of $R_{b_0}^{+}$ contains no dual vertices not in $\ol{G}_{j}^{+}$, then $T_{b_0}$ will contain at least one cycle.
Also because $T_{b_0}$ is connected, $R_{b_0}^{+}$ is acyclic; hence $R_{b_0}^{+}$ is a forest; i.e., each component of $R_{b_0}^{+}$ is a tree. Moreover, $R_{b_0}^{+}$ contains all the dual vertices in $\ol{G}_{j}^{+}$  because $T_{b_0}$ is acyclic.

We may let the unique dual vertex in one component of $R_{b_0}^{+}$ not in $\ol{G}_{j}^{+}$ be the root of the tree,  and orient the edges of $R_{b_0}^+$ such that every dual vertex in $R_{b_0}^+\cap\ol{G}_{j}^{+}$ has exactly one outgoing edge; and each dual vertex in $R_{b_0}^+\setminus\ol{G}_{j}^{+}$ has no outgoing edges.

We can obtain a dimer configuration $M$ on $\mathcal{G}_{j}$ from $T_{b_0}$ as follows: an edge $e=(u,v)\in M$ if and only if one of the following two cases occurs
\begin{enumerate}
\item there exists an edge $f=(u,w)\in T_{b_0}$, such that $e\subset f$, and $f$  is the unique outgoing edge at $u$ in $T_{b_0}$; or 
\item there exists an edge $g=(z,v)\in R_{b_0}^{+}$ such that $e\subset g$, and $g$ is the unique outgoing edge at $v$ in $R_{b_0}^+$.
\end{enumerate}
\end{proof}

\begin{remark}The finite graph $\mathcal{G}_j$ constructed as above is said to exhibit a Temperley boundary condition; see Figure \ref{fig:tpl} for an example.
\end{remark}

\subsection{Discrete Laplacian Operator and Discrete Harmonic Functions}

\begin{definition}\label{df17}Suppose Assumption \ref{ap24} holds. Define
\begin{align}
\Delta:=\overline{D}^*\overline{D}\label{dfd}
\end{align}
Let
\begin{align*}
\ol{D}_{j}:=\ol{D}|_{W(\mathcal{G}_j),B(\mathcal{G}_j)}
\end{align*}
i.e., $\ol{D}_{j}$ is the submatrix of $\ol{D}$ with rows restricted to white vertices of $\mathcal{G}_j$ and columns restricted to black vertices of $\mathcal{G}_j$. Note also that the white vertices of $\mathcal{G}_j$ are identified with edges of $G_j$ or edges of $G_j^+$, and black vertices of $\mathcal{G}_j$ are identified with $[V(G_j)\cup V(G_j^+)]\setminus\{b_0\}$, where $b_0$ be the vertex removed from $\ol{G}_j$ to obtain $\mathcal{G}_j$.
 Define the discrete Laplacian operator $\Delta_{G_{j}^+}$ on $G_{j}^+$ by
\begin{align}
\Delta_{G_{j}^+}:=\left.\ol{D}_j^*\ol{D}_j\right|_{V(G_j^+),V(G_j^+)}=\left.
\Delta\right|_{V(G_{j}^+),V(G_{j}^+)}.\label{ddhd}
\end{align}
 
Define the discrete Laplacian operator $\Delta_{G_{j}}$ on $G_{j}$ by
\begin{align}
&\Delta_{G_{h}}:=\left.\ol{D}_j^*\ol{D}_j\right|_{V(G_j)\setminus\{b_0\},V(G_j)\setminus\{b_0\}};
\label{ddhp}
\end{align}

Here $V(G_{j})$ (resp.\ $V(G_{j}^+)$) is the set consisting of all the vertices of $G_{j}$ (resp.\ $G_{j}^+$).
\end{definition}

The next two lemma gives the asymptotics of the entries of $\Delta_{G_{j}}$ and $\Delta_{G_{j}^+}$ as $\epsilon \rightarrow 0$. We first make the following assumption:

\begin{lemma}\label{le11}Suppose Assumption \ref{ap24} holds.
\begin{enumerate}
\item If $b_1,b_2$ are adjacent vertices in $G$, we have
\begin{align*}
\Delta(b_1,b_2)=\Delta_{G}(b_1,b_2)=-\frac{\nu(e)}{\nu(e^+)};
\end{align*}
where $e=\langle b_1,b_2 \rangle$ is the edge in $G$ with endpoints $b_1,b_2$.
\item If $b_1,b_2$ are adjacent vertices in $G^+$, we have
\begin{align*}
\Delta(b_1,b_2)=\Delta_{G^{+}}(b_1,b_2)=-\frac{\nu(e)}{\nu(e^+)};
\end{align*}
where $e=\langle b_1,b_2 \rangle$ is the edge in $G^+$ with endpoints $b_1,b_2$.
\item If $d_{\overline{G}}(b_1,b_2)>2$, then
\begin{align*}
\Delta(b_1,b_2)=0.
\end{align*}
where $d_{\overline{G}}(\cdot,\cdot)$ is the graph distance between two vertices.
\item If $d_{\overline{G}}(b_1,b_2)=2$ and $b_1\in V$, $b_2\in V^+$, then
\begin{align*}
\Delta(b_1,b_2)=0.
\end{align*}
\item For each vertex $b$ of $G$
\begin{align*}
\Delta(b,b)=\Delta_{G}(b,b)=\sum_{e\in E(G):e\sim b}\frac{\nu(e)}{\nu(e^+)}
\end{align*}
where the sum is over all the incident edges of $b$ in $G$.
\item For each vertex $b$ of $G^+$
\begin{align*}
\Delta(b,b)=\Delta_{G^+}(b,b)=\sum_{e\in E(G^+):e\sim b}\frac{\nu(e)}{\nu(e^+)}
\end{align*}
where the sum is over all the incident edges of $b$ in $G^+$.
\end{enumerate}
\end{lemma}

\begin{proof}
(1), (2) and (3) (5) (6) are straightforward to be verified; it remains to prove (4).

Assume $d_{\overline{G}}(b_1,b_2)=2$ and $b_1,b_2$ are not adjacent in $G$. In this case exactly one of $b_1$ and $b_2$ is a vertex of $G$ and the other is a vertex of $G^+$. Without loss of generality, assume $b_1$ is a vertex of $G$ and $b_2$ is a vertex of $G^+$. Then we can find two white vertices $w_1$ and $w_2$, such that each one of $w_1$ and $w_2$ is adjacent to both $b_1$ and $b_2$ in $\overline{G}$.
Assume $(b_1,w_1)$ is part of an edge $(b_1,b_1')$ in $G$, $(b_2,w_1)$ is part of an edge $(b_2,b_2')$ in $G^+$, $(b_1,w_2)$ is part of an edge $(b_1,b_1'')$ in $G$, $(b_2,w_2)$ is part of an edge $(b_2,b_2'')$ in $G^+$. Let $\alpha_1$ (resp.\ $\alpha_2$) be the unit vector pointing from $b_1$ to $b_1'$ (resp.\ $b_2$ to $b_2'$); let $\alpha_1'$ (resp.\ $\alpha_2'$) be the unit vector pointing from $b_1$ to $b_1''$ (resp.\ $b_2$ to $b_2''$).

Then we have
\begin{align*}
\Delta(b_1,b_2)&=[\overline{D}^*(b_1,w_1)][\overline{D}(w_1,b_2)]+
[\overline{D}^*(b_1,w_2)][\overline{D}(w_2,b_2)]\\
&=\sqrt{\frac{\nu(b_1,b_1')}{\nu(b_2,b_2')}}\cdot 
\sqrt{\frac{\nu(b_2,b_2')}{\nu(b_1,b_1')}}\overline{\alpha}_1\alpha_2+
\sqrt{\frac{\nu(b_1,b_1'')}{\nu(b_2,b_2'')}}\cdot 
\sqrt{\frac{\nu(b_2,b_2'')}{\nu(b_1,b_1'')}}\overline{\alpha}_1'\alpha_2'\\
&=\overline{\alpha}_1\alpha_2+\overline{\alpha}_1'\alpha_2'
\end{align*}
Note also that either
\begin{align*}
\alpha_2=\mathbf{i}\alpha_1;\qquad \alpha_2'=-\mathbf{i}\alpha_1'
\end{align*}
or 
\begin{align*}
\alpha_2=-\mathbf{i}\alpha_1;\qquad \alpha_2'=\mathbf{i}\alpha_1'.
\end{align*}
where $\mathbf{i}^2=-1$ is the imaginary unit. Then (4) follows.
\end{proof}

\begin{lemma}Suppose Assumption \ref{ap24} holds. Let $\Delta_{G_{j
}}$, $\Delta_{G_{j}^+}$ be defined as in (\ref{ddhd}) and (\ref{ddhp}). Then $\Delta_{G_{j}}$ is the Neumann Laplacian operator on $G_{j
}$ and $\Delta_{G_{j}^+}$ is the Dirichlet Laplacian operator on $G_{j}^+$.
\end{lemma}

\begin{proof}By Lemma \ref{le11}, it suffices to verify the boundary conditions. Let $f$ be a function defined on the vertices of $G_{j}$. Let $b$ be a boundary vertex of $G_{j}$. Explicit computations show that
\begin{align*}
[\Delta_{G_{j}}f](b)=\sum_{b'\in V(G_{j}):b'\sim b}\frac{\nu(bb')}{\nu([bb']^+)}[f(b)-f(b')]
\end{align*}
Hence in order to make $[\Delta_{G_{j}}f]=0$, it suffices that $f$ is harmonic at every vertex of $G_{j}$ and $f(u)=f(b)$ for all vertices $u\sim b$ and $u\in V(G
)\setminus V(G_{j})$. It follows that $\Delta_{G_{j}}$ is the Neumann Laplacian operator on $G_{j}$.

Let $g$ be a function defined on the vertices of $G_{j}^+$. Let $s$ be a boundary vertex of $G_{j}^+$, which corresponds to a boundary face of $G_{j}$. Explicit computations show that
\begin{align*}
[\Delta_{G_{j}^+}f](b)=\sum_{s'\in V(G_{j}^+):s'\sim s}\frac{\nu(ss')}{\nu([ss']^+)}[f(b)-f(b')]+\sum_{s''\in V(G^+)\setminus V(G_{j}^+):s''\sim s}
\frac{\nu(ss'')}{\nu([ss'']^+)}f(b)
\end{align*}
Hence in order to make $[\Delta_{G_{j}^+}f]=0$, it suffices that $f$ is harmonic at every vertex of $G_{j}^+$ and $f(u)=0$ for all vertices $u\sim s$ and $u\in V(G^+)\setminus V(G_{j}^+)$. It follows that $\Delta_{G_{j}^+}$ is the Dirichlet Laplacian operator on $G_{j
}^+$.
\end{proof}

\begin{definition}Let $G_{\Lambda}\subseteq G$ be a simply-connected subgraph of $G$ induced by a set of vertices $\Lambda\subseteq V$. Let $\Lambda_{o}$ consist of all the vertices in $\Lambda$ in which all of whose adjacent vertices in $G$ are in $\Lambda$ as well. A discrete harmonic function on $G_{\Lambda}$ is a function $f:\Lambda\rightarrow \CC$ satisfying 
\begin{align}
\Delta f(v)=0,\qquad\forall v\in \Lambda_o.\label{dhs}
\end{align}

We can similarly define the discrete harmonic function on $G^+$.
\end{definition}

\begin{lemma}\label{l310}(Maximal modulus principle for discrete harmonic functions)Let $G_{\Lambda}\subseteq G$ be a finite simply-connected subgraph of $G$ induced by a set of vertices $\Lambda\subseteq V$. Let $f:\Lambda\rightarrow \CC$ be a discrete harmonic function on $G_{\Lambda}$. Then
\begin{align*}
\max_{v\in \Lambda}|f(v)|=\max_{u\in [\Lambda\setminus \Lambda_o]}|f(u)|
\end{align*}
\end{lemma}

\begin{proof}Note that for each $v\in \Lambda_o$, by (\ref{dhs}) we have
\begin{align*}
f(v)=\sum_{u\in V: u\sim v\ \mathrm{in}\ G}-\frac{\Delta(v,u)}{\Delta(v,v)}f(u)
\end{align*}
with
\begin{align*}
-\frac{\Delta(v,u)}{\Delta(v,v)}>0,\ \mathrm{and}\ 
\sum_{u\in V: u\sim v\ \mathrm{in}\ G}-\frac{\Delta(v,u)}{\Delta(v,v)}=1.
\end{align*}
Then we obtain for each $v\in \Lambda_o$,
\begin{align*}
|f(v)|\leq\max_{u\in V: u\sim v\ \mathrm{in}\ G}|f(u)|.
\end{align*}
Then the lemma follows.
\end{proof}

\begin{lemma}\label{le311}There exists at most one $\ol{D}^{-1}:E\times [V\cup V^+]\rightarrow \CC$ satisfying both of the following conditions
\begin{enumerate}
\item $\ol{D}[\ol{D}]^{-1}=I$; and
\item there exists a function $f:\NN\rightarrow \RR^+$, such that
\begin{align*}
\lim_{m\rightarrow\infty}f(m)=0;
\end{align*}
and
\begin{align*}
\left|\ol{D}^{-1}(b,w)\right|\leq f(d_{\ol{G}}(b,w))
\end{align*}
\end{enumerate}
\end{lemma}

\begin{proof}Assume there exist two functions $K_1:E\times [V\cup V^+]\rightarrow\CC$ and $K_2:E\times [V\cup V^+]\rightarrow\CC$ satisfying both (1) and (2). Then 
\begin{align*}
\Delta(K_1-K_2)=\ol{D}^*[\ol{D}(K_1-K_2)]=0;
\end{align*}
which implies that $K_1-K_2$ is discrete harmonic. Then $K_1-K_2=0$ follows from (2) and Lemma \ref{l310}.
\end{proof}

\subsection{Discrete Green's Function on an Infinite Transient Graph}

\begin{definition}We say a random walk on a graph is recurrent if it visits its starting position infinitely often with probability one and transient if it visits its starting point finitely often with strictly positive probability.

Let $x$ be the starting point of the random walk $\{X_n\}_{n=0}^{\infty}$. Let $N_x$ be the total time of visits of the random walk at $x$ and $\theta_x$ be the first return time to $x$, i.e.
\begin{align*}
\theta_x:=\inf\{n\geq 1: X_n=x,\mathrm{if}\ X_0=x\}
\end{align*}
By writing
\begin{align*}
&\mathbb{P}(N_x=1)=\mathbb{P}(\theta_x=\infty);\\
&\mathbb{P}(N_x=n)=\mathbb{P}(\theta_x<\infty)^{n-1}\mathbb{P}(\theta_x=\infty).
\end{align*}
We see that a random walk starting at $x$ is transient if and only if $\mathbb{E}N_x<\infty$, in this case the random walk a.s.~visits $x$ finitely many times; see also \cite{mb}.
\end{definition}

\begin{definition}\label{df12}For $u,v\in V(G^+)$, define the discrete Green's function by
\begin{align*}
G(u,v)=\sum_{n=0}^{\infty}\mathbb{P}^{u}(X_n=v)=
\mathbb{E}N_u
\end{align*}
where $X_n$ is the random walk on $G^+$ starting at $u$ with transition probability
\begin{align}
\mathbb{P}^u(X_{n+1}=v|X_n=y)=\begin{cases}0&\mathrm{If}\ y\ \mathrm{and}\ v\ \mathrm{are\ not\ adjacent\ in}\ G^+\\ \left|\frac{
\Delta(y,v)
}{\Delta(y,y)} \right|&\mathrm{otherwise}\end{cases}\label{drwd}
\end{align}
Similarly we can define a random walk and discrete Green's function on $G$.
\end{definition}

One can verify the following lemma by straightforward computations.
\begin{lemma}\label{l14}Let the discrete Green's function on $G^+$ be defined as in Definition $\ref{df12}$; write
\begin{align}
F^v(u):=\frac{G(u,v)}{|\Delta(v,v)|}.\label{dfuv}
\end{align}
then for each fixed u
\begin{align*}
[\Delta F^v](u)=\delta_v(u)
\end{align*}
where 
\begin{align*}
\delta_v(u)=\begin{cases}1&\mathrm{If}\ u=v\\0&\mathrm{Otherwise}.\end{cases}
\end{align*}
\end{lemma}

\subsection{Dirichlet Green's Function on a Finite Graph}
\begin{definition}(See also Definition 1.6 in \cite{wp} and Definition 1.1 in \cite{nbl})\label{df214}Let the random walk on $G^+=(V^+,E^+)$ be defined as in (\ref{drwd}). For $u\in V(G_{j}^+)$, let $\tau^{u}_j$ be the first time when a random walk on $G^+$ visits a vertex in $V^+\setminus V(G_{j}^+)$

For $u,v\in V(G_{j}^+)$, define the discrete Dirichlet Green's function on $G_{j}^+$ to be
\begin{align*}
G_{j,d}(u,v)=\sum_{n=0}^{\infty}\mathbb{P}^u(X_n=v;\tau^{u}_j>n)
\end{align*}
\end{definition}

\begin{assumption}\label{ap16}Suppose the random walk on $G^+$ defined by (\ref{drwd}) satisfies
\begin{align*}
\sum_{n=0}^\infty \mathbb{P}^u(X_n=v)<\infty,\ \forall u,v\in V(G^+)
\end{align*}
\end{assumption}

\begin{lemma}\label{l17}Suppose Assumption \ref{ap16} holds. Let $\{G_{n}^+\}_{n=1}^{\infty}$ be a sequence of simply-connected finite subgraphs of $G^+$ induced by subset of vertices such that
\begin{align*}
G_{1}^+\subseteq G_{2}^+\subseteq \ldots\subseteq G_{n}^+\subseteq\ldots
\end{align*}
and
\begin{align}
\cup_{n=1}^{\infty}G_n^+=G^+\label{cw}
\end{align}
Then for any $u,v\in V(G_{\epsilon}^+)$
\begin{align*}
\lim_{n\rightarrow\infty}G_{n,d}(u,v)=G(u,v)
\end{align*}
where $G(u,v)$ is defined as in (\ref{df12}).
\end{lemma}

\begin{proof}Note that
\begin{align*}
G(u,v)-G_{n,d}(u,v)=\sum_{j=0}^{\infty}\mathbb{P}^u(X_j=v;\tau^{u}_{n}\leq j)
\end{align*}
Let $\mathcal{N}_{v,[t,\infty)]}$ be the number of visits to $v$ of a random walk starting from $u$ in $[t,\infty)$ steps. Then
\begin{align}
G(u,v)-G_{n,d}(u,v)=\mathbb{E}\mathcal{N}_{v,[\tau_{n}^u,\infty)}\label{pc1}
\end{align}

By (\ref{cw}) we have $\tau_{n}^u\leq \tau_{n+1}^u$, and therefore
$\mathcal{N}_{v,[\tau_{n^u,\infty)}}\geq 
\mathcal{N}_{v,[\tau_{n+1^u,\infty)}}
$. Moreover, (\ref{cw}) also implies
\begin{align}
\lim_{n\rightarrow\infty}\tau_{n}^u=\infty.\label{tif}
\end{align}
Then
\begin{align}
\lim_{n\rightarrow\infty}\mathbb{E}\mathcal{N}_{v,[\tau_{n}^u,\infty)}=0.\label{pc2}
\end{align}
follows from Assumption \ref{ap16} and (\ref{tif}). Then the lemma follows from (\ref{pc1}) and (\ref{pc2}).
\end{proof}

\begin{lemma}\label{le217}Let the Dirichlet Green's function on $G_j^+$ be defined as in Definition $\ref{df214}$; write
\begin{align}
F_{j,d}^v(u):=\frac{G_{j,d}(u,v)}{\Delta(v,v)}.\label{dfuvd}
\end{align}
then for each fixed $u\in V(G_j^+)$
\begin{align*}
[\Delta_{G_j^+} F_{j,d}^v](u)=\delta_v(u).
\end{align*}
\end{lemma}

\begin{proof}Note that
\begin{align*}
[\Delta_{G_j^+} F_{j,d}^v](u)&=\frac{1}{\Delta(v,v)}\left[\Delta(u,u)F^{v}(u)+\sum_{y\in V(G_h^+):y\sim u)}\Delta(u,y)F^v(y)\right]\\
&=\frac{1}{\Delta(v,v)}\sum_{n=0}^{\infty}\left[\Delta(u,u)\mathbb{P}^u(X_n=v,\tau_j^u>n)+\sum_{y\in V(G_h^+):y\sim u}\Delta(u,y)\mathbb{P}^y(X_n=v,\tau_j^u>n)\right]\\
&=\sum_{n=0}^{\infty}\frac{\Delta(u,u)}{\Delta(v,v)}[\mathbb{P}^u(X_n=v,\tau_h^u>n)-\mathbb{P}^u(X_{n+1}=v,\tau_j^u>n+1)]\\
&=\begin{cases}0&\mathrm{If}\ u\neq v\\
1&\mathrm{If}\ u=v.
\end{cases}
\end{align*}
\end{proof}

\subsection{Neumann Green's Function on a Finite Graph}
\begin{definition}\label{df218} Let $b_0\in V(G_{j})$ be a fixed vertex on the boundary of $G_{j}$.
For $u,v,x,y\in V(G_{j})$ and $u\neq b_0$, let $W_{n,j}$ is the random walk on $T_{j}$ starting from $u$ defined by
\begin{align*}
\mathbb{P}_h^u(W_{n+1,j}=v|W_{n,j}=y)=\begin{cases}0&\mathrm{If}\ y\ \mathrm{and}\ v\ \mathrm{are\ not\ adjacent\ in}\ T_{h}\\ \left|\frac{
\Delta(y,v)
}{\sum_{x\in V(G_{j
}):x\sim y}\Delta(y,x)} \right|&\mathrm{otherwise}\end{cases}
\end{align*}
Let $\eta$ be the first hitting time of $W_{n,j}$ to $b_0$, i.e,
\begin{align*}
\eta=\inf\{m\geq 0:W_{m,j}=b_0\}.
\end{align*}
Define the discrete Neumann Green's function on $G_{j}$ to be
\begin{align*}
G_{j,N}(u,v)=\sum_{n=0}^{\infty}\mathbb{P}^u_{j}(W_{n,j}=v;n< \eta).
\end{align*}
\end{definition}

It is straightforward to see from the definition that
\begin{align*}
G_{j,N}(u,b_0)=0.
\end{align*}

\begin{lemma}\label{le219}Let the Neumann Green's function on $G_h$ be defined as in Definition $\ref{df218}$; write
\begin{align}
F_{j,N}^v(u):=\frac{G_{j,N}(u,v)}{\sum_{y\in V(G_j):y\sim v}|\Delta(v,y)|}.\label{dfuvn}
\end{align}
then for each fixed $u\in V(G_h)\setminus \{b_0\}$
\begin{align*}
[\Delta_{G_j} F_{j,N}^v](u)=\delta_v(u)
\end{align*}
\end{lemma}

\begin{proof}Note that
\begin{align*}
&[\Delta_{G_j} F_{j,N}^v](u)\\
&=\frac{1}{\sum_{y\in V(G_j):y\sim V}|\Delta(v,y)|}\left[\sum_{y\in V(G_j):y\sim u}|\Delta(u,y)|F_{j,N}^{v}(u)+\sum_{y\in V(G_h)\setminus \{b_0\}:y\sim u}\Delta(u,y)F_{j,N}^v(y)\right]\\
&=\frac{1}{\sum_{y\in V(G_h):y\sim v}|\Delta(v,y)|}\sum_{n=0}^{\infty}\left[\sum_{y\in V(G_h):y\sim u}|\Delta(u,y)|\mathbb{P}^u(W_{n,j}=v,\eta>n)\right.\\&+\left.\sum_{y\in V(G_h)\setminus\{b_0\}:y\sim u}\Delta(u,y)\mathbb{P}^y(W_{n,j}=v,\eta>n)\right]\\
&=\sum_{n=0}^{\infty}\frac{\sum_{y\in V(G_h)}|\Delta(u,y)|}{\sum_{y\in V(G_h)}|\Delta(v,y)|}[\mathbb{P}^u(W_{n,j}=v,\eta>n)-\mathbb{P}^u(W_{n+1,j}=v,\eta>n+1)]\\
&=\begin{cases}0&\mathrm{If}\ u\neq v\\
1&\mathrm{If}\ u=v.
\end{cases}
\end{align*}
\end{proof}

\subsection{Inverse matrices}
\begin{lemma}\label{l18}Suppose Assumption \ref{ap24} holds. Let $G_j$ be a finite, connected subgraph of $G$ consisting of faces of $G$.
\begin{enumerate}
\item 
\begin{align*}
\ol{D}_{j}^{-1}(b,w)=\frac{1}{\xi}(F^{b_1}_{j,d}(b)-F^{b_2}_{j,d}(b))\sqrt{\frac{\nu(b_1b_2)}{\nu([b_1b_2]^+)}};\qquad \forall b\in V(G_j^+).
\end{align*}
where 
\begin{itemize}
\item $b_1,b_2\in V^+$ such that $w$ is the white vertex of $G$ corresponding to the edge $(b_1,b_2)$; 
\item $\xi$ is the unit vector pointing from $b_2$ to $b_1$ in the Euclidean plane; 
\item $F^b_{j,d}(\cdot)$ is defined as in (\ref{dfuvd}).
\end{itemize}
\item \begin{align*}
\ol{D}_j^{-1}(b,w)=\frac{1}{\zeta}(F_{j,N}^{b_3}(b)-F_{j,N}^{b_4}(b))\sqrt{\frac{\nu(b_3b_4)}{\nu([b_3b_4]^+)}};\qquad \forall b\in V(G_j)\setminus \{b_0\}.
\end{align*}
where 
\begin{itemize}
\item $b_3,b_4\in V$ such that $w$ is the white vertex of $\ol{G}$ corresponding to the edge $(b_3,b_4)$; 
\item $\zeta$ is the unit vector pointing from $b_4$ to $b_3$ in the Euclidean plane;
\item $F^b_{j,N}(\cdot)$ is defined as in (\ref{dfuvn}).
\end{itemize}
\end{enumerate}
\end{lemma}

\begin{proof}Explicit computations show that
\begin{align*}
\ol{D}_j^*\ol{D}_j\ol{D}_j^{-1}(b,w)=\ol{D}_j^*(b,w)=\begin{cases}0&\mathrm{If}\ b\ \mathrm{and}\ w\ \mathrm{are\ not\ adjacent}\\ \frac{1}{\mathbf{i}\xi}\sqrt{\frac{\nu(b_3b_4)}{\nu(b_1b_2)}}&\mathrm{If}\ b\sim w,\ \mathrm{and}\ b\in V,\ \ \mathrm{and}\ b=b_3\\
-\frac{1}{\mathbf{i}\xi}\sqrt{\frac{\nu(b_3b_4)}{\nu(b_1b_2)}}&\mathrm{If}\ b\sim w,\ \mathrm{and}\ b\in V,\ \ \mathrm{and}\ b=b_4\\
\frac{1}{\xi}\sqrt{\frac{\nu(b_1b_2)}{\nu(b_3b_4)}}&\mathrm{If}\ b\sim w,\ \mathrm{and}\ b\in V^+,\  \mathrm{and}\ b=b_1,
\\ -\frac{1}{\xi}\sqrt{\frac{\nu(b_1b_2)}{\nu(b_3b_4)}}&\mathrm{If}\ b\sim w,\ \mathrm{and}\ b\in V^+,\  \mathrm{and}\ b=b_2
\end{cases}
\end{align*}
where the relative locations of $b_1,b_2,w$ are described in the lemma; and $(b_3,b_4)$ is the dual edge of $b_1b_2$; the unit vector pointing from $b_4$ to $b_3$ in the Euclidean plane is $\mathbf{i}\xi$, given that the unit vector pointing from $b_2$ to $b_1$ in the Euclidean plane is $\xi$. Then by (\ref{ddhd}), we obtain
\begin{align*}
\Delta_{G_j^+}\ol{D}_j^{-1}(b,w)=\ol{D}_j^*(b,w);\forall b\in V^+
\end{align*}
Let $b\in V^+$. Define a matrix
\begin{align*}
B(b,w):=\frac{1}{\xi}(F^{b_1}(b)-F^{b_2}(b))\sqrt{\frac{\nu(b_1b_2)}{\nu(b_3b_4)}};\qquad \forall b\in V^+.
\end{align*}

By Lemma \ref{l14}, we obtain
\begin{align*}
[\Delta_{G_j^+} B](b,w)
&=\sum_{b'\in V^+}\Delta_{G_j^+}(b,b')B(b',w)\\
&=\frac{1}{\xi}\sqrt{\frac{\nu(b_1b_2)}{\nu(b_3b_4)}}\left[\sum_{b'\in V^+}\Delta_{G_j^+}(b,b')F^{b_1}(b')-\sum_{b'\in V^+}\Delta_{G_j^+}(b,b')F^{b_2}(b')\right]\\
&=\frac{1}{\xi}\sqrt{\frac{\nu(b_1b_2)}{\nu(b_3b_4)}}\left[\delta_{b}(b_1)-\delta_b(b_2)\right];
\end{align*}
where the last identity follows from Lemma \ref{le217}. Then Part (1) of the Lemma follows. Parts (2) of the lemma follows from Lemma \ref{le219} similarly.
\end{proof}

\begin{theorem}\label{l422}Suppose that $G$ is an infinite, 3-connected, simple, proper plane graph with locally finite dual and admitting a double circle packing (see Proposition \ref{p27}) in the hyperbolic plane. Suppose Assumption \ref{ap24} holds. Assume $\{G_j\}_{j=1}^{\infty}$ is a sequence of finite, connected subgraphs of $G$ exhausting $G$, (i.e.$\cup_{j=1}^{\infty}G_j=\infty$). Let $b_{0,j}$ be the removed black vertex to obtain $\mathcal{G}_j$ from $\ol{G}_j$.
Assume 
\begin{align*}
\lim_{j\rightarrow\infty }b_{0,j}=b_0\in \partial \HH^2;
\end{align*}
where $b_0$ is an arbitrary point along $\partial\HH^2$ except for
\begin{itemize}
\item the limit points of a null family of singly infinite paths; or
\item the limit points of simple random walks with probability 0.
\end{itemize}

Let $\ol{D}_j$ be defined as in Definition \ref{df17}.
Then
\begin{enumerate}
\item
\begin{align*}
\lim_{j\rightarrow\infty}\ol{D}_{j}^{-1}(b,w)=\frac{1}{\xi}\left(\frac{G_{G^+}(b,b_1)}{|\Delta(b_1,b_1)|}-\frac{G_{G^+}(b,b_2)}{|\Delta(b_2,b_2)|}\right)\sqrt{\frac{\nu(b_1b_2)}{\nu([b_1b_2]^+)}};\qquad \forall b\in V(G_j^+).
\end{align*}
where $b_1,b_2,\xi$ are given as in Lemma \ref{l18}(1), and $G_{G^+}$ is the Green's function on $G^+$ as defined in Definition \ref{df12}.
\item  \begin{align}
\lim_{j\rightarrow\infty}\ol{D}_{j}^{-1}(b,w)=\frac{1}{\zeta}
\left(\frac{G_{G}(b,b_3)}{|\Delta(b_3,b_3)|}-\frac{G_{G}(b,b_4)}{|\Delta(b_4,b_4)|}+H_{b_3b_4}(b)\right)\sqrt{\frac{\nu(b_3b_4)}{\nu([b_3b_4]^+)}};\qquad \forall b\in V(G_j).\label{ss}
\end{align}
where $b_3,b_4,\zeta$ are given as in Lemma \ref{l18}(2), $G_G$ is the Green's function on $G$ as defined in Definition \ref{df12} and $H_{b_3b_4}$ is a Dirichlet harmonic function on $G$ satisfying
\begin{enumerate}
\item $H_{b_3b_4}(b_0)=0$;
\item \begin{align*}
\triangledown H_{b_3b_4}(e)=P_{\triangledown \mathbf{HD}}\chi_{b_3b_4}(e);
\end{align*}
i.e. the gradient of $H_{b_3b_4}$ is the orthogonal projection of $\chi_{b_3b_4}$ on the subspace $\triangledown\mathbf{HD}$.
\item Let $Z_n$ be the random walk on $G$ as defined in Definition \ref{df12} with $G^+$ replaced by $G$, then almost surely
\begin{align*}
\lim_{n\rightarrow\infty}\lim_{j\rightarrow\infty}\ol{D}_{j}^{-1}(Z_n,w)=\lim_{n\rightarrow\infty}\frac{1}{\zeta}
\left(\frac{G_{G}(Z_n,b_3)}{|\Delta(b_3,b_3)|}-\frac{G_{G}(Z_n,b_4)}{|\Delta(b_4,b_4)|}+H_{b_3b_4}(Z_n)\right)\sqrt{\frac{\nu(b_3b_4)}{\nu([b_3b_4]^+)}};
\end{align*}
\end{enumerate}
\end{enumerate}
\end{theorem}

\begin{proof}Part(1) of the lemma follows from Lemma \ref{l18}(1) and Lemma \ref{l17}.

Now we prove Part(2) of the lemma.
Let
\begin{align*}
F_{j,d,G}^{b_i}(b):=\frac{G_{j,d,G}(b,b_i)}{|\Delta(b_i,b_i)|};\qquad i\in\{3,4\};
\end{align*}
where $G_{j,d,G}$ is the Dirichlet Green's function on $G_j$ as defined in Definition \ref{df214} with $G_j^+$ replaced by $G_j$. Let
\begin{align*}
F_{j,N,G}^{b_i}(b):=F_{j,N}^{b_i}(b);\qquad i\in\{3,4\}.
\end{align*}
Define
\begin{align*}
H_{b_3b_4,j}(b):=F_{j,N,G}^{b_3}(b)-F_{j,N,G}^{b_4}(b)-F_{j,d,G}^{b_3}(b)-F_{j,d,G}^{b_4}(b);
\end{align*}
then it satisfies
\begin{align*}
\Delta H_{b_3,b_4,j}(z)=0;\qquad\forall z\in V(G_j).
\end{align*}
Moreover, let $e\in \vec{E}_j$ be a directed edge of $G_j$, then 
\begin{align*}
\triangledown H_{b_3b_4,j}(e)&=
\frac{\nu(e)}{\nu(e^+)}\left\{\left[F_{j,N,G}^{b_3}(\ol{e})-F_{j,N,G}^{b_4}(\ol{e})-F_{j,N,G}^{b_3}(\underline{e})+F_{j,N,G}^{b_4}(\underline{e})\right]\right.
\\&-\left.\left[F_{j,d,G}^{b_3}(\ol{e})-F_{j,d,G}^{b_4}(\ol{e})-F_{j,d,G}^{b_3}(\underline{e})+F_{j,d,G}^{b_4}(\underline{e})\right]\right\}.
\end{align*}

By Lemma \ref{l17}, we have
\begin{align*}
&\lim_{j\rightarrow\infty}\left[F_{j,d,G}^{b_3}(\ol{e})-F_{j,d,G}^{b_4}(\ol{e})-F_{j,d,G}^{b_3}(\underline{e})+F_{j,d,G}^{b_4}(\underline{e})\right]\\
&=\frac{G_{G}(\ol{e},b_3)}{|\Delta(b_3,b_3)|}-
\frac{G_{G}(\ol{e},b_4)}{|\Delta(b_4,b_4)|}-\frac{G_{G}(\underline{e},b_3)}{|\Delta(b_3,b_3)|}
+\frac{G_{G}(\underline{e},b_4)}{|\Delta(b_4,b_4)|};
\end{align*}
where $G_{G}(\cdot,\cdot)$ is the Green's function on the infinite transient graph $G$ as defined in Definition \ref{df12}.

Let $f$ be the directed edge with $\underline{f}=b_4$ and $\ol{f}=b_3$. Note that
\begin{align*}
&I_{e,\partial\HH^2}(f)=\frac{\nu(f)}{\nu(f^+)}\langle P_{\bigstar}\chi_e, \chi_f\rangle_{R};\qquad
&I_{e,\emptyset}(f)=\frac{\nu(f)}{\nu(f^+)}\langle P_{\diamondsuit^{\perp}}\chi_e, \chi_f\rangle_{R}.
\end{align*}
For $s\in \vec{E}$, define
\begin{align*}
J_e(s):&=\frac{\nu(s)}{\nu(s^+)}\left[\frac{G_{G}(\ol{e},\ol{s})}{|\Delta(\ol{s},\ol{s})|}-
\frac{G_{G}(\ol{e},\underline{s})}{\left|\Delta(\underline{s},\underline{s})\right|}-\frac{G_{G}(\underline{e},\ol{s})}{\left|\Delta(\ol{s},\ol{s})\right|}
+\frac{G_{G}(\underline{e},\underline{s})}{\left|\Delta(\underline{s},\underline{s})\right|}\right]\\
&=K_{\underline{e}}(s)-K_{\overline{e}}(s)
\end{align*}
where for each $v\in V$ and $s\in\vec{E}$, 
\begin{align*}
K_v(s):=\frac{\nu(s)}{\nu(s^+)}\left[-\frac{G_{G}(v,\ol{s})}{\left|\Delta(\ol{s},\ol{s})\right|}
+\frac{G_{G}(v,\underline{s})}{\left|\Delta(\underline{s},\underline{s})\right|}\right]
\end{align*}
is the expected number of times that a weighted random walk (as defined in \ref{drwd}) starting from $v$ uses $s$ minus the expected number of times that it uses $-s$. Then
\begin{align*}
\mathrm{div}K_{v}(w)=\frac{\delta_{v}(w)}{|\Delta(v,v)|}
\end{align*}
Then 
\begin{align*}
\mathrm{div} J_e(v)=\frac{1}{|\Delta(v,v)|}[\delta_{v}(\underline{e})-\delta_v(\overline{e})]=\mathrm{div} I_{e,\partial\HH^2}(v).
\end{align*}
Therefore $\mathrm{div} (J_e-I_{e,\partial\HH^2})=0$, which implies $(J_e-I_{e,\partial\HH^2})\perp \bigstar$. Since $I_{e,\partial\HH^2}\in \bigstar$, we deduce that $P_{\bigstar}J_e=I_{e,\partial \HH^2}$.

For any $v\in V$ and $s\in \vec{E}$, let $\theta_v(f)$ be the probability that the first step of the random walk starting at $v$ will use the edge $f$ minus the probability that it will use the edge $-f$. Then $\Delta(v,v)\theta_v=-\triangledown \mathbf{1}_v$. Hence $\theta_v\in \bigstar$. Note that
\begin{align*}
J_e=\sum_{v}[G_{G}(\underline{e},v)\theta_v-G_{G}(\overline{e},v)\theta_v];
\end{align*}
then $J_e\in \bigstar$; therefore $J_e=I_{e,\partial\HH^2}$.

For $s\in \vec{E}$, define
\begin{small}
\begin{align*}
&J_{e,j,N}(s):\\
&=\frac{\nu(s)}{\nu(s^+)}\left[\frac{G_{j,N,G}(\ol{e},\ol{s})}{\left|\sum_{x\in V(G_j):x\sim b_{0,j}}\Delta(\ol{s},x)\right|}-
\frac{G_{j,N,G}(\ol{e},\underline{s})}{\left|\sum_{x\in V(G_j),x\sim\underline{s}}\Delta(\underline{s},x)\right|}-\frac{G_{j,N,G}(\underline{e},\ol{s})}{\left|\sum_{x\in V(G_j),x\sim\ol{s}}\Delta(\ol{s},x)\right|}
+\frac{G_{j,N,G}(\underline{e},\underline{s})}{\left|\sum_{x\in V(G_j),x\sim\underline{s}}\Delta(\underline{s},x)\right|}\right]\\
&=K_{\underline{e},j,N}(s)-K_{\overline{e},j,N}(s).
\end{align*}
\end{small}
where for each $v\in V$ and $s\in\vec{E}$, 
\begin{align*}
K_{v,j,N}(s):=\frac{\nu(s)}{\nu(s^+)}\left[-\frac{G_{j,N,G}(v,\ol{s})}{\left|\sum_{x\in V(G_j),x\sim\ol{s}}\Delta(\ol{s},x)\right|}
+\frac{G_{j,N,G}(v,\underline{s})}{\left|\sum_{x\in V(G_j),x\sim\underline{s}}\Delta(\underline{s},x)\right|}\right]
\end{align*}
is the expected number of times that a weighted random walk on $G_j$ absorbed at $b_{0,j}$ (as defined in Definition \ref{df218}) starting from $v$ uses $s$ minus the expected number of times that it uses $-s$. Then
\begin{align*}
\mathrm{div}K_{v,j,N}(w)=
\begin{cases}\frac{\delta_{v}(w)}{|\sum_{x\in V(G_j):x\sim v}\Delta(v,x)|};&\mathrm{If}\ v\neq b_{0,j}\ \mathrm{and}\ w\neq b_{0,j};\\
0;&\mathrm{If}\ v= b_{0,j};\\
-\frac{1}{|\sum_{x\in V(G_j):x\sim b_{0,j}}\Delta(b_{0,j},x)|};&\mathrm{If}\ v\neq b_{0,j}\ \mathrm{and}\ w=b_{0,j}.
\end{cases}
\end{align*}
Assume neither endpoints of $e$ coincide with $b_{0,j}$. Then 
\begin{align*}
\mathrm{div} J_{e,j,N}(v)&=\frac{1}{\left|\sum_{x\in V(G_j):x\sim v}\Delta(v,x)\right|}[\delta_{v}(\underline{e})-\delta_v(\overline{e})]
=\mathrm{div} [P_{\bigstar_j}\chi_e](v).
\end{align*}
where $\bigstar_j:=\triangledown l^2(V_j)$.
Therefore $\mathrm{div} (J_{e,j,N}-P_{\bigstar_j}\chi_e)=0$, which implies $(J_{e,j,N}-P_{\bigstar_j}\chi_e)\perp \bigstar_j$. Since $P_{\bigstar_j}\chi_e\in \bigstar_j$, we deduce that $P_{\bigstar_j}J_{e,j,N}=P_{\bigstar_j}\chi_e$.

For any $v\in V_j$ and $s\in \vec{E}$, let $\theta_{j,N,v}(f)$ be the probability that the first step of the weighted random walk on $G_j$ absorbed at $b_{0,j}$ starting at $v$ will use the edge $f$ minus the probability that it will use the edge $-f$. Then $\left|\sum_{x\in V(G_j):x\sim v}\Delta(v,x)\right|\theta_{j,N,v}=-\triangledown \mathbf{1}_v$ if $v\neq b_{0,j}$ and $\theta_{j,N,b_{0,j}}=0$. Hence $\theta_{j,N,v}\in \bigstar_j$. Note that
\begin{align*}
J_{e,N,j}=\sum_{v\in V_j}[G_{j,N,G}(\underline{e},v)\theta_{j,N,v}-G_{j,N,G}(\overline{e},v)\theta_{j,N,v}];
\end{align*}
then $J_{e,N,j}\in \bigstar_j$; therefore 
\begin{align*}
J_{e,N,j}=P_{\bigstar_j}\chi_e.
\end{align*}
By Proposition \ref{pp16}, we obtain
\begin{align*}
\lim_{j\rightarrow\infty}P_{\bigstar_j}\chi_e=I_{e,\emptyset};
\end{align*}
and the convergence is in the norm $\|\cdot\|_{R}$. Then we obtain that
for any two adjacent vertices $b,b'\in V(G_j)$ with $e=(b,b')$
\begin{align*}
&\lim_{j\rightarrow\infty}\left[\ol{D}_{j}^{-1}(b,w)-\ol{D}_{j}^{-1}(b',w)\right]=\frac{1}{\zeta}\sqrt{\frac{\nu(b_3b_4)}{\nu([b_3b_4]^+)}}\frac{\nu([bb']^+)}{\nu(bb')}P_{\triangledown\mathbf{HD}}\chi_{b_3b_4}(e)\\
&+\frac{1}{\zeta}
\left(\frac{G_{G}(b,b_3)}{|\Delta(b_3,b_3)|}-\frac{G_{G}(b,b_4)}{|\Delta(b_4,b_4)|}-\frac{G_{G}(b',b_3)}{|\Delta(b_3,b_3)|}+\frac{G_{G}(b',b_4)}{|\Delta(b_4,b_4)|}\right)\sqrt{\frac{\nu(b_3b_4)}{\nu([b_3b_4]^+)}};
\end{align*}
where $b_3,b_4,\zeta$ are given as in Lemma \ref{l18}(2), $G_G$ is the Green's function on $G$ as defined in Definition \ref{df12}. 
Then (2a) follows. 

It remains to prove (2b). The proof is inspired by Theorem 1.1 in \cite{ARP99}. 
Let
\begin{align*}
&R_{n}:=\lim_{j\rightarrow\infty}\ol{D}_{j}^{-1}(Z_n,w)
-\frac{1}{\zeta}
\left(\frac{G_{G}(Z_n,b_3)}{|\Delta(b_3,b_3)|}-\frac{G_{G}(Z_n,b_4)}{|\Delta(b_4,b_4)|}+H_{b_3b_4}(Z_n)\right)\sqrt{\frac{\nu(b_3b_4)}{\nu([b_3b_4]^+)}}\\
&-\lim_{j\rightarrow\infty}\ol{D}_{j}^{-1}(Z_{n-1},w)
+\frac{1}{\zeta}
\left(\frac{G_{G}(Z_{n-1},b_3)}{|\Delta(b_3,b_3)|}-\frac{G_{G}(Z_{n-1},b_4)}{|\Delta(b_4,b_4)|}+H_{b_3b_4}(Z_{n-1})\right)\sqrt{\frac{\nu(b_3b_4)}{\nu([b_3b_4]^+)}}
\end{align*}

We shall prove that 
\begin{align}
&\lim_{n\rightarrow\infty}\sum_{k=1}^n R_{k}=0.\label{dfz}
\end{align}
Then the lemma follows.

Now we prove (\ref{dfz}). 
We have for any $\epsilon>0$
\begin{align*}
\mathbb{P}\left(\left|\sum_{k=1}^nR_{k}\right|>\epsilon\right)\leq \frac{\mathbb{E}\left|\sum_{k=1}^nR_{k}\right|^2}{\epsilon^2}
\end{align*}
Note that 
\begin{align*}
&\mathbb{E}\left|\sum_{k=1}^n R_{k}\right|^2=\sum_{k=1}^n\mathbb{E}[R_{k}]^2
+2\sum_{1\leq r<s\leq n}\mathbb{E}R_{r}R_{s}
\end{align*}
For any $1\leq r<s\leq n$, let $\mathcal{F}_{r}$ be the sigma algebra generated by $Z_0,\ldots,Z_r$
\begin{align*}
&\mathbb{E}R_{r}R_{s}
=\mathbb{E}\left[\mathbb{E}\left\{R_{r}R_{s}|\mathcal{F}_{r}\right\}\right]=\mathbb{E}R_{r}\mathbb{E}\left\{R_{s}|\mathcal{F}_{r}\right\}=0
\end{align*}
where the last identity follows from the fact that 
\begin{align*}
\mathbb{E}\left\{R_{s}|\mathcal{F}_{s-1}\right\}=0.
\end{align*}
and the harmonicity of $R_{n}$ with respect to the variable $Z_n$.
Hence we have 
\begin{align*}
&\mathbb{E}\left|\sum_{k=1}^n R_{k}\right|^2=\sum_{k=1}^n\mathbb{E}[R_{k}]^2
\end{align*}
Let $Z_0=x$ be an arbitrary vertex in $V$. Moreover
\begin{align*}
\sum_{k=1}^\infty\mathbb{E}[R_{k}]^2=\sum_{y\in V}G_{G}(x,y)\mathbb{E}[R_{1}|Z_0=y]^2
\end{align*}
Let $\tau_x$ be the first hitting time of the vertex $x$ by $Z_n$. Note that
\begin{align*}
\frac{G_{G}(x,y)}{|\Delta(y,y)|}=\frac{G_{G}(y,x)}{|\Delta(x,x)|}=\frac{\mathbb{P}(\tau_x<\infty|Z_0=y)G_{G}(x,x)}{|\Delta(x,x)|}
\leq \frac{G_{G}(x,x)}{|\Delta(x,x)|}
\end{align*}
Then we have 
\begin{align*}
\sum_{k=1}^\infty\mathbb{E}[R_{k}]^2\leq \frac{G_{G}(x,x)}{|\Delta(x,x)|}\sum_{y\in V}\mathbb{E}[R_{1}|Z_0=y]^2= 0,
\end{align*}
where the last identity follows from the fact that the left hand side of (\ref{ss}) converges to the right hand side of (\ref{ss}) in the $\|\cdot\|_{R}$ norm. Then we obtain for any $\epsilon>0$, $\mathbb{P}\left(\left|\sum_{k=1}^{\infty}R_{k}\right|>\epsilon\right)=0$, then (\ref{dfz}) follows.
\end{proof}

\subsection{Decay of the Dirichlet Green's Function}

\begin{lemma}(Proposition 6.6 in \cite{ly16})Let $G=(V,E)$ be an infinite connected, non-amenable graph with spectral radius $\rho<1$. For simple random walk on $G$,
\begin{align*}
p_n(x,y)\leq \sqrt{\frac{d(y)}{d(x)}}\rho^n
\end{align*}
where $d(y)$ (resp.~$d(x)$) is the degree of $y$ (resp.~$x$).
\end{lemma}

\begin{lemma}\label{ap31}Let $G=(V,E)$ be an infinite, connected, bounded degree, 3-connected, nonamenable, one-ended, properly-embedded planar graph with spectral radius $\rho(G)$. Assume the edge weights satisfy
\begin{align}
\nu(e)=\nu(e^+)=1;\ \forall e\in E.\label{aww}
\end{align}
Let $G_j$ be a finite subgraph of $G$ consisting of faces of $G$ bounded by a simple closed curve consisting of edges. 
Let $D$ be the maximal vertex degree of $G$.
\begin{align}
G_{j,d}(x,y)\leq G(x,y)\leq \sqrt{\frac{D}{3}}\frac{[\rho(G)]^{d_G(x,y)}}{1-\rho(G)}.,\ \forall x,y\in V(G)\label{ed2}
\end{align}
where $G_{j,d}(\cdot,\cdot)$ is the Dirichlet Green's function for $G_j$, and $G(\cdot,\cdot)$ is the discrete Green's function for the infinite graph $G$.
\end{lemma}

\begin{proof}Note that 
\begin{align*}
&G_{j,d}(u,v)\leq G(u,v)=\sum_{n=d_{G}(u,v)}^{\infty}p_n(u,v)
\leq \sum_{n=d_{G}(u,v)}^{\infty}\sqrt{\frac{d(v)}{d(u)}}[\rho(G)]^{n}\leq \sqrt{\frac{D}{3}}\frac{[\rho(G)]^{d_{G}(u,v)}}{1-\rho(G)}.
\end{align*}
Then (\ref{ed2}) follows.
\end{proof}

Here we obtain the exponential decay of $G(u,v)$ with respect to the distance of $u$ and $v$ from the nonamenability of the underlying graph. The decay rate of the Green's function $G(u,v)$ with respect to the distance of $u$ and $v$ is known to be related to the growth rate of the graph; see e.g. \cite{HSC93,LS97}.

Recall also the following general results proved in \cite{ly16}.

\begin{lemma}\label{le425}(Corollary 6.32 in \cite{ly16}) Let $p(\cdot,\cdot)$ be the transition probabilities of an irreducible Markov chain on a countable state space $V$. Assume the chain has an infinite stationary measure $\pi$. Let
$\pi_{\min}:=\inf_{x\in v}\pi(x)$.
For $x,y\in V$, let
\begin{align*}
Q(x,y):=\pi(x)p(x,y)
\end{align*}
For $S,A\subset V$, define
\begin{align*}
Q(S,A):=\sum_{s\in S,a\in A}Q(s,a);\qquad \mathrm{and}\ |\partial_E S|_Q:=Q(S,S^c)
\end{align*}
The edge expansion of a finite set $S\subset V$ is
\begin{align*}
\Phi_S:=\frac{|\partial_E S|_Q}{\pi (S)}.
\end{align*}
The expansion profile of the chain is defined for $u>0$ by 
\begin{align}
\Phi(u):=\inf\{\Phi_S:0<\pi(S)\leq u\},\label{dpu}
\end{align}
where we take the convention that the inf of an empty set is $\infty$. Suppose the Markov chain $(V,P)$ is reversible. Then
\begin{enumerate}
\item If $\Phi(u)\geq \varphi_0$ for some $\varphi_0>0$, then
\begin{align*}
\frac{p_n(x,y)}{\pi(y)}\leq \frac{4}{\pi_{\min}}\mathrm{exp}\left(-\frac{\varphi_0^2(n-1)}{16}\right)\qquad \forall x,y\in V,\ \mathrm{and}\ n\geq 1
\end{align*}
\item Let $d>0$. If $\Phi(u)\geq cu^{-\frac{1}{d}}$ for some $c>0$ and all $u>0$, then
\begin{align*}
\frac{p_n(x,y)}{\pi(y)}\leq C'n^{-\frac{d}{2}};\qquad \forall x,y\in V,\ \mathrm{and}\ n\geq 1
\end{align*}
where $C'=C'(d,c)$.
\item If $\Phi(u)\geq \frac{c}{\log (bu)}$ for some $b,c>0$ and all $u>0$, then
\begin{align*}
\frac{p_n(x,y)}{\pi(y)}\leq C_1\exp(-C_2n^{\frac{1}{3}});\qquad \forall x,y\in V,\ \mathrm{and}\ n\geq 1
\end{align*}
\end{enumerate}
\end{lemma}

\begin{lemma}\label{le426}Let $G=(V,E)$ be an infinite, connected, bounded degree, 3-connected, nonamenable, one-ended, properly-embedded planar graph. Consider the weighted random walk on defined by (\ref{drwd}) with $G^+$ replaced by $G$. Assume
\begin{align}
\Phi(u)\geq cu^{-\frac{1}{d}}\label{pud}
\end{align}
for some $d>2$. Then the discrete Green's function on the infinite graph $G$ satisfies
\begin{align*}
\frac{G(u,v)}{|\Delta(v,v)|}\leq C'' [d_{G}(u,v)]^{-\frac{d-2}{2}}
\end{align*}
where $C''=C''(d,c)$.
\end{lemma}

\begin{proof}The random walk defined by (\ref{drwd}) on $d$ is reversible and has an infinite stationary measure $\pi(x)=|\Delta(x,x)|$.

By Lemma \ref{le425}(2), we have
\begin{align*}
\frac{G(u,v)}{|\Delta(v,v)|}=\sum_{n=d_{G}(u,v)}^{\infty}\frac{p_n(u,v)}{|\Delta(v,v)|}\leq C'\sum_{n=d_{G}(u,v)}^{\infty}n^{-\frac{d}{2}}
\end{align*}
Then the lemma follows.
\end{proof}

\section{General Boundary Conditions}\label{sect:gbc}

In this section, we discuss the infinite volume Gibbs measure of perfect matchings on $\ol{G}$ obtained by using sequences of finite graphs exhausting the infinite graph with boundary conditions other than the Temperley boundary conditions.  We then prove that using a sequence of finite graphs with two convex corners and no concave corners, such that the two boundaries portions divided by the two convex corners converge to two closed subsets $\sA_0$ and $\sA_1$ of $\partial\HH^2$, the limits of the entries of the inverse of the weighted adjacency matrix converge to the difference of the Dirichlet Green's function plus explicit harmonic Dirichlet functions; see Lemma \ref{le67}.

\begin{assumption}\label{ap51}Suppose Assumption \ref{ap24} holds.
Consider a sequence of finite, simply-connected subgraphs of $\ol{G}$ given by
\begin{align*}
\tilde{G}_1,\tilde{G}_2,\ldots,\tilde{G}_n,\ldots
\end{align*}
satisfying all the following conditions
\begin{enumerate}
\item Each $\tilde{G}_i$ is induced by a finite subset of vertices $\tilde{V}_i$ of $\ol{G}$; and
\item For any two vertices $u,v\in V(\tilde{G}_i)\cap V$, such that $(u,v)\in E$ is an edge of $G$, the white vertex of $\ol{G}$ corresponding to $(u,v)$ is also a vertex of $\tilde{G}_i$; and  
\item For any two vertices $u,v\in V(\tilde{G}_i)\cap V^+$, such that $(u,v)\in E^+$ is an edge of $G^+$, the white vertex of $\ol{G}$ corresponding to $(u,v)$ is also a vertex of $\tilde{G}_i$; and
\item For each $i\geq 1$, $\tilde{G}_i\subset \tilde{G}_{i+1}$; and
\item $\cup_{i=1}^{\infty}\tilde{G}_i=\ol{G}$; and
\item Each $\tilde{G}_i$ admits a perfect matching.
\end{enumerate}
\end{assumption}

\begin{definition}Let $\tilde{G}_j$ be a finite subgraph of $\ol{G}$. We call a white vertex $w$ of $\tilde{G}_j$ a concave corner if there exists a face of $\ol{G}$ consisting of vertices $w,b_1,w',b_2$ such that $w,b_1,b_2$ are vertices of $\tilde{G}_j$, but $w'$ is not a vertex of $\tilde{G}_j$.

We call a white vertex $w$ a convex corner if there is exactly one incident face of $w$ in $\tilde{G}_j$.
\end{definition}

\begin{lemma}\label{le64}Let $\{\tilde{G}_j\}$ be a finite subgraph of $\ol{G}$ satisfying Assumption \ref{ap51}(1)(2)(3)(6). 
Let $K_j$ be the number of concave white corners in $\tilde{G}_j$. Let $w$ be a white vertex of $\ol{G}$ and let $b\in V^+$. Let
\begin{align*}
\ol{D}_{\tilde{G}_j}:=\left.\ol{D}\right|_{W(\tilde{G}_j),B(\tilde{G}_j)}
\end{align*}
i.e.~$\ol{D}_{\tilde{G}_j}$ is the submatrix of $\ol{D}$ with rows restricted to white vertices of $\tilde{G}_j$, and columns restricted to black vertices of $\tilde{G}_j$.

If $K_j=0$, then
\begin{itemize}
\item If $b$ is a vertex of $G$ and $v\in V(\tilde{G}_j)\cap V$, then
\begin{align*}
\ol{D}^{-1}_{\tilde{G}_j}(b,w)
=\frac{1}{\zeta}\sqrt{\frac{\nu(b_3b_4)}{\nu([b_3b_4]^+)}}\left[\Delta_{{G}_{1,j}}^{-1}(b,b_3)-\Delta_{G_{1,j}}^{-1}(b,b_4)\right]
\end{align*}
where $b_3$, $b_4$, $\zeta$ are given as in Lemma \ref{l18}(2).
\item If $b$ is a vertex of $\tilde{G}_j^+$, then
\begin{align*}
\ol{D}_{\tilde{G}_j}^{-1}(b,w)=\frac{1}{\xi}\sqrt{\frac{\nu(b_1b_2)}{\nu([b_1b_2]^+)}}
\left[\Delta_{G_{0,j}}^{-1}(b,b_1)-\Delta_{G_{0,j}}^{-1}(b,b_2)\right]
\end{align*}
where $b_1$, $b_2$, $\xi$ are given as in Lemma \ref{l18}(1).
\end{itemize}
\end{lemma}

\begin{proof}For each pair of black vertices $b_1,b_2$ of $\tilde{G}_j$, the following cases might occur:
\begin{enumerate}
\item If $b_1,b_2$ are adjacent vertices in $G$ or if $b_1, b_2$ are adjacent vertices in $G^+$
\begin{align*}
\ol{D}_{\tilde{G}_j}^*\ol{D}_{\tilde{G}_j}(b_1,b_2)=\frac{\nu(b_1b_2)}{\nu((b_1b_2)^+)}
\end{align*}
\item If $d_{\ol{G}}(b_1,b_2)>2$, then
\begin{align*}
\ol{D}_{\tilde{G}_j}^*\ol{D}_{\tilde{G}_j}(b_1,b_2)=0.
\end{align*}
\item If $d_{\ol{G}}(b_1,b_2)=2$, $b_1$ is a vertex of $G$, $b_2$ is a vertex of $G^+$, in the quadrilateral face $b_1,w,b_2,w'$ of $\ol{G}$, both $w$ and $w'$ are vertices of $\tilde{G}_j$, then
\begin{align*}
\ol{D}_{\tilde{G}_j}^*\ol{D}_{\tilde{G}_j}(b_1,b_2)=\ol{D}_{\tilde{G}_j}^*\ol{D}_{\tilde{G}_j}(b_2,b_1)=0.
\end{align*}
\item If $d_{\ol{G}}(b_1,b_2)=2$, $b_1$ is a vertex of $G$, $b_2$ is a vertex of $G^+$, in the quadrilateral face $b_1,w,b_2,w'$ of $\ol{G}$, exactly one of $w$ and $w'$ is a vertex of $\tilde{G}_j$, then
\begin{align*}
\ol{D}_{\tilde{G}_j}^*\ol{D}_{\tilde{G}_j}(b_1,b_2)=\mathbf{i}\  \mathrm{and}\ \ol{D}_{\tilde{G}_j}^*\ol{D}_{\tilde{G}_j}(b_2,b_1)=-\mathbf{i}
\end{align*}
or 
\begin{align*}
\ol{D}_{\tilde{G}_j}^*\ol{D}_{\tilde{G}_j}(b_1,b_2)=-\mathbf{i}\ \mathrm{and}\ \ol{D}_{\tilde{G}_j}^*\ol{D}_{\tilde{G}_j}(b_2,b_1)=\mathbf{i}
\end{align*}
\item 
\begin{align*}
\ol{D}_{\tilde{G}_j}^*\ol{D}_{\tilde{G}_j}(b_1,b_1)=-\sum_{w\in V(\tilde{G}_j):w\sim b_i}
\frac{\nu(e_{wb})}{\nu(e_{wb}^+)};
\end{align*}
where $e_{wb}$ is the edge in $G$ or $G^+$ with $wb$ as a half edge.
\end{enumerate}
Then we obtain
\begin{align*}
\ol{D}_{\tilde{G}_j}^*\ol{D}_{\tilde{G}_j}=\left(\begin{array}{cc}\Delta_{G_{1,j}}&\mathbf{i}A\\ -\mathbf{i}A^t&\Delta_{G_{0,j}}\end{array}\right)
\end{align*}
where $A$ is a matrix with entries $+1$ or $-1$, rows index by vertices of $G_{j,1}$ and columns indexed by vertices of $G_{j,0}$; and $A^t$ is the transpose of $A$. Note that the number of non-zero entries in $A$ is exactly $K_j$. Since
\begin{align*}
\ol{D}_{\tilde{G}_j}^*\ol{D}_{\tilde{G}_j}\ol{D}_{\tilde{G}_j}^{-1}=\ol{D}_{\tilde{G}_j}^*;
\end{align*}
We obtain that
\begin{align*}
\ol{D}_{\tilde{G}_j}^{-1}&=\left(\begin{array}{cc}\Delta_{G_{1,j}}&\mathbf{i}A\\ -\mathbf{i}A^t&\Delta_{G_{0,j}}\end{array}\right)^{-1}\ol{D}_{G_{0,j}}^*\\
&=\left(I+\left(\begin{array}{cc}0&\mathbf{i}\Delta_{G_{1,j}}^{-1}A\\ \mathbf{i}\Delta_{G_{0,j}}^{-1}A&0\end{array}\right)\right)^{-1}\left(\begin{array}{cc}\Delta_{G_{1,j}}^{-1}&0\\ 0&\Delta_{G_{0,j}}^{-1}\end{array}\right)\ol{D}_{\tilde{G}_j}^*
\end{align*}
Note that
\begin{itemize}
\item If $b$ is a vertex of $G$ and $v\in V(\tilde{G}_j)\cap V$, then
\begin{align*}
\left(\begin{array}{cc}\Delta_{G_{1,j}}^{-1}&0\\ 0&\Delta_{G_{0,j}}^{-1}\end{array}\right)\ol{D}_{\tilde{G}_j}^*(b,w)
=\frac{1}{\zeta}\sqrt{\frac{\nu(b_3b_4)}{\nu([b_3b_4]^+)}}\left[\Delta_{{G}_{1,j}}^{-1}(b,b_3)-\Delta_{G_{1,j}}^{-1}(b,b_4)\right]
\end{align*}
where $b_3$, $b_4$, $\zeta$ are given as in Lemma \ref{l18}(2).
\item If $b$ is a vertex of $\tilde{G}_j^+$, then
\begin{align*}
\left(\begin{array}{cc}\Delta_{G_{1,j}}^{-1}&0\\ 0&\Delta_{G_{0,j}}^{-1}\end{array}\right)\ol{D}_{\tilde{G}_j}^*(b,w)=\frac{1}{\xi}\sqrt{\frac{\nu(b_1b_2)}{\nu([b_1b_2]^+)}}
\left[\Delta_{G_{0,j}}^{-1}(b,b_1)-\Delta_{G_{0,j}}^{-1}(b,b_2)\right]
\end{align*}
where $b_1$, $b_2$, $\xi$ are given as in Lemma \ref{l18}(1).
\end{itemize}
If $K_j=0$ we have $A=0$, then the lemma follows.
\end{proof}

\begin{definition}Let $\tilde{G}_j$ be a finite subgraph of $\ol{G}$. Assume $\tilde{G}_j$ consist of faces of $\ol{G}$. Let $w$ be a white vertex of $\tilde{G}_j$.  We call $w$ a convex corner of $\tilde{G}_j$ if exactly one incident face of $w$ in $\ol{G}$ is in $\tilde{G}_j$.
\end{definition}

\begin{figure}
\centering
\includegraphics[width=.35\textwidth]{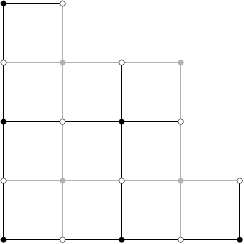}
\caption{A graph with two convex white corners and no concave white corners: the dual graph is represented by black lines; the primal graph is represented by gray lines.}
\label{fig:gpl}
\end{figure}

\begin{assumption}\label{ap66}Suppose Assumption \ref{ap311} holds with $\sA$ replaced by $\sA_0$. Assume each $G_{j,s}=(V_{j,s},E_{j,s})$ is a finite subgraph of $G$  such that $V_{j,s}$ consists all the vertices and edges in the bounded component of $\RR^2\setminus C_{j,s}$ (where $C_{j,s}$ is a simple closed curve consisting of edges of $G$), as well as all the vertices and edges along $C_{j,s}$; and $G_{j,s}$ is the induced subgraph of $V_{j,s}$. Assume there exist two vertices $v_{1,j,s}$ and $v_{2,j,s}$ along $C_{j,s}$, such that
\begin{align*}
C_{j,s}\setminus \{v_{1,j,s},v_{2,j,s}\}=C_{0,j,s}\cup C_{1,j,s}\ \mathrm{and}\ C_{0,j,s}\cap C_{1,j,s}=\emptyset.
\end{align*}
Assume 
\begin{align*}
\lim_{s\rightarrow\infty}v_{1,j,s}=z_1;\ \mathrm{and}\ \lim_{s\rightarrow\infty}v_{2,j,s}=z_2;\ \forall j\in \NN.
\end{align*}
and
\begin{align*}
\lim_{s\rightarrow\infty}C_{0,j,s}=\sA_0;
\end{align*}
where $\sA_0$ is a closed interval along $\partial\HH^2$, and the convergence is in the sense of Hausdorff distance of two closed set. Assume $\sA_1=\partial\HH^2\setminus \sA_0^{\circ}$; where $\sA_0^{\circ}$ is the interior of $\sA_0$.
Let $G_{j,s}^+=(V_{j,s}^+,E_{j,s}^+)$ be a finite subgraph of $G^+$ constructed as follows:
\begin{enumerate}
\item Let $\hat{G}_{j,s}^+$ be a subgraph of $G^+$ that contains $G_{j,s}$ as an interior dual graph;
\item Assume $\hat{G}_{j,s}^+$ is bounded by a simple closed curve $\hat{C}_{j,s}$ consisting of edges of $G^+$; i.e., the vertex set of $\hat{G}_{j,s}^+$ consists of all the vertices of $G^+$ in the bounded component of $\RR^2\setminus \hat{C}_{j,s}$; as well as vertices along $\hat{C}_{j,s}$, and $\hat{G}_{j,s}^+$ is the induced subgraph by its vertex set. Let $v^+_{1,j,s}$ be a vertex along $\hat{C}_{j,s}$ sharing a face with $v_{1,j,s}$ in $\ol{G}$; let $v^+_{2,j,s}$ be a vertex along $\hat{C}_{j,s}$ sharing a face with $v_{2,j,s}$ in $\ol{G}$;
\item Assume
\begin{align*}
\hat{C}_{j,s}\setminus \{v^+_{1,j,s},v^+_{2,j,s}\}=\hat{C}_{0,j,s}\cup \hat{C}_{1,j,s}\ \mathrm{and}\ \hat{C}_{0,j,s}\cap \hat{C}_{1,j,s}=\emptyset.
\end{align*}
Assume vertices in $\hat{C}_{1,j,s}\cap V^+$ share no faces in $\ol{G}$ with vertices in $C_{0,j,s}\cap V$; and vertices  in $\hat{C}_{0,j,s}\cap V^+$ share no faces in $\ol{G}$ with vertices in $C_{1,j,s}\cap V$.

Let $G_{j,s}^+$ be the subgraph of $\hat{G}_{j,s}^+$ obtained by removing all the vertices in $\hat{C}_{1,j,s}\cap V^+$ as well as their incident edges in $G^+$.
\end{enumerate}
Let $\tilde{G}_{j,s}$ be the bipartite graph obtained from the superposition of $G_{j,s}$ and $\hat{G}_{j,s}^+$, and then removing all the vertices along $\hat{C}_{j,s}$ between $v_{1,j,s}$ and $v_{2,j,s}$ surrounding $C_{1,j,s}$ as well as their incident edges in $\ol{G}$. See Figure \ref{fig:gpl}.
\end{assumption}

\begin{theorem}\label{le67}Suppose the Assumption \ref{ap66} holds. 

Then we have
\begin{enumerate}
\item If $b$ is a vertex of $G$ and $v\in V(\tilde{G}_{j,s})\cap V$, then
\begin{align}
\lim_{j\rightarrow\infty}\lim_{s\rightarrow\infty}\ol{D}^{-1}_{\tilde{G}_{j,s}}(b,w)
=\frac{1}{\zeta}\sqrt{\frac{\nu(b_3b_4)}{\nu([b_3b_4]^+)}}\left[\frac{G_{G}(b,b_3)}{\Delta(b_3,b_3)}-\frac{G_G(b,b_4)}{\Delta(b_4,b_4)}+H_{\sA_0,G,b_3b_4}(b)\right]
\label{piv}
\end{align}
where $b_3$, $b_4$, $\zeta$ are given as in Lemma \ref{l18}(2), $G_G$ is the Green's function on $G$ as defined in Definition \ref{df12} and $H_{\mathcal{A}_0,G,b_3b_4}$ is a harmonic Dirichlet function on $G$ satisfying
\begin{enumerate}[label=(\Alph*)]
\item For almost every singly infinite path $\gamma:=(\gamma(0),\gamma(1),\ldots)$ consisting of vertices in $G$ satisfying 
\begin{align*}
\lim_{n\rightarrow\infty}d_m(\gamma(n),z)=0  
\end{align*}
for some $z\in \mathcal{A}_0$,
we have
\begin{align*}
\lim_{n\rightarrow\infty}H_{\mathcal{A}_0,G,b_3b_4}(\gamma(n))=0;
\end{align*}
and
\item $\triangledown H_{\sA_0,G,b_3b_4}(e)=P_{\mathcal{S}_{\sA_0}(G)}\chi_{b_3b_4}(e)$;
i.e. the gradient of $H_{\sA_0,G,b_3b_4}$ is the orthogonal projection of $\chi_{b_3b_4}$ on the subspace $\mathcal{S}_{\sA_0}(G)$ of the gradient of harmonic Dirichlet functions on $G$ which are constant on $\mathcal{A}_0$.
\end{enumerate}
\item If $b$ is a vertex of $G^+$ and and $v\in V(\tilde{G}_{j,s})\cap V^+$, then
\begin{align*}
\lim_{j\rightarrow\infty}\lim_{s\rightarrow\infty}\ol{D}_{\tilde{G}_{j,s}}^{-1}(b,w)=\frac{1}{\xi}\sqrt{\frac{\nu(b_1b_2)}{\nu([b_1b_2]^+)}}
\left[\frac{G_{G^+}(b,b_1)}{\Delta(b_1,b_1)}-\frac{G_{G^+}(b,b_2)}{\Delta(b_2,b_2)}+H_{\sA_1,G^+,b_1b_2}(b)\right]
\end{align*}
where $b_1$, $b_2$, $\xi$ are given as in Lemma \ref{l18}(1), $G_{G^+}$ is the Green's function on $G^+$ as defined in Definition \ref{df12} and $H_{\mathcal{A}_1,G^+,b_3b_4}$ is a Dirichlet harmonic function on $G^+$ satisfying
\begin{itemize}
\item For almost every singly infinite path $\gamma:=(\gamma(0),\gamma(1),\ldots)$ consisting of vertices in $G^+$ satisfying 
\begin{align*}
\lim_{n\rightarrow\infty}d_m(\gamma(n),z)=0  
\end{align*}
for some $z\in \mathcal{A}_1$,
we have
\begin{align*}
\lim_{n\rightarrow\infty}H_{\mathcal{A}_1,G,b_3b_4}(\gamma(n))=0;
\end{align*}
and
\item $\triangledown H_{\sA_1,G^+,b_1b_2}(e)=P_{\mathcal{S}_{\sA_1}(G^+)}\chi_{b_1b_2}(e)$;
i.e. the gradient of $H_{\mathcal{A}_1,G^+,b_1b_2}$ is the orthogonal projection of $\chi_{b_1b_2}$ on the subspace $\mathcal{S}_{\sA_1}(G^+)$ of the gradient of harmonic Dirichlet functions on $G^+$ which are constants on $\sA_1$.
\end{itemize}
\item Suppose condition (a) in Lemma \ref{le39} holds with $\sA$ replaced by $\sA_0$; and condition (a) in Lemma \ref{le39} holds with $\sA$ replaced by $\sA_1$ and $G$ replaced by $G^+$. Then the infinite volume Gibbs measure obtained above is distinct from the infinite volume Gibbs measures obtained from the Temperley boundary conditions as obtained in Theorem \ref{l422}.
\end{enumerate}
\end{theorem}

\begin{proof}Note that under Assumption \ref{ap66}, the sequence of graphs $\{G_{j,s}\}$ satisfy Assumption \ref{ap311} with $\sA$ replaced by $\sA_0$, while the sequence of graphs $\{G_{j,s}^+\}$ satisfy Assumption \ref{ap312} with $\sA$ replaced by $\sA_1$. We only prove part (1) here; part (2) can be proved similarly.


Following similar arguments as in the proof of Lemma \ref{l422}, one obtains for each fixed $w$
\begin{align*}
\zeta \sqrt{\frac{\nu([b_3b_4]^+)}{\nu(b_3b_4)}}\ol{D}^{-1}_{\tilde{G}_{j,s}}(\cdot,w)-F_{j,d,G}^{b_3}(\cdot)+F_{j,d,G}^{b_4}(\cdot)\rightarrow H_{\sA_0,G,b_3,b_4}(\cdot)
\end{align*}
with respect to $\sqrt{\mathcal{E}(\triangledown)}$. Then the convergence of the left hand side to the right hand side for each $b\in V$ follows from Lemma \ref{ll39} and condition (A) of $H_{\sA_0,G,b_3,b_4}(\cdot)$.

Part (3) of the lemma follows from the explicit expressions of the infinite volume Gibbs measures in Theorem \ref{l422}, part (1)(2) and Theorem \ref{le39}.
\end{proof}

\section{Extremity}\label{sect:ex}

In this section, we prove that the infinite-volume  measure for dimer configurations on $\ol{G}$ obtained as in Theorem \ref{le67} is extreme; see Proposition \ref{lle83}. The idea is as follows. We first show that the extremity of the infinite-volume measure is equivalent to the tail-triviality; then we show that the infinite-volume measure is tail-trivial. The proof of the tail-triviality of the infinite volume measure is inspired by the proof of Theorem 8.4 in \cite{BLPS01}. One of the major differences is instead of considering uniform spanning trees, we are considering perfect matchings. The corresponding anti-symmetric function for the perfect matchings is a ``discrete integral" version of the anti-symmetric function for the spanning trees. Also it is straightforward to see that the anti-symmetric function for spanning trees has finite energy, the fact that the anti-symmetric function for perfect matchings has finite energy is not so obvious; but we still manage to prove it by using the transience property of the graph. Using similar technique, we prove that the infinite-volume  measure for dimer configurations on $\ol{G}$ obtained as in Theorem \ref{l422} is extreme under natural technical assumptions; see Propositions \ref{lle614} and \ref{p620}.

\subsection{Extremity and Tail-triviality}

We first review a few definitions and lemmas.

\begin{definition}\label{df71}Suppose Assumption \ref{ap24} holds. Consider the embedding of $\ol{G}$ into the hyperbolic plane such that $(G,G^+)$ are nerves of the double circle packings. Define a unit flow from black vertices of $\ol{G}$ to white vertices of $\ol{G}$ as follows. Let $e=(w,b)$ be an edge of $\ol{G}$ with black vertex $b$ and white vertex $w$. The edge $e$ is incident to two quadrilateral faces $f_1$ and $f_2$ in $\ol{G}$. For $i\in\{1,2\}$, let $d_i$ be the diagonal of $f_i$ with $b$ as one endpoint. Let $\theta_e$ be the angle formed by $d_1$ and $d_2$ containing the edge $e$. Defined the flow from $b$ to $w$ to be $\frac{\theta_e}{2\pi}$. Then it is straightforward to check the following
\begin{align}
\sum_{e:e\sim w}\frac{\theta_e}{2\pi}=1;\qquad \sum_{e:e\sim b}\frac{\theta_e}{2\pi}=1.\label{uf}
\end{align}
Let $M$ be a perfect matching on $\ol{G}$. Define a preliminary height function $\ol{h}_M$ associate to $M$ on faces of $\ol{G}$ as follows. Let $f_1$ and $f_2$ be two faces of $G$ sharing an edge $e=(w,b)$.
\begin{itemize}
\item If $e\in M$, and
\begin{itemize}
\item moving from $f_1$ to $f_2$ along the dual edge $e^+$, the black vertex $b$ is on the left, then $\ol{h}_{M}(f_1)-\ol{h}_{M}(f_2)=1-\frac{\theta_e}{2\pi}$;  
\item moving from $f_1$ to $f_2$ along the dual edge $e^+$, the black vertex $b$ is on the right, then $\ol{h}_{M}(f_1)-\ol{h}_{M}(f_2)=-1+\frac{\theta_e}{2\pi}$; 
\end{itemize}
\item If $e\notin M$, and
\item moving from $f_1$ to $f_2$ along the dual edge $e^+$, the black vertex $b$ is on the left, then $\ol{h}_{M}(f_1)-\ol{h}_{M}(f_2)=-\frac{\theta_e}{2\pi}$;  
\item moving from $f_1$ to $f_2$ along the dual edge $e^+$, the black vertex $b$ is on the right, then $\ol{h}_{M}(f_1)-\ol{h}_{M}(f_2)=\frac{\theta_e}{2\pi}$;
\end{itemize}
The condition (\ref{uf}) guarantees that the preliminary height change around any vertex is 0. Therefore $\ol{h}_M$ is well-defined up to an additive constant. One can make $\ol{h}_M$ uniquely defined by fixing the height at one specific face to be 0.

Let $M_0$ be a fixed perfect matching on $\ol{G}$. Define the height function $h_M$ associated to $M$ as a function from faces of $\ol{G}$ to the set of integers $\ZZ$ by
\begin{align*}
h_M(f)=\ol{h}_{M}(f)-\ol{h}_{M_0}(f).
\end{align*}
\end{definition}

\begin{definition}\label{df72}Suppose Assumption \ref{ap24} holds. Let $M_0$ be the fixed perfect matching on $\ol{G}$ as in Definition \ref{df71}. 
\begin{enumerate}[label=(\alph*)]
\item For each edge $e$ of $\ol{G}$, let $f_1,f_2$ be the two faces of $\ol{G}$ incident to $e$, such that moving along $e^+$ from $f_1$ to $f_2$, the black vertex of $e$ is no the left.
\end{enumerate}
Define a potential $V_e: \ZZ\rightarrow [0,\infty)$ as follows:
\begin{align*}
V_e(\eta)=\begin{cases}0&\mathrm{if}\ \eta=0\\0&\mathrm{if}\ h_{M_0}(f_1)>h_{M_0}(f_2)\ \mathrm{and}\ \eta=-1\\0&\mathrm{if}\ h_{M_0}(f_1)<h_{M_0}(f_2)\ \mathrm{and}\ \eta=1\\ \infty&\mathrm{otherwise}.
\end{cases}
\end{align*}

Let $F(\ol{G})$ be the set of all the faces of $\ol{G}$. Let $\Omega$ be the set of functions from $F(\ol{G})$ to $\ZZ$; and let $\mathcal{F}$ be the Borel $\sigma$-algebra of the product topology on $\Omega$. Let $\Lambda$ be an arbitrary finite subset of $F(\ol{G})$.
A gradient Gibbs measure $\mu$ on $\Omega$ is a probability measure on $\Omega$ satisfying the following conditions. For each $\phi\in \Omega$, let $\phi|_{\Lambda}$ (resp.\ $\phi|_{F(\ol{G})\setminus\Lambda}$) be the restriction of $\phi$ onto $\Lambda$ (resp.\ $F(\ol{G})\setminus \Lambda$).
\begin{align*}
\mu(\phi_{\Lambda}|\phi_{F(\ol{G})\setminus{\Lambda}})\propto \exp (-H_{\Lambda}(\phi))
\end{align*}
where 
\begin{align*}
H_{\Lambda}(\phi):&=\sum_{e^+=(f_1,f_2):f_1,f_2\in \Lambda}V_e(\phi(f_1)-\phi(f_2))+\sum_{e^+=(f_1,f_2):f_1\in \Lambda,f_2\notin \Lambda}V_e(\phi(f_1)-\phi(f_2))\\
&+\sum_{e^+=(f_1,f_2):f_1\notin \Lambda,f_2\in \Lambda}V_e(\phi(f_1)-\phi(f_2))
\end{align*}
and $f_1,f_2$ satisfies (a).
\end{definition}

Then it is straightforward to verify the following lemma
\begin{lemma}Given the 1-1 correspondence between perfect matchings on $\ol{G}$ and its height function, any gradient Gibbs measure satisfying  on Definition \ref{df72} gives a Gibbs measure on perfect matchings of $\ol{G}$ with edge weights given by
\begin{align}
\nu(e)=\nu(e^+)=1;\label{wtw}
\end{align}
for any edge $e$ of $G$.
\end{lemma}

\begin{lemma}\label{l74}Let $G=(V,E)$ be an infinite, connected graph. Let $\mathcal{Q}$ be the collection of all gradient Gibbs measures on $\Omega$ as defined by Definition \ref{df72}. A Gibbs measure is called  extreme in $\mathcal{Q}$ if it is an extreme element in the convex set $\mathcal{Q}$. Let $\mathrm{ex}\mathcal{Q}$ be the set of all the extremal Gibbs measures in $\mathcal{Q}$. For each $A\in \mathcal{F}$, let $e_{A}:\mu\mapsto \mu(A)$ be the evaluation map.
Denote by $e(\mathrm{ex}\mathcal{Q})$ the smallest $\sigma$-algebra on $\mathrm{ex}(\mathcal{Q})$ with respect to which each $e_A$ is measurable. Then 
\begin{enumerate}
\item For each $\nu\in\sQ$, $\nu\in ex\sQ$ if and only if $\nu$ is tail-trivial;
\item For each $\mu\in \mathcal{Q}$, there exists a unique weight $w_{\mu}$, which is a probability measure on $(\mathrm{ex}\mathcal{Q},e(\mathrm{ex}\mathcal{Q}))$, such that for each $A\in \mathcal{F}$, 
\begin{align}
\mu(A)=\int_{\mathrm{ex}\mathcal{Q}}\nu(A)w_{\mu}(d\nu)\label{edc}
\end{align}
The mapping $\mu\mapsto w_{\mu}$ is a bijection between $\mathcal{Q}$ and $\mathcal{P}(\mathrm{ex}\mathcal{Q},e(\mathrm{ex}\mathcal{Q}))$ (the set of all the probability measures on $(\mathrm{ex}\mathcal{Q},e(\mathrm{ex}\mathcal{Q}))$).
\end{enumerate}
\end{lemma}

\begin{proof}Part (1) of the lemma follows from Theorem 7.7(1) in \cite{geo88}. Part (2) of the lemma follows from Lemma 3.2.4 of \cite{SS03}; see also Theorem 7.26 of \cite{geo88}.
\end{proof}

\subsection{General boundary conditions.}

\begin{proposition}\label{lle83}Let $G=(V,E)$ be a 3-connected, transient, simple proper planar graph with bounded vertex degree and locally finite dual. Suppose Assumptions \ref{ap24},  \ref{ap66} hold. Let $\mu$ be the probability measure for dimer coverings on $\ol{G}$ obtained as in Lemma \ref{le67}. Assume the edge weights satisfy
(\ref{wtw}).
Then the measure $\mu$ is tail-trivial. 
\end{proposition}

In order to prove Proposition \ref{lle83}, we first review and prove a few definitions and lemmas. Recall the flows are defined in Definition \ref{df311}.

\begin{lemma}(Exercise 9.2 in \cite{ly16})\label{ls84}Let $H_n$ be increasing closed subspaces of a Hilbert space $H$ and $P_n$ be the orthogonal
projection on $H_n$. Let $P$ be the orthogonal projection on the closure of $\cup_n
H_n$. Then for
all $u\in H$, we have $\|P_n u-Pu\| \rightarrow 0$ as $n\rightarrow\infty$.
\end{lemma}

\begin{lemma}\label{ls85}
(Thomson’s Principle; see page 35 of \cite{ly16}). Let $G$ be a finite network and $A$ and $Z$ be two disjoint subsets of its
vertices. Let $\theta$ be a flow from A to Z and $i=P_{\bigstar}\theta$ be the current flow from $A$ to $Z$ with $\mathrm{div} i= \mathrm{div} \theta$.
Then $\mathcal{E}(\theta)>\mathcal{E}(i)$ unless $\theta=i$.
\end{lemma}

\begin{lemma}\label{ls86}(Proposition 2.12 in \cite{ly16})Let G be a transient connected network and $G_n$ be finite induced subnetworks that contain a vertex $a$ and that exhaust $G$. Identify the vertices outside $G_n$ to $z_n$,
forming $G^W_n$. Let $i_n$ be the unit current flow in $G^W_n$ from $a$ to $z_n$. Then $\{i_n\}$ has a pointwise
limit $i$ on $G$, which is the unique unit flow on $G$ from $a$ to $\infty$ of minimum energy. Let $v_n$ be
the voltages on $G^W_n$
corresponding to $i_n$ and with $v_n(z_n):= 0$. Then $v := lim_{n\rightarrow\infty} v_n$ exists on G
and has the following properties.
Start a random walk at $a$. For all vertices $x$, the expected number of visits to $x$ is
$G(a, x) = c(x)v(x)$. 
For all edges $e$, the expected signed number of crossings of e is $i(e)$.
\end{lemma}

\begin{lemma}\label{le87}Let $G=(V,E)$ be a 3-connected, transient, simple proper planar graph with bounded vertex degree and locally finite dual. Suppose Assumptions \ref{ap24} and \ref{ap66} hold. Assume the edge weights satisfy (\ref{wtw}). Assume $\sA_1$ in Lemma \ref{le67} satisfies condition (a) of Lemma \ref{le39} with $\sA$ replaced by $\sA_1$.  For each black vertex $b$ of $\ol{G}$, if $b$ is a vertex of $G$, 
\begin{align*}
i_b(e)=\lim_{j\rightarrow\infty}\lim_{s\rightarrow\infty}\zeta_e\ol{D}_{\tilde{G}_{j,s}}^{-1}(b,w_e);\ \forall (b,e)\in  (V,\vec{E});
\end{align*}
is a finite energy unit flow from $b$ to $\infty$
on $G$, where $w_e$ is the white vertex corresponding to the edge $e$, and
\begin{align}
\zeta_e:=\frac{\ol{e}-\underline{e}}{|\ol{e}-\underline{e}|}\label{dze}
\end{align}
is the unit vector pointing from $\underline{e}$ to $\ol{e}$.
\end{lemma}

\begin{proof}
Let $b\in V$ be arbitrary. By condition (a) of Lemma \ref{le39} and Corollary \ref{c313}, there exists an open set $\mathcal{O}_{\sA_1}\supset \sA_1$, such that $G\setminus U_{\sA_1}$ is transient 
for every open set $U_{\sA_1}$ satisfying $\sA_1\subset U_{\sA_1}\subset \mathcal{O}_{\sA_1}$. We may choose $U_{\sA_1}$ such that $b\in G\setminus U_{\sA_1}$.
By Lemma \ref{ls83}, there exists a unit flow $q_b$ on $G\setminus U_{\sA_1}$ of finite energy from every vertex $b$ of $G\setminus U_{\sA_1}$ to $\infty$. 

Let $G_j$ be the graph as Assumption \ref{ap311} with $\sA$ replaced by $\sA_0$.
We may choose $j$ sufficiently large so that $G\setminus U_{\sA_1}$ is a subgraph of $G_j$.
Then $q_b$ is a unit flow on $G_j$ from $b$ to $\infty$. Let $G_{j,s}^{\bullet}$ be the subgraph of $G_{j}$ as constructed in Assumption \ref{ap311}, and let $\bigstar_{j,s}$ be the $\bigstar$ subspace of antisymmetric functions on directed edges of $G_{j,s}^{\bullet}$. Define
\begin{align*}
q_{b,j,s}=P_{\bigstar_{j,s}}q_b
\end{align*}
Then we have
\begin{align*}
\bigstar_{j,s}\subset \bigstar_{j,s+1},\ \mathrm{and}\ \cup_s \bigstar_{j,s}=\bigstar_{j}
\end{align*}
where $\bigstar_{j}$ is the $\bigstar$ subspace of anti-symmetric functions on directed edges of $G_j$. By Lemma \ref{ls84},
\begin{align*}
\lim_{s\rightarrow\infty}\|q_{b,j,s}-P_{\bigstar_j}q_b\|=0.
\end{align*}
Moreover,
\begin{align*}
\bigstar_j\subset \bigstar_{j+1};\ \mathrm{and}\ \cup_j\bigstar_j=\diamondsuit_{\sA_0}^{\perp}
\end{align*}
where $\diamondsuit_{\sA_0}$ is the $\diamondsuit$ subspace of anti-symmetric functions on directed edges of $G$ with $\sA_0$ contracted to one vertex. Again by Lemma \ref{ls84} we infer that
\begin{align*}
\lim_{j\rightarrow\infty}\|P_{\bigstar_j}q_b-P_{\diamondsuit_{\sA_0}^{\perp}}q_b\|=0;
\end{align*}

We claim that $P_{\diamondsuit_{\sA_0}^{\perp}}q_b$ is a unit current flow from $b$ to $\infty$, which gives us exactly $i_b$. To see why that is true, note that
$P_{\bigstar_{j,s}}q_b$ is a unit current flow on $G_{j,s}^{\bullet}$ with divergence $0$ everywhere other than $b$ and $z_{j,s}$. Its voltage function is harmonic everywhere other than $b$ and $U_{j,s}$, and equals a constant in $U_{j,s}$, hence it must be a constant multiple of the Greens function of the random walk on $G_{j,s}$ absorbed at $U_{j,s}$, plus an arbitrary constant.
Taking limit by letting $s\rightarrow\infty$, and then $j\rightarrow\infty$, gives us exactly
$i_b$.

Hence $i_b$ has finite energy because $q_b$ has finite energy, and the energy of $P_{\diamondsuit_{\sA_0}^{\perp}}q_b$ is less than that of $q_b$.
\end{proof}

\noindent\textbf{Proof of Proposition \ref{lle83}.} 
Let $F$ be a nonempty set of edges of $\ol{G}$, and let $K$ be a nonempty set of white vertices of $\ol{G}$, such that the vertex set of $F$ and $K$ are disjoint. 
Let $\ol{E}_K$ be all the incident edges of $K$ in $\ol{G}$.

Let $A_1\in \mathcal{F}(F)$ and $A_2\in \mathcal{F}(\ol{E}_K)$. 

We shall show that
\begin{small}
\begin{align}
\label{covd}|\mu(A_1\cap A_2)-\mu(A_1)\mu(A_2)|\leq 4^{2|F|}|F|\left[\sum_{b\in B(F)\cap V}\|P_{\langle \chi_{F_1^c}\rangle}i_b\|^2+\sum_{b\in B(F)\cap V^+}\|P_{\langle\chi_{F_2^c}\rangle}i_b\|^2\right]
\end{align}
\end{small}
where $\tilde{F}_1$ (resp.\ $\tilde{F}_2$) consists of all the edges of $G$ (resp.\ $G^+$) with a half edge in $F_1$ (resp.\ $F_2$); then the tail-triviality of the measure $\mu$ follows.

Let $M$ be a random perfect matching. We have
\begin{align*}
\mu[F\subset M|M\cap \ol{E}_K]=\left|\det \left[\ol{D}^{K}\right]^{-1}_F\right|
\end{align*}
where $\ol{D}^{K}$ is the weighted adjacency matrix obtained from $\ol{D}$ by removing all the vertices in $M\cap \ol{E}_K$ and their incident edges, and $[\ol{D}^{K}]^{-1}_F$ is the square submatrix of $[\ol{D}^{K}]^{-1}$ with rows indexed by black vertices of edges in $F$ and columns indexed by white vertices in $F$. It is straightforward to see that if two edges in $F$ share a white vertex, then $\det[\ol{D}^{K}]^{-1}_F=0$ since there are two identical columns in $[\ol{D}^{K}]^{-1}_F$; and the probability that all these edges appear in a perfect matching is indeed 0.

For each black vertex $w$ of $\ol{G}$, let $e_{w}$ be the edge corresponding to $w$ in $E$ and $e_w^+$ be the edge corresponding to $w$ in $E^+$.
We shall construct a new graph $\tilde{G}_{j,s}^{K}$ as a subgraph of $G_{j,s}^{\bullet}$ in the following way 
\begin{itemize}
\item Consider Assumption \ref{ap311}. For each edge $(b,w)\in M\cap \ol{E}_K$ and $b\in V$, remove the edge $e_w$ from $G_j$ and identify $b$ with all the vertices in $U_{j,s}$.
\item For each edge $(b,w)\in  \ol{E}_K\cap M$ and $b\in V^+$,  remove the edge $e_w$ from $G_{j,s}$. 
\end{itemize}
The resulting graph obtained from $G_{j,s}^{\bullet}$ is $\tilde{G}_{j,s}^{K}$.

Similarly, we construct a new graph $\tilde{G}_{j,s,+}^{K}$ in the following way 
\begin{itemize}
\item Consider Assumptions \ref{ap311} and \ref{ap66}, let $G_{j,s}^{+,\bullet}$ be the graph obtained from $G_{j,s}^+$ by contracting all the dual vertices (or equivalently, faces) in $G\setminus G_j$ to the same vertex $z_{j,s}^{+}$ as well as dual edges joining two such dual vertices.
\item For each edge $(b,w)\in  \ol{E}_K\cap M$ and $b\in V$,  remove the edge $e_w^+$ from $G_{j,s}^{+,\bullet}$. 
\item For each edge $(b,w)\in  \ol{E}_K\cap M$ and $b\in V^+$, remove the edge $e_w^+$ from $G_{j,s}^{+,\bullet}$ and identify the vertex $b$ with all the dual vertices in $G\setminus G_j$.
\end{itemize}
The resulting graph obtained from $G_{j,s}^{+,\bullet}$ is $\tilde{G}_{j,s,+}^{K}$.

Let $Q_{\tilde{G}_{j,s}^{K}}$ (resp.\ $Q_{\tilde{G}_{j,s+}^{K}}$) be the orthogonal projection onto the $\bigstar$ subspace of $\tilde{G}_{j,s}^{K}$ (resp.\ $\tilde{G}_{j,s,+}^{K}$).

For a black vertex $b$ of edges in $F$, and a white vertex $w$ of edges in $F$, the following cases may occur
\begin{itemize}
\item $b$ is a vertex of $V$; then 
\begin{align*}
[\ol{D}^{K}]^{-1}(b,w)=\frac{1}{\zeta_{e_w}}\langle  Q_{\tilde{G}_{j,s}}^{K}i_b, \chi_{e_w}\rangle 
\end{align*}
where the direction of $e_w$ can be any of the two possible ones.
\item $b$ is a vertex of $V^+$; then 
\begin{align*}
[\ol{D}^{K}]^{-1}(b,w)=\frac{1}{\zeta_{e_w^+}}\langle  Q_{\tilde{G}_{j,s,+}}^{K}i_b, \chi_{e_w^+}\rangle 
\end{align*}
where the direction of $e_w^+$ can be any of the two possible ones.
\end{itemize}

Let
\begin{align*}
Q_{j,s,K}i_b:=\begin{cases}Q_{\tilde{G}_{j,s}}^{K}i_b&\mathrm{if}\ b\in V\\ Q_{\tilde{G}_{j,s,+}}^{K}i_b&\mathrm{if}\ b\in V^+ \end{cases}
\end{align*}

Then
\begin{small}
\begin{align}
\label{dq}\mathrm{Var}\left(\mu[F\subset M|M\cap \ol{E}_K]\right)&=\mathrm{Var}\left(\left|\det\left[\langle Q_{j,s,K}i_b,\chi_{e_{w,b}} \rangle \right]_{b\in B(F),w\in W(F)}\right|\right)\\
&\leq \mathrm{Var}\left(\det\left[\langle Q_{j,s,K}i_b,\chi_{e_{w,b}} \rangle \right]_{b\in B(F),w\in W(F)}\right)\notag
\end{align}
\end{small}
where the $\mathrm{Var}$ is with respect to the distribution of $M\cap \ol{E}_K$, $B(F)$ is the set of all the black vertices of edges in $F$, $W(F)$ is the set of all the white vertices of edges in $F$; where $\zeta_{e_w}$ is defined as in (\ref{dze}) and
\begin{align*}
\chi_{e_{w,b}}=\begin{cases}\chi_{e_w}&\mathrm{if}\ b\in V\\ \chi_{e_w^+}&\mathrm{if}\ b\in V^+\end{cases}
\end{align*}

We make the following claim
\begin{claim}\label{cl89}
\begin{align*}
&\mathrm{Var}(\det[\langle Q_{j,s,K}i_b,\chi_{e_{w,b}} \rangle ]_{b\in B(F),w\in W(F)}\leq 2|F|\left[\sum_{b\in B(F)\cap V}\|P_{K_1}i_b\|_R^2+\sum_{b\in B(F)\cap V^+}\|P_{K_2}i_b\|_R^2\right]
\end{align*}
\end{claim}
Claim \ref{cl89} will be proved later. 

Now we apply Claim \ref{cl89} to prove Proposition (\ref{covd}). As discussed before, (\ref{covd}) will imply the tail-triviality of the measure $\mu$. 
Given $A_1\in \mathcal{F}(F)$, $A_1$ is the union at most $2^{|F|}$ events, each of which specifies the states of each edge in the perfect matching; each such event is the union or difference of at most $2^{|F|}$ events, each of which is the event that a certain subset of edges of $F$ is present in the perfect matching. Therefore if we write
\begin{align*}
a:= 4^{2|F|+1}|F|\left[\sum_{b\in B(F)\cap V}\|P_{K_1}i_b\|_R^2+\sum_{b\in B(F)\cap V^+}\|P_{K_2}i_b\|_R^2\right]
\end{align*}
Then
\begin{align*}
\mathrm{Var}(\mu(A_1)|M\cap \ol{E}_K)\leq a
\end{align*}

By convexity of quadratic functions we have
\begin{align*}
|\mu[A_1 |A_2]-\mu(A_1)|^2\mu(A_2)
\leq a
\end{align*}
and therefore 
\begin{align*}
|\mu[A_1 |A_2]-\mu(A_1)|^2 
\leq a\mu(A_2)\leq a.
\end{align*}
$\hfill\Box$

We recall a lemma before proving Claim \ref{cl89}.

\begin{lemma}\label{l810}(Lemma 8.7 in \cite{BLPS01})Let $\mathbb{P}$ be any probability measure on the set of $k\times k$ real matrices. For a matrix $A$, write its rows as $A_i$ and its entries as $A_{i,j}$ ($i=1,2,\ldots,k$), if $\mathbb{E}\det A=\det \mathbb{E}A$ and each $\|A_i\|_2\leq 1$, then
\begin{align*}
\mathrm{Var}\det A\leq k\sum_{i,j=1}^k \mathrm{Var}A_{i,j}.
\end{align*}
\end{lemma}

\noindent\textbf{Proof of Claim \ref{cl89}.} 
Now we fix the direction of each $e_{w,b}$ to be from $w$ to $b$.
Note that
\begin{small}
\begin{align}
\label{ssa}&\mathbb{E}\det[\langle Q_{j,s,K}i_b,\chi_{e_{w,b}} \rangle]_{b\in B(F),w\in W(F)}
=\mathbb{E}\frac{(-1)^{\sum_{b\in B(F)}\sum_{w\in W(F)}|b|+|w|}\det\ol{D}^K_{\setminus\{B(F),W(F)\}}}{\det\ol{D}^K}\\
&=\mathbb{E}\mu[F\subset M|M\cap \ol{E}_K]=\mu[F\subset M]=\frac{(-1)^{\sum_{b\in B(F)}\sum_{w\in W(F)}|b|+|w|}\det\ol{D}_{\tilde{G}_{j,s}\setminus \{B(F),W(F)\}}}{\det\ol{D}_{\tilde{G}_{j,s}}}
\notag
\end{align}
\end{small}
where 
\begin{itemize}
\item $\ol{D}^K_{\setminus\{B(F),W(F)\}}$ is the submatrix of $\ol{D}^K$ by removing rows and columns indexed by $B(F)$ and $W(F)$, respectively.
\end{itemize}

Now we show that
\begin{align}
\frac{(-1)^{\sum_{b\in B(F)}\sum_{w\in W(F)}|b|+|w|}\det\ol{D}_{\tilde{G}_{j,s}\setminus \{B(F),W(F)\}}}{\det\ol{D}_{\tilde{G}_{j,s}}} =\det\mathbb{E}[\langle Q_{j,s,K}i_b,\chi_{e_{w,b}} \rangle]_{b\in B(F),w\in W(F)}\label{sss}
\end{align}
Since
\begin{align*}
&\mathbb{E}[\langle Q_{j,s,K}i_b,\chi_{e_{w,b}} \rangle]=\sum_{W':W'=W(M\cap\ol{E}_K )}\frac{(-1)^{|w|+|b|+\sum_{w'\in W'}\sum_{b'\in K}|w'|+|b'|}\det{\ol{D}}_{\tilde{G}_{j,s}\setminus[\{b,w\}\cup W'\cup K]}}{\det \ol{D}_{\tilde{G}_{j,s}}}\\
&=\frac{(-1)^{|w|}\det{\ol{D}}_{\tilde{G}_{j,s}\setminus[\{b,w\}\cup W'\cup K]}}{\det \ol{D}_{\tilde{G}_{j,s}}}
\end{align*}
\begin{itemize}
\item $W(M\cap \ol{E}_K)$ is the set of white vertices of edges in $M\cap \ol{E}_K$ (recall that the set of black vertices of edges in $M\cap \ol{E}_K$ is $K$).
\end{itemize}
We have (\ref{sss}) follows.

It follows from (\ref{ssa}), (\ref{sss}) that
\begin{align*}
\mathbb{E}\det[\langle Q_{j,s,K}i_b,\chi_{e_{w,b}} \rangle]_{b\in B(F),w\in W(F)}=\det\mathbb{E}[\langle Q_{j,s,K}i_b,\chi_{e_{w,b}} \rangle]_{b\in B(F),w\in W(F)}
\end{align*}

One may also interpret (\ref{ssa}) as follows. By Temperley's bijection, perfect matching configurations are in 1-1 correspondence with directed spanning trees (recall that the dual tree configuration is uniquely determined by the spanning tree configuration on the primal graph). The probabilistic interpretation of $\langle Q_{j,s,K}i_b,\chi_{e_w,b} \rangle$ is, conditional on the directed tree configurations on edges and dual edges corresponding to white vertices in $F$, the probability that the branch of the tree staring at $b$ passes through $e_w$ minus the probability that it passes through $-e_w$. Then the last identity of (\ref{ssa}) follows.

Then by Lemma \ref{l810}, we have
\begin{small}
\begin{align}
&\mathrm{Var}(\det[\langle Q_{j,s,K}i_b,\chi_{e_{w,b}} \rangle ]_{b\in B(F),w\in W(F)}\leq |F|\sum_{b\in B(F),w\in W(F)}\mathrm{Var}
\langle Q_{j,s,K}i_b,\chi_{e_{w,b}}\rangle\label{lss}\\
&=|F|\left[\sum_{b\in B(F)\cap V,w\in W(F)}\mathrm{Var}
\langle Q_{j,s,K}i_b,\chi_{e_w}\rangle+\sum_{b\in B(F)\cap V^+,w\in W(F)}\mathrm{Var}
\langle Q_{j,s,K}i_b,\chi_{e_w^+}\rangle\right]\notag
\end{align}
\end{small}

Recall that the graph $G_{j,s}^{\bullet}$ in Assumption \ref{ap311} has a unique identified boundary vertex $z_{j,s}$ with Dirichlet boundary condition; let $\bigstar_{G_{j,s}^{\bullet}}$ be the $\bigstar$ space of anti-symmetric functions on directed edges of $\bigstar_{G_{j,s}^{\bullet}}$.
Let 
\begin{align*}
K_1:=\{e_w:w\in K\};\qquad K_2:=\{e_w^+:w\in K\}.
\end{align*}

For each possible $S=\ol{E}_K\cap M$, let $L_S$ consist of all the shortest paths in $G_{j,s}^{\bullet}$ joining $z_{j,s}$ and a vertex in $S\cap V$ without using edges in $K_1$. 

Then
\begin{align}
\mathbb{E}Q_{\tilde{G}_{j,s}^K}i_b=\sum_{S\subset \ol{E}_K}\mathbb{P}[M\cap \ol{E}_K=S]P_{[\bigstar_{G_{j,s}^{\bullet}}+\langle \chi_{K_1}\rangle]\cap\langle \chi_{K_1}\rangle^{\perp}\cap\langle L_S\rangle^{\perp}}i_b\label{eqe}
\end{align}
where $\langle \chi_{K_1}\rangle$ is the subspace of antisymmetric functions generated by $\{\chi_{e}\}_{e\in K_1}$.

We shall need the following claims

\begin{claim}\label{cl612}\begin{enumerate}
\item 
\begin{align*}
\mathbb{E}Q_{\tilde{G}_{j,s}^K}=P_{K_1}^{\perp}P_{\bigstar_{G_{j,s}^{\bullet}}}P_{K_1}^{\perp}; \forall b\in B(F)\cap V,\  b\notin K.
\end{align*}
\item 
\begin{align*}
\mathbb{E}Q_{\tilde{G}_{j,s}^K}=P_{K_2}^{\perp}P_{\bigstar_{G_{j,s}^{+,\bullet}}}P_{K_2}^{\perp}; \forall b\in B(F)\cap V^+,\  b\notin K.
\end{align*}
\end{enumerate}
\end{claim}

\begin{claim}\label{cl613}For any $\xi\in l_{-}(E(G_{j,s}^{\bullet}))$, we have
\begin{align*}
\mathrm{Var}(Q_{j,s,K}\xi)
:=\mathbb{E}\|Q_{j,s,K}\xi-\mathbb{E}Q_{j,s,K}\xi\|_R^2=\|P_{K_1}P_{\bigstar_{G_{j,s}^{\bullet}}}P_{K_1}^{\perp}\xi\|_R^2.
\end{align*}
For any $\eta\in l_{-}(E(G_{j,s}^{+,\bullet}))$, we have
\begin{align*}
\mathrm{Var}(Q_{j,s,K}\eta)
:=\mathbb{E}\|Q_{j,s,K}\eta-\mathbb{E}Q_{j,s,K}\eta\|_R^2=\|P_{K_2}P_{\bigstar_{G_{j,s}^{+,\bullet}}}P_{K_2}^{\perp}\eta\|_R^2.
\end{align*}
\end{claim}

Claims \ref{cl612} and \ref{cl613} will be proved later.

We infer that
\begin{small}
\begin{align}
&\sum_{w\in W(F)}\mathrm{Var}\langle Q_{j,s,K}i_b,\chi_{e_w}\rangle
\leq\mathbb{E}\|Q_{j,s,K}i_b-\mathbb{E}Q_{j,s,K}i_b\|^2_R\leq 
\|P_{K_1}P_{\bigstar_{G_{j,s}^{\bullet}}}P_{K_1}^{\perp}i_b\|_R^2\label{qq1}\\
&\notag=\|P_{K_1}P_{\bigstar_{G_{j,s}^{\bullet}}}(1-P_{K_1})i_b\|_R^2
\leq 2\|P_{K_1}i_b\|^2
\end{align}
\end{small}

Similarly we can show that if $b\in V^+$
\begin{align}\label{qq2}
\sum_{w\in W(F)}\mathrm{Var}\langle Q_{j,s,K}i_b,\chi_{e_w}\rangle
\leq 2
\|P_{K_2}i_b\|_R^2.
\end{align}
Then the Claim \ref{cl89} follows from (\ref{qq1}), (\ref{qq2}) and (\ref{lss}).
$\hfill\Box$

\bigskip

\noindent\textbf{Proof of Claim \ref{cl612}.}\ 
We only prove part (1) of Claim \ref{cl612}, part (2) can be proved using the same arguments.

Let $T$ be a random spanning tree of $G_{j,s}^{\bullet}$ and let $e=(u,v)\in \vec{E}(G_{j,s}^{\bullet})$. Let $L_{u,v}$ be the oriented path in $T$ starting from $u$ and ending in $v$.
Let
\begin{align*}
\zeta_{T}^e=\sum_{e_i\in L_{u,v}}\chi_{e_i}
\end{align*}
Then for any $f\in \vec{E}(G_{j,s}^{\bullet})$
\begin{align}\label{iez}
&P_{\bigstar_{G_{j,s}^{\bullet}}}\chi_e(f)=
\mathbb{E}\zeta_T^{e}(f)
=\left[-\frac{G_{G_{j,s}^{
\bullet
},N}(u,\ol{f})}{\left|\sum_{x\in V(G_{j,s}^{\bullet}),x\sim\ol{f}}\Delta(\ol{f},x)\right|}
+\frac{G_{G_{j,s}^{\bullet},N}(u,\underline{f})}{\left|\sum_{x\in V(G_{j,s}^{\bullet}),x\sim\underline{f}}\Delta(\underline{f},x)\right|}\right.\\&\left.+\frac{G_{G_{j,s}^{
\bullet
},N}(v,\ol{f})}{\left|\sum_{x\in V(G_{j,s}^{\bullet}),x\sim\ol{f}}\Delta(\ol{f},x)\right|}
-\frac{G_{G_{j,s}^{\bullet},N}(v,\underline{f})}{\left|\sum_{x\in V(G_{j,s}^{\bullet}),x\sim\underline{f}}\Delta(\underline{f},x)\right|}\right]\notag
\end{align}
where $G_{G_{j,s}^{
\bullet
},N}$ is the Green's function for the random walk on $G_{j,s}^{
\bullet
}$ absorbed at $z_{j,s}$.

Let $M$ be the random perfect matching obtained from $T$ via the Temperley bijection.
More precisely, let $T^+$ be the dual spanning tree on $G_{j,s}^{+,\bullet}$ such that a dual edge $e^+\in T^+$ if and only if its corresponding primal edge $e\notin T$. For each $e\in T$ (resp.\ $e^+\in T^+$), let $L_e$ ($L_{e^+}$)  be the unique branch of $T$ (resp.~$T^+$) starting from $e$ (resp.~$T^+$) and ending in  $z_{j,s}$ (resp.~$z_{j,s}^+$), orient $e=(u,v)$ (resp.~$e^+=(x,y)$) from $u$ (resp.~$x$) to $v$ (resp.~$y$) such that $u$ (resp.~$x$) is the first vertex along $L_{e}$ ($L_{e^+}$).
Then an edge $s=(b,w)$ of $\ol{G}$ is present in the perfect matching $M$ if and only if one of the following two conditions holds
\begin{itemize}
\item $s$ is a half edge of a directed edge $e\in T$ such that $e$ starts at $b$; or 
\item $s$ is a half edge of a directed edge $e^+\in T^+$ such that $e^+$ starts at $b$.
\end{itemize}

Similarly to (\ref{iez}) we have if $b\in V$ and $e\notin K_1$,
\begin{align*}
P_{[\bigstar_{G_{j,s}^{\bullet}}+\langle \chi_{K_1}\rangle]\cap\langle \chi_{K_1}\rangle^{\perp}\cap\langle L_S\rangle^{\perp}}\chi_e
=P_{K_1}^{\perp}\mathbb{E}[\zeta_T^e|M\cap \ol{E}_K=S]
\end{align*}

To prove Claim \ref{cl612}(1), it suffices to show that for any $e,h\in \vec{E}(G_{j,s}^{\bullet})$,
\begin{align}
\langle\mathbb{E}Q_{\tilde{G}_{j,s}^K}\chi_e,\chi_h\rangle=\langle P_{K_1}^{\perp}P_{\bigstar_{G_{j,s}^{\bullet}}}P_{K_1}^{\perp}\chi_e, \chi_h\rangle.\label{p6131}
\end{align}
If $h\in K_1$ or $e\in K_1$, by (\ref{eqe}) we see both the left hand side and the right hand side of (\ref{p6131}) are 0. Assume $h\notin K_1$ and $e\notin K_1$; then to prove (\ref{p6131}) it suffices to show that
\begin{align}
\langle\mathbb{E}Q_{\tilde{G}_{j,s}^K}\chi_e,\chi_h\rangle=\langle P_{\bigstar_{G_{j,s}^{\bullet}}}\chi_e, \chi_h\rangle\label{pp6131}
\end{align}
Note that
\begin{align*}
\langle P_{\bigstar_{G_{j,s}^{\bullet}}}\chi_e, \chi_h\rangle&=\langle \mathbb{E}\zeta_{T}^e,\chi_h \rangle
=\langle \mathbb{E}\zeta_{T}^e,P_{K_1}^{\perp}\chi_h \rangle
=\langle P_{K_1}^{\perp}\mathbb{E}\zeta_{T}^e,\chi_h \rangle\\
&=\sum_{S\subset \ol{E}_K}\mathbb{P}(M\cap \ol{E}_K=S)\langle P_{K_1}^{\perp}\mathbb{E}[\zeta_T^e|M\cap \ol{E}_K=S],\chi_h \rangle\\
&=\sum_{S\subset \ol{E}_K}\mathbb{P}(M\cap \ol{E}_K=S)\langle P_{[\bigstar_{G_{j,s}^{\bullet}}+\langle \chi_{K_1}\rangle]\cap\langle \chi_{K_1}\rangle^{\perp}\cap\langle L_S\rangle^{\perp}}\chi_e,\chi_h \rangle\\
&=\langle\mathbb{E}Q_{\tilde{G}_{j,s}^K}\chi_e,\chi_h\rangle.
\end{align*}
Then Claim \ref{cl612}(1) follows.
$\hfill\Box$

\bigskip
\noindent\textbf{Proof of Claim \ref{cl613}.} Same arguments as the proof of Lemma 8.6 in \cite{BLPS01}.
$\hfill\Box$

\begin{remark}\label{rk7}Note that although the infinite-volume Gibbs measure obtained in Theorem \ref{le67} is extremal, it is in general not automorphism-invariant with respect to the whole automorphism group of $\ol{G}$. For example let $G$ be the degree-7 transitive triangulation of the hyperbolic plane. Let $b_1$ be a vertex of $G$ and let $b_2$ be a vertex of $G^+$. Let $\{e_i\}_{i\in\{1,2,\ldots,7\}}$ be all the edges incident to $b_1$ and let 
$\{f_j\}_{j\in\{1,2,3\}}$ be all the edges incident to $b_2$. If $\eta$ is an automorphism-invariant measure for perfect matchings on $\ol{G}$ with respect to the whole automorphism group, then
\begin{align*}
&\eta(e_i\in M)=\frac{1}{7};\qquad \forall i\in\{1,2,\ldots,7\}
\\
&\eta(f_j\in M)=\frac{1}{3};\qquad \forall j\in\{1,2,3\}
\end{align*}
then at each white vertex, the sum of probabilities of the 4 incident edges present in a perfect matching is $\frac{1}{3}+\frac{1}{3}+\frac{1}{7}+\frac{1}{7}<1$, which is not possible. However, from Theorem \ref{le67}, one sees that this infinite-volume Gibbs measure is invariant with respect to automorphisms of $G$ that preserving $\sA_0$. But the automorphisms of $G$ that preserving $\sA_0$ do not have a finite orbit.
\end{remark}

\subsection{Temperley boundary conditions}

\begin{assumption}\label{ap615}Suppose the edge weights satisfy (\ref{wtw}). Let $G_j$ be a finite subgraph of $G$ as in section \ref{sect42}. Let $b_{0,j}$ be the removed vertex to obtain $\mathcal{G}_j$ from $\ol{G}_j$. For $b\in V(G_j)$, add an extra edge $e_{bb_{0,j}}$ with end points $b$ and $b_{0,j}$ in $G_j$ and let $G_{j,e_{bb_{0,j}}}$ be the new graph $G_j\cup \{e_{bb_{0,j}}\}$. Define a flow on the new graph by 
\begin{align*}
\theta_j:=P_{\bigstar_{G_{j,e_{bb_{0,j}}}}}\chi_{(b,b_{0,j})}
\end{align*}
where $(b,b_{0,j})$ is the directed edge from $b$ to $b_{0,j}$, and $\bigstar_{G_{j,e_{bb_{0,j}}}}$ is the $\bigstar$ subspace of the graph $G_{j,e_{bb_{0,j}}}$. Assume $\lim_{j\rightarrow\infty}b_{0,j}=b_0$.
Assume as $j\rightarrow\infty$, $\theta_j$ converges edgewise to a flow $\theta$ such that
\begin{align*}
\langle \theta, \chi(b,b_{0}) \rangle\neq 1 
\end{align*}
\end{assumption}

\begin{lemma}Suppose Assumption \ref{ap615} holds. Then 
\begin{align}
\frac{P_{(b,b_{0})}^{\perp}\theta}{1-\langle \theta, \chi(b,b_{0}) \rangle} \label{ucf}
\end{align}
is a finite energy unit current flow from $b$ to $b_0$ in $G$.
\end{lemma}

\begin{proof}Note that $\chi_{(b,b_{0,j})}$ is a unit flow from $b$ to $b_{0,j}$ on $G_{j,e_{bb_{0,j}}}$. Since the orthogonal projection does not change the divergence, we infer that $\theta_j$ is a unit current flow from $b$ to $b_{0,j}$ on $G_{j,e_{bb_{0,j}}}$. The edgewise limit $\theta$ is a unit current flow from $b$ to $b_0$ in $G_{e_{bb_{0}}}$, where $G_{e_{bb_{0}}}$ is the graph obtained from $G$ by adding an extra edge joining $b$ and $b_0$. $P_{(b,b_{0})}^{\perp}\theta$ is a current flow on $G$ with divergence $1-\langle \theta, \chi(b,b_{0})\rangle$ at $b$ and divergence $-1+\langle \theta, \chi(b,b_{0})\rangle$ at $b_0$. When $1-\langle \theta, \chi(b,b_{0})\rangle\neq 0$, (\ref{ucf}) is a unit current flow on $G$ with Dirichlet energy bounded above by $\frac{1}{[1-\langle \theta, \chi(b,b_{0})\rangle]^2}$. Then the lemma follows.
\end{proof}

\begin{lemma}\label{lle615}Let $G=(V,E)$ be a 3-connected, transient, simple proper planar graph with bounded vertex degree and locally finite dual. Suppose Assumptions \ref{ap24} and \ref{ap66} hold. Assume the edge weights satisfy (\ref{wtw}). Consider the infinite volume Gibbs measure obtained in Theorem \ref{l422}.  For each black vertex $b$ of $\ol{G}\cap V^+$, 
\begin{align*}
i_b(e):=\lim_{j\rightarrow\infty}\zeta_e\ol{D}_{\ol{G}_{j}}^{-1}(b,w_e);\ \forall (b,e)\in  (V^+,\vec{E});
\end{align*}
is a finite energy unit flow from $b$ to $\infty$
on $G$, where $w_e$ is the white vertex corresponding to the edge $e$, and $\zeta_e$ is given as in (\ref{dze}).
\end{lemma}

\begin{proof}Same arguments as in the proof of Lemma \ref{le87}.
\end{proof}

\begin{proposition}\label{lle614}Let $G=(V,E)$ be a 3-connected, transient, simple proper planar graph with bounded vertex degree and locally finite dual. Suppose Assumptions \ref{ap24},  \ref{ap615} hold. Let $\mu$ be the probability measure for dimer coverings on $\ol{G}$ obtained as in Lemma \ref{l422}. Assume the edge weights satisfy
(\ref{wtw}).
Then the measure $\mu$ is tail-trivial. 
\end{proposition}

\begin{proof}
 Let $F,K$, $K_1$, $K_2$ be given as in the proof of Proposition \ref{lle83}. 

Compared to Proposition \ref{lle83} in this case we may consider $\sA_0=\{b_0\}\in \partial\HH^2$. In this case we no longer have condition (a) in Lemma \ref{le39} holds with $\sA$ replaced by $\sA_1$. Therefore we cannot prove the $i_b$ for $b\in V$ in this case has finite energy as in Lemma \ref{le87}. 
However, we can still follow the computations as in the proof of Proposition \ref{lle83} and obtain
(\ref{lss}). Following the same arguments as in the proof of Claim \ref{cl89} and applying Lemma \ref{lle615}, we obtain (\ref{qq2}).
It remains to give an upper bound for the first term in the right hand side of (\ref{ss}) in the parentheses. Let
\begin{align*}
i_{b,j}(e):=\zeta_e\ol{D}_{\ol{G}_{j}}^{-1}(b,w_e)
\end{align*}

By Claim \ref{cl612} we obtain
\begin{align*}
&\sum_{b\in B(F)\cap V,w\in W(F)}\mathrm{Var}\langle Q_{j,s,K}i_{b,j},\chi_{e_w}\rangle
=\sum_{b\in B(F)\cap V,w\in W(F)}
\mathbb{E}[\langle Q_{j,s,K}i_{b,j},\chi_{e_w} \rangle 
-\mathbb{E}\langle Q_{j,s,K}i_{b,j},\chi_{e_w} \rangle
]^2\\
&=\sum_{b\in B(F)\cap V,w\in W(F)}
\mathbb{E}[\langle i_{b,j},Q_{j,s,K}\chi_{e_w} \rangle 
-\mathbb{E}\langle i_{b,j},Q_{j,s,K}\chi_{e_w} \rangle
]^2\\
&=\sum_{b\in B(F)\cap V,w\in W(F)}
\mathbb{E}[\langle i_{b,j},Q_{j,s,K}\chi_{e_w}-\mathbb{E}Q_{j,s,K}\chi_{e_w} \rangle 
]^2\leq \|i_{b,j}\|_R^2\mathbb{E}\|Q_{j,s,K}\chi_{e_w}-\mathbb{E}Q_{j,s,K}\chi_{e_w}\|^2
\end{align*}
When Assumption \ref{ap615} holds, we have
\begin{align*}
\lim_{j\rightarrow\infty}\|i_{b,j}\|_R^2<\infty
\end{align*}
and by Claim \ref{cl613} we have
\begin{align*}
\mathbb{E}\|Q_{j,s,K}\chi_{e_w}-\mathbb{E}Q_{j,s,K}\chi_{e_w}\|^2=\|P_{K_1}P_{\bigstar G_j}P_{K_1}^{\perp}\chi_{e_w}\|^2
\end{align*}
When $e_w\notin K_1$ we have 
\begin{align*}
\lim_{j\rightarrow\infty}\|P_{K_1}P_{\bigstar G_j}P_{K_1}^{\perp}\chi_{e_w}\|^2\|
=\|P_{K_1}P_{\bigstar G_j}\chi_{e_w}\|^2
\end{align*}
which goes to 0 as the distance of $K_1$ and $e_w$ goes to infinity. Then the proposition follows.
\end{proof}

\begin{example}\label{ex619}We shall give examples where Assumption \ref{ap615} holds and where Assumption \ref{ap615} does not hold.

First consider the graph $G$ with vertex set $\NN$ and edges joining two nonnegative integers with distance 1. Let $G_j$ be the graph with vertex set $\{0,1,\ldots,j\}$. Let $b=0$ and $b_{0,j}=j$. Then 
\begin{align*}
&\theta_j(i-1,i)=\frac{1}{j},\qquad\forall 1\leq i\leq j;\\
&\theta_j(0,j)=\frac{j-1}{j}
\end{align*}
We have
\begin{align*}
\lim_{j\rightarrow\infty}\theta_j(0,j)=1.
\end{align*}
Assumption \ref{ap615} does not hold in this case.

Now consider $G_j$ to be constructed from two depth $j$ trees $T_1$ and $T_2$ as follows
\begin{itemize}
\item The root of $T_1$ is $b$;
\item The root of $T_2$ is $b_{0,j}$;
\item Every vertex in $T_1$ and $T_2$ other than the leaves has two children; such that both $T_1$ and $T_2$ have exactly $2^j$ leaves;
\item Identify each leave with $T_1$ with a unique leave of $T_2$ gives $G_j$.
\end{itemize}
In this case we have
\begin{align*}
\theta_j(b,b_{0,j})=\frac{4-2^{2-j}}{6-2^{2-j}}
\end{align*}
If $x$ is a depth $k-1$ vertex of $T_1$ while $y$ is a depth $k$ vertex of $T_1$ (or if $x$ is a depth $k$ vertex of $T_2$ and $y$ is a depth $k-1$ vertex of $T_2$), we have
\begin{align*}
\theta_j(x,y)=\frac{2^{1-k}}{6-2^{2-j}};\qquad \forall 1\leq k\leq j
\end{align*}
In this case we have
\begin{align*}
\lim_{j\rightarrow\infty}\theta_j(b,b_{0,j})=\frac{2}{3}\neq 1;
\end{align*}
Assumption \ref{ap615} holds in this case.
\end{example}

In Example \ref{ex619}, we have an nonamenable graph where Assumption \ref{ap615} holds. However, as we shall see in the next proposition, nonamenability is not a necessary condition for Assumption \ref{ap615}.

\begin{proposition}\label{p620}Let $G_j$ be finite subgraphs of $G$ as in section \ref{sect42}. Let $b_{0,j}$ be the removed vertex of $G_j$ to obtain $\mathcal{G}_j$ from $\ol{G}_j$.  Assume for each vertex $b$ of $G$ and all sufficiently large $j$, there exists a subtree $T_{j,b}$ of an infinite transient tree $T$ rooted at $b$ and a subtree $T_{j,b_{0,j}}$ of $T$ rooted at $b_{0,j}$ such that
\begin{itemize}
\item $T_{j,b}$ and $T_{j,b_{0,j}}$ are edge-disjoint subgraphs of $G$; and
\item there exists an isomorphism $\sigma: T_{j,b}\rightarrow T_{j,b_{0,j}}$, such that $\sigma(b)=b_{0,j}$; and if $v$ is a leave of $T_{j,b}$ then $\sigma(v)=v$; and
\item $\lim_{j\rightarrow\infty }T_{j,b}$ is isomorphic to $T$. 
\end{itemize}
Let $G=(V,E)$ be a 3-connected, transient, simple proper planar graph with bounded vertex degree and locally finite dual. Suppose Assumptions \ref{ap24},  \ref{ap615} hold. Let $\mu$ be the probability measure for dimer coverings on $\ol{G}$ obtained as in Lemma \ref{l422}. Assume the edge weights satisfy
(\ref{wtw}).
Then the measure $\mu$ is tail-trivial.
\end{proposition}

\begin{proof}Since $T$ is an infinite transient tree, there exists a finite energy unit flow $\theta$ from the root of $T$ to $\infty$. Let $\theta|_{T_{j,b}}$ be the restriction of the flow to the subtree $T_{j,b}$.
Then $\left[\theta|_{T_{j,b}}\right]\cup \left[-\theta|_{T_{j,b_{0,j}}}\right]$ gives a unit flow on $G_j$ from $b$ to $b_{0,j}$. As $j\rightarrow\infty$ this gives a finite energy unit flow from $b$ to $b_0$; where $b_0=\lim_{j\rightarrow\infty}b_{0,j}$.
Then
\begin{align*}
P_{\bigstar}\lim_{j\rightarrow\infty}\left[\theta|_{T_{j,b}}\right]\cup \left[-\theta|_{T_{j,b_{0,j}}}\right]
\end{align*}
gives a finite energy unit current flow from $b$ to $b_0$. Following the same arguments as in the proof of Proposition \ref{lle614}, we infer that the measure $\mu$ is tail-trivial.
\end{proof}

\section{Double Dimer Contours}\label{sect:ddc}

In this section, we prove that for the infinite-volume Gibbs measure given in Theorem \ref{le67} for uniformly weighted perfect matchings on $\ol{G}$, if $G$ satisfies certian isoperimetric inequalities, then the variance of the height difference of two i.i.d.~perfect matchings is always finite, given that the the height difference on the boundary is 0. This is in contrast with the 2D Euclidean case, where variance of the difference at two points of height differences of two i.i.d. uniformly weighted perfect matchings grows like $\log n$, where $n$ is the distance between two points. See Theorems \ref{le96} and \ref{le85} for precise statements.
The proof depends on the fact that the Dirichlet Green's function on such graphs has certain decay rate with respect to the distance of two points due to isoperimetric inequalities.

For any two dimer configurations $M_1$, $M_2$ on the same graph, the symmetric difference $M_1\triangle M_2$ is a disjoint union of cycles and doubly infinite self-avoiding paths.

\begin{definition}\label{df81}Let $\ol{G}$ be a bipartite, planar graph.  Let $M_1$, $M_2$ be two dimer configurations on $\ol{G}$. 
\begin{enumerate}
\item
The height function $h_{M_1,M_2}$ of double dimer configurations is a function from faces of $\ol{G}$ to the set of all integers, such that for any pair of faces $f_1,f_2$ sharing an edge $e$,
\begin{enumerate}
\item If $e\notin M_1\triangle M_2$, $h_{M_1,M_2}(f_1)=h_{M_1,M_2}(f_2)$;
\item If $e\in M_1\triangle M_2$, 
\begin{enumerate}
\item If $e\in M_1\setminus M_2$, and moving from $f_1$ to $f_2$, the white vertex is on the left, then $h_{M_1,M_2}(f_1)-h_{M_1,M_2}(f_2)=1$;
\item If $e\in M_2\setminus M_1$, and moving from $f_1$ to $f_2$, the white vertex is on the left, then $h_{M_1,M_2}(f_1)-h_{M_1,M_2}(f_2)=-1$.
\end{enumerate}
\end{enumerate}
\item  We define a cluster in $M_1\triangle M_2$ to be a maximal connected subgraph of $\ol{G}^+$ in which every vertex has the same height. A finite (resp.\ infinite) cluster is a cluster consisting of finitely many (resp.\ infinitely many) vertices.

Let $k\in \ZZ$. A level-$k$ cluster in $M_1\triangle M_2$ is a cluster in which every vertex of $\ol{G}^+$ in the cluster has height $k$. 
\end{enumerate}
\end{definition}

One can see that the height function defined in Definition \ref{df81} satisfies
\begin{align*}
h_{M_1,M_2}=h_{M_2}-h_{M_1}
\end{align*}
where $h_M$ is defined as in Definition \ref{df71}.

Recall the following definition:

\begin{definition}We say a graph $G$ satisfies an isoperimetric inequality of dimension $d$ if there exists some constant $c>0$ such that 
\begin{align}
|\partial K|\geq c|K|^{\frac{d-1}{d}}\ \forall\ \mathrm{finite}\ K\subset V.\label{isi}
\end{align}
where $\partial K:=\{(x,y)\in E: x\in K,y\notin K\}$ is the edge boundary of $K$. The isoperimetric dimension of $G$ which we denote by $\mathrm{Dim}(G)$, is defined as the supreme over all $d$ such that such that (\ref{isi}) holds.
\end{definition}

\subsection{General Boundary Conditions}

\begin{theorem}\label{le85} Let $G,G^+,\ol{G}$ be given as in Proposition \ref{lle83}. Let $\mu$ be the probability measures for dimer coverings on $\ol{G}$ obtained as in Theorem \ref{le67}, or Proposition \ref{lle614}, or Proposition \ref{p620}.  Let $M_1$, $M_2$ be two independent dimer configurations on $\ol{G}$ with distribution $\mu$.  Let $\mathcal{N}$ be the number of doubly-infinite self-avoiding paths in $M_1\triangle M_2$.  
Then
\begin{align}
\mu\times \mu(\mathcal{N}=0)=1.\label{pn01}
\end{align}
\end{theorem}

\begin{proof}For $k\in\ZZ$, let $\mathcal{H}_k$ be the event that there exists an infinite level-$k$ cluster in $M_1\triangle M_2$. Since $\mathcal{H}_k$ is measurable with respct to the tail-$\sigma$ field and $\mu$ is tail-trivial by Lemma \ref{lle83},
\begin{align*}
\mu\times \mu(\mathcal{H}_k)\in\{0,1\}.
\end{align*}

Let $m\geq 1$. Assume 
\begin{align}
\mu\times \mu(\mathcal{H}_m)=1. \label{pfm1}
\end{align}
Then interchanging the configuration $M_1$ and $M_2$ we obtain $\mu\times \mu(\mathcal{H}_{-m})=1$. We infer that $\mu\times \mu(\mathcal{H}_m\cap \mathcal{H}_{-m})=1$. Since when $m\geq 1$, $\mathcal{H}_m$ and $\mathcal{H}_{-m}$ are distinct, each infinite level-$m$ cluster has a doubly infinite self-avoiding path in $M_1\triangle M_2$ as its boundary. For any $r\geq m$, if we interchange the configurations of $M_1$ and $M_2$ along all the doubly-infinite self-avoiding paths which are boundaries of some infinite level-$r$ cluster for all $r\geq m$, we obtain a configuration with no infinite level-$m$ cluster. Then $\mu\times \mu(\mathcal{H}_m)=0$; which contradicts (\ref{pfm1}). The contradiction implies
\begin{align}
\mu\times \mu(\mathcal{H}_m)=0.\qquad\forall m\neq 0\label{ss}
\end{align}

We infer from (\ref{ss}) that all the infinite clusters have level 0 a.s. However if there exists a doubly infinite self-avoiding path in $M_1\triangle M_2$, on the two sides of the path there are two infinite clusters of different levels. The contradiction implies (\ref{pn01}).
\end{proof}

\begin{lemma}\label{l83}Let $G=(V,E)$ be a 3-connected, transient, simple proper planar graph with bounded vertex degree and locally finite dual. Suppose Assumptions \ref{ap24}, \ref{ap66} holds. Suppose Assumption \ref{ap311} holds with $\sA$ replaced by $\sA_0$.
Assume the edge weights satisfy (\ref{wtw}). 

Let $\{G_j\}_{j=1}^{\infty}$ be given as in Assumption \ref{ap311}. 
\begin{enumerate}
\item Assume 
\begin{align}
\Phi(u;G_j)>cu^{-\frac{1}{d}}\label{ipj}
\end{align}
for some $d>2$, where $\Phi(u; G_j)$ is defined as in (\ref{dpu}) for the graph $G_j$, and $c$ is a constant independent of $j$.

Then for any $b\in V$, and any white vertex $w$ of $\ol{G}$, we have
\begin{align*}
\left|\lim_{j\rightarrow\infty}\lim_{s\rightarrow\infty}\ol{D}_{\tilde{G}_{j,s}}^{-1}(b,w)\right|\leq C_1[d_{G}(u,v)]^{-\frac{d-2}{2}}
\end{align*}
where $C_1=C_1(c,d)$ is a positive constant.
\item Assume (\ref{ipj}) holds for some $d>4$ with a constant $c>0$ independent of $j$. Then $\left|\lim_{j\rightarrow\infty}\lim_{s\rightarrow\infty}\ol{D}_{\tilde{G}_{j,s}}^{-1}(b,w)\right|$ is uniformly bounded for all $b\in V^+$ and white vertices $w$ of $\ol{G}$.
\end{enumerate}
\end{lemma}

\begin{proof}Let $G_{G_j}(\cdot,\cdot)$ be the discrete Green's function on the infinite graph $G$. If $b\in V$, we see from Theorem \ref{le67}(1) that 
\begin{align*}
\left|\lim_{j\rightarrow\infty}\lim_{s\rightarrow\infty}\ol{D}_{\tilde{G}_{j,s}}^{-1}(b,w)\right|=\left|\lim_{j\rightarrow\infty}\frac{G_{G_j}(b,b_3)}{\Delta(b_3,b_3)}-\frac{G_{G_j}(b,b_3)}{\Delta(b_3,b_3)}\right|
\end{align*}
Then part (1) the lemma follows from Lemma \ref{le426}.

Now we prove part (2) of the lemma. Let
 $b\in V^+$. Let $u,v\in V^+$ be adjacent vertices in $G^+$ such that the white vertex $w$ corresponds to the edge $(u,v)\in E^+$. Let  $l_{v,b}$ be a shortest path in $G^+$ joining $v$ and $b$. More precisely, let
\begin{align*}
l_{v,b}:=v_0(=v),v_1,\ldots,v_n(=b)
\end{align*}
Then we have
\begin{align}
\label{ps3}\ol{D}_{\tilde{G}_{j,s}}^{-1}(b,w)
=\sum_{k=1}^n\left[\ol{D}_{\tilde{G}_{j,s}}^{-1}(v_k,w)-\ol{D}_{\tilde{G}_{j,s}}^{-1}(u_{k-1},w)\right]+\ol{D}_{\tilde{G}_{j,s}}^{-1}(v,w)
\end{align}
For $1\leq k\leq n$, let $x_{k},y_{k}$ be the dual edge of $(v_{k-1},v_k)$. Assume $x_{k},v_{k-1},y_{k},v_k$ are adjacent to a white vertex of $\ol{G}$ in clockwise order. When edge weights satisfy (\ref{wtw}) we obtain
\begin{align}
[\ol{D}_{\tilde{G}_{j,s}}^{-1}(v_k,w)-\ol{D}_{\tilde{G}_{j,s}}^{-1}(v_{k-1},w)]=-\mathbf{i}[\ol{D}_{\tilde{G}_{j,s}}^{-1}(y_{k},w)-\ol{D}_{\tilde{G}_{j,s}}^{-1}(x_{k},w)]\label{pss1}
\end{align}
Moreover,
\begin{align}
|\ol{D}_{\tilde{G}_{j,s}}^{-1}(v,w)|
=\mathbb{P}((v,w)\in M(\tilde{G}_{j,s}))\in [0,1]\label{ps2}
\end{align}
where $\mathbb{P}((v,w)\in M(\tilde{G}_{j,s}))$ is the probability that the edge $(v,w)$ appear in a perfect matching of $\tilde{G}_{j,s}$.
Then the uniform boundedness of $\ol{D}_{\tilde{G}_{j,s}}(b,w)$ when $b\in V$ follows from (\ref{ps3}),  (\ref{pss1}), (\ref{ps2}) and part (1) of the lemma.
\end{proof}

\begin{theorem}\label{le96}Let $G=(V,E)$ be a 3-connected, transient, simple proper planar graph with bounded vertex degree and locally finite dual. Suppose Assumption \ref{ap24} holds.

Let $\mathbb{P}$ be the probability measures for dimer coverings on $\ol{G}$ obtained as in Theorem \ref{le67}. Assume the edge weights satisfy (\ref{wtw}).

Suppose
\begin{itemize}
\item  Assumption \ref{ap311} holds with $\sA$ replaced by $\sA_0$. Let $\{G_j\}_{j=1}^{\infty}$ be given as in Assumption \ref{ap311}. Assume 
\begin{align*}
\Phi(u;G_j)>c_1u^{-\frac{1}{d}}
\end{align*}
for some $d>8$, and $c_1$ is a constant independent of $j$.
\end{itemize}

Let $M_1$, $M_2$ be two independent dimer configurations, both have distribution $\mathbb{P}$. Let $f_0$ be an arbitrary a fixed face of $\ol{G}$.

Then all the following conclusions hold.
\begin{enumerate}
\item For any face $f$, $\mathrm{Var}[h_{M_1,M_2}(f)]^2<\infty$.
\item $\PP\times\PP$ a.s., each face of $\ol{G}$ is enclosed by finitely many self-avoiding cycles in $M_1\triangle M_2$.
 \end{enumerate}
\end{theorem}

\begin{proof}We first prove Part (1) of the lemma. Let $\tilde{G}_{j,s}$ be a  finite subgraph of $G$ obtained as in Assumption \ref{ap66}. Let $M_1$ and $M_2$ be two i.i.d.~dimer configurations on $\tilde{G}_{j,s}$. Without loss of generality, assume that for any face $g$  of $\ol{G}$ not in $\tilde{G}_{j,s}$, $h_{M_1,M_2}(g)=0$.

  Let $g_1$ be a face of $\ol{G}$ not in $\tilde{G}_{j,s}$ but are incident to an edge in $\tilde{G}_{j,s}$. Assume $g_0$ is a face of $\ol{G}$ not in the superposition of $G_{j,s}$ and $\hat{G}_{j,s}^{+}$ but is incident to an edge in $\hat{C}_{1,j,s}$.
Let $l_{fg_0}$ (resp.\ $l_{fg_1}$) be a path in $\ol{G}^+$ joining $f$ and $g_0$ (resp. $f$ and $g_1$)  of $\ol{G}$, such that all the faces along $l_{fg_0}$ (resp.\ $l_{fg_1}$) between $f,g_0$ (resp. $f$ and $g_1$) are faces of $\tilde{G}_{j,s}$. Assume
\begin{itemize}
\item $l_{fg_0}\cup l_{fg_1}$ is a geodesic in $\ol{G}^+$ joining $g_0$ and $g_1$; and
\item 
\begin{align*}
l_{f,g_0}&=s_0=(f),s_1,\ldots,s_m(=g_0);\\
l_{f,g_1}&=q_0=(f),q_1,\ldots,q_k(=g_1).
\end{align*}
 Let
\begin{align*}
\alpha_i&=(s_{i-1}s_{i})^+;\qquad \forall 1\leq i\leq m\\
\beta_t&=(q_{t-1},q_{t})^+;\qquad \forall 1\leq t\leq k\\
\end{align*}
such that $\{\alpha_i\}_{i=1}^{m-1}$, $\{\beta_t\}_{t=1}^k$ are edges of $\tilde{G}_{j,s}$, and $\alpha_m$ is an edge incident to a vertex along $\hat{C}_{1,j,s}$. 
\item Let $L_0$ (resp.\ $L_1$) be the path in $G^+$ (resp.\ $G$) consisting of all the vertices of $G^+$ (resp.\ $G$) in a face of $\ol{G}$ corresponding to a vertex of $\ol{G}^+$ along $l_{f,g_0}$ (resp.\ $l_{f,g_1}$). Assume
\begin{align*}
L_0:&=b_{0,0}(\in f),b_{0,1},\ldots,b_{0,m_0}(\in g_0)\\
L_1:&=b_{1,0}(\in f),b_{1,1},\ldots,b_{1,k_1}(\in g_1)\\
\end{align*}
where $b_{0,m_0}$ is a vertex along $\hat{C}_{1,j,s}$.
\end{itemize}

Let
\begin{align*}
E_j:=\{\alpha_1,\ldots,\alpha_m\};\qquad F_j:=\{\beta_1,\ldots,\beta_k\};
\end{align*}
We shall write 
\begin{itemize}
\item $E_j:=E_{j,1}\cup E_{j,2}$, and $E_{j,1}\cap E_{j,2}=\emptyset$, where for $1\leq i\leq m$, $\alpha_i\in E_{j,1}$ if and only if moving along $l_{f,g_1}$ from $f$ to $g_1$, when crossing $\alpha_i$ the white vertex of $\alpha_i$ is on the left.
\item $F_j:=F_{j,1}\cup F_{j,2}$, and $F_{j,1}\cap F_{j,2}=\emptyset$, where for $1\leq t\leq k$, $\beta_t\in F_{j,1}$ if and only if moving along $l_{f,g_2}$ from $f$ to $g_2$, when crossing $\beta_t$ the white vertex of $\beta_t$ is on the left.
\end{itemize}

Then
\begin{small}
\begin{align*}
&\mathrm{Var} [h_{M_1,M_2}(f)]^2=\EE [h_{M_1,M_2}(f)-h_{M_1,M_2}(g_1)]
[h_{M_1,M_2}(f)-h_{M_1,M_2}(g_2)]\\
&=\mathbb{E}\left[\sum_{\alpha\in E_{j,1}}[\mathbf{1}_{\alpha\in M_1}-\mathbf{1}_{\alpha\in M_2}]
-\sum_{\alpha\in E_{j,2}}[\mathbf{1}_{\alpha\in M_1}-\mathbf{1}_{\alpha\in M_2}]
\right]\left[\sum_{\beta\in F_{j,1}}[\mathbf{1}_{\beta\in M_1}-\mathbf{1}_{\beta\in M_2}]
-\sum_{\beta\in F_{j,2}}[\mathbf{1}_{\beta\in M_1}-\mathbf{1}_{\beta\in M_2}]
\right]\\
&=2\sum_{\alpha\in E_{j,1}}\sum_{\beta\in F_{j,1}}\mathbb{E}(\mathbf{1}_{\alpha\in M_1}-\mathbb{E}\mathbf{1}_{\alpha\in M_1})(\mathbf{1}_{\beta\in M_1}-\mathbb{E}\mathbf{1}_{\beta\in M_1})-2\sum_{\alpha\in E_{j,1}}\sum_{\beta\in F_{j,2}}\mathbb{E}(\mathbf{1}_{\alpha\in M_1}-\mathbb{E}\mathbf{1}_{\alpha\in M_1})(\mathbf{1}_{\beta\in M_1}-\mathbb{E}\mathbf{1}_{\beta\in M_1})\\
&+2\sum_{\alpha\in E_{j,2}}\sum_{\beta\in F_{j,2}}\mathbb{E}(\mathbf{1}_{\alpha\in M_1}-\mathbb{E}\mathbf{1}_{\alpha\in M_1})(\mathbf{1}_{\beta\in M_1}-\mathbb{E}\mathbf{1}_{\beta\in M_1})-2\sum_{\alpha\in E_{j,2}}\sum_{\beta\in F_{j,1}}\mathbb{E}(\mathbf{1}_{\alpha\in M_1}-\mathbb{E}\mathbf{1}_{\alpha\in M_1})
(\mathbf{1}_{\beta\in M_1}-\mathbb{E}\mathbf{1}_{\beta\in M_1})
\end{align*}
\end{small}
where we used the fact that $M_1$ and $M_2$ are i.i.d. Moreover, note that if 
 $\alpha$ and $\beta$ are two distinct edges sharing a vertex, 
\begin{align*}
\mathbb{E}(\mathbf{1}_{\alpha\in M_1}-\mathbb{E}\mathbf{1}_{\alpha\in M_1})(\mathbf{1}_{\beta\in M_1}-\mathbb{E}\mathbf{1}_{\beta\in M_1})=
-\mathbb{P}(\beta\in M_1)
\mathbb{P}(\alpha\in M_1).
\end{align*}
If $\alpha$ and $\beta$ are two distinct edges sharing no vertices; assume $\alpha=(w_1,b_1)$ and $\beta=(w_2,b_2)$; then by Corollary \ref{c18}(2)  we have
\begin{align*}
\left|\mathbb{E}(\mathbf{1}_{\alpha\in M_1}-\mathbb{E}\mathbf{1}_{\alpha\in M_1})(\mathbf{1}_{\beta\in M_1}-\mathbb{E}\mathbf{1}_{\beta\in M_1})\right|=\sqrt{\frac{\nu(\alpha)\nu(\beta)}{\nu(\alpha^+)\nu(\beta^+)}}\lim_{j\rightarrow\infty}\left|\ol{D}_{\tilde{G}_{j,s}}^{-1}(w_1,b_2)\ol{D}_{\tilde{G}_{j,s}}^{-1}(w_2,b_1)\right|
\end{align*}

By Lemma \ref{l83},
$|\lim_{j\rightarrow\infty}\lim_{s\rightarrow\infty}\ol{D}^{-1}_{\tilde{G}_{j,s}}(b_2,w_1)|$ is uniformly bounded for all $b_2,w_1$.
Consider $\left|\ol{D}_{\tilde{G}_{j,s}}^{-1}(w_2,b_1)\right|$. The following cases might occur:
\begin{enumerate}
\item $b_1$ is a vertex of $G$. By Lemma \ref{l422}(1) and the  decay rate of the Dirichlet Green's function $G_{G_j}$ given by Lemma \ref{le426}, we have
\begin{small}
\begin{align*}
|\ol{D}_{\tilde{G}_{j,s}}^{-1}(b_1,w)|\leq \left|\frac{G_{G_j}(b_1,v_1)}{\Delta(b_1,v_1)}\right|+\left|\frac{G_{G_j}(b_1,v_2)}{\Delta(b_1,v_2)}\right|
\leq C''\left[[d_{G}(b_1,v_1)]^{-\frac{d-2}{2}}+[d_{G}(b_1,v_2)]^{-\frac{d-2}{2}}\right]
\end{align*}
\end{small}
where $(v_1,v_2)$ is the edge of $G$ corresponding to $w$. 

\item $b_1$ is a vertex of $G^+$. Assume $b_1=b_{0,i}$. In this case we have
\begin{small}
\begin{align*}
\lim_{j\rightarrow\infty}\ol{D}_{\tilde{G}_{j,s}}^{-1}(b_1,w)
=\sum_{c=i+1}^{m_0}\left[\ol{D}_{\tilde{G}_{j,s}}^{-1}(b_{0,c-1},w)-\ol{D}_{\tilde{G}_{j,s}}^{-1}(b_{0,c},w)\right]
+\ol{D}_{\tilde{G}_{j,s}}^{-1}(b_0,w)
\end{align*}
\end{small}
Recall that $\ol{D}_{\tilde{G}_{j,s}}^{-1}(b_0,w)=0$.
For $1\leq k\leq a_0$, let $x_{k},y_{k}$ be the dual edge of $(b_{3,k-1},b_{3,k})$.
For $m_0+1\leq k\leq m_0+i$, let $x_{k},y_{k}$ be the dual edge of $(b_{1,k-m_0},b_{1,k-m_0-1})$.

Assume for $1\leq k\leq m_0$, $x_{k},b_{0,k-1},y_{k},b_{0,k}$ are adjacent to a white vertex of $\ol{G}$ in clockwise order. Then we obtain
\begin{align}
\nu(b_{0,k}b_{0,k-1})[\ol{D}_{\tilde{G}_{j,s}}^{-1}(b_{0,k},w)-\ol{D}_{\tilde{G}_{j,s}}^{-1}(b_{0,k-1},w)]=-\mathbf{i}\nu(x_{k}y_{k})[\ol{D}_{\tilde{G}_{j,s}}^{-1}(y_{k},w)-\ol{D}_{\tilde{G}_{j,s}}^{-1}(x_{k},w)]\label{ps1}
\end{align}
\end{enumerate}

 Given that $l_{f,g_0}\cup l_{f,g_1}$ is a geodesic joining $g_0$ and $g_1$, and $d>8$, we infer that 
$\EE [h_{M_1,M_2}(f)]^2$ is uniformly bounded for all $j,s$; then Part (1) follows.

Now we prove Part (2) by contradiction. Assume that with strictly positive probability, a face $f$ is enclosed by infinitely many self-avoiding cycles in $M_1\triangle M_2$. Then there exists $c_0>0$, such that for any $N$,
\begin{align*}
\lim_{j\rightarrow\infty}\mathbb{P}(f\ \mathrm{is\ enclosed\ by\ at\ least}\ N\ \mathrm{self-avoiding\ cycles\ in}\ M_1\triangle M_2\ \mathrm{of}\ \mathcal{G}_j)>c_0.
\end{align*}

Note that crossing each self-avoiding cycle enclosing $f_0$ in $M_1\triangle M_2$, with probability $\frac{1}{2}$, the $h_{M_1,M_2}$ increases by 1, and with probability $\frac{1}{2}$, $h_{M_1,M_2}$ decreases by 1. Then $f$ is enclosed by at least $N$ self-avoiding cycles in $M_1\triangle M_2$, and the boundary heights are 0, $h_{M_1,M_2}(f_n)$ is the sum of at least $N$ i.i.d.~random variables $\xi_{1},\ldots,\xi_{N},\ldots$, each of which takes value $1$ with probability $\frac{1}{2}$ and takes value $0$ with probability $\frac{1}{2}$. 

Let $\mathcal{S}_{N,j}$ be the event that $f$ is enclosed\ by\ at\ least $N$ self-avoiding cycles in $M_1\triangle M_2$ of $\mathcal{G}_j$, then we have
\begin{align*}
&\mathbb{E}[h_{M_1,M_2}(f)]^2\geq
\mathbb{E}[(\xi_{1}+\ldots+\xi_{N})^2|\mathcal{S}_{N,j}]\mathbb{P}(\mathcal{S}_{N,j})
= \sum_{j=1}^N \mathbb{E}(\xi_j)^2\mathbb{P}(\mathcal{S}_{N,j})=c_0N
\end{align*}
Since $N$ is arbitrary we obtain
\begin{align*}
\lim_{j\rightarrow\infty}\mathbb{E}[h_{M_1,M_2}(f)]^2=\infty;
\end{align*}
which contradicts Part (1). The contradiction implies Part (2).
\end{proof}

\begin{example}We now give an example in which (\ref{ipj}) holds with a universal constant $c$ for all the $G_j$'s. Consider for example, a nonamenable, vertex-transitive regular tiling of the hyperbolic plane with vertex degree $D$ and face degree $D^+$. By \cite{hjl02},
\begin{align*}
\inf\left\{\frac{|\partial_E K|}{|K|}:K\subset V\ \mathrm{finite\ and\ nonempty}\right\}=(D-2)\sqrt{1-\frac{4}{(D-2)(D^+-2)}};=i_E(G)
\end{align*}
Let $v_0$ be a fixed vertex of $G$. Let $G_{j,s}$ be the finite subgraph of $G$ constructed as follows
\begin{itemize}
\item Let $F_{j,s}$ be the union of all the faces containing at least one vertex in $G_{j,s}$. Let $\partial F_{j,s}$ be all the edges in $F_{j,s}$ incident to at least one face in $F_{j,s}$ and one face not in $F_{j,s}$ as well as their endvertices. Assume $\partial F_{j,s}$ is a simple closed curve.

the boundary $\partial^V G_{j,s}$ of $G_{j,s}$, defined to be the union of vertices in $G_{j,s}$ incident to at least a vertex not in $G_{j,s}$, as well as edges in $G_*$ (the matching graph) joining two such vertices sharing a face.  There are 4 vertices  $z_1,z_2,z_3,z_4$ in $\partial F_{j,s}\cap \partial^{V}G_{j,s}$
in clockwise order
\begin{itemize}
\item $\partial F_{j,s}$ is divided by $u_1,u_2,u_3,u_4$ into 4 segments $L_{u_1u_2}$, $L_{u_2u_3}$, $L_{u_3u_4}$, and $L_{u_4,u_1}$;
\item $d_G(u_1,v_0)=d_{G}(u_2,v_0)=j$;
\item $L_{u_1u_2}$ is a geodesic joining $u_1$ and $u_2$;
\item all the vertices in $L_{u_3u_4}\cap \partial^VG_{j,s}$ have graph distance exactly $s$ from $v_0$;
\item $L_{u_3u_3}$ is a geodesic in $G$ joining $u_1$ and $u_2$;
\item $L_{u_4u_1}$ is a geodesic in $G$ joining $u_4$ and $u_1$;
\item Let $Q_j$ be the union of all the faces containing at least one vertex with distance at most $j$ to $v_0$. Let $\partial Q_j$ be the union of all the edges incident to at least one face in $Q_j$ and one face not in $Q_j$. Let $R_{z_1z_2}$ be the portion of $\partial Q_j$ from $z_1$ to $z_2$ in clockwise order.

Let $N_j$ be the number of vertices in $R_{z_1z_2}$, then
\begin{align*}
1<\limsup_{j\rightarrow\infty}N_j^{\frac{1}{j}}< 1+i_E(G);\qquad\forall j
\end{align*}
where $\alpha>0$ is a constant independent of $j$
\end{itemize}
\end{itemize}
As in Assumption \ref{ap66}, let $\hat{G}_{j,s}^+$ be the subgraph of $G^+$ that contains $G_{j,s}$ as an interior dual graph. Let $v_{1,j,s}=z_1$ and $v_{2,j,s}=z_2$, $C_{1,j,s}=L_{z_1z_2}$. 
Let $G_j$ in Assumption \ref{ap311} be obtained from $G_{j,s}$ by letting $s\rightarrow\infty$. Then by Theorem \ref{le67}, the limit measure obtained by such an approximation with $j,s\rightarrow\infty$ is different from the one obtained with Temperley boundary conditions as in Theorem \ref{l422}. However, one can check that for any $d>8$, (\ref{ipj}) holds with a universal constant $c$ independent of $j$. By Theorem \ref{le96}, the variance of double dimer height function is finite for any face.
\end{example}

\bigskip
\noindent\textbf{Acknowledgements.}\ ZL acknowledges support from National Science Foundation DMS 1608896 and Simons Foundation grant 638143. ZL thanks Gourab Ray for discussions.

\bibliography{dimerhyp}
\bibliographystyle{plain}

\end{document}